\documentclass[12pt,a4paper]{article} %Remove draft from submission version
\usepackage[left=1.2in,right=1.2in,top=.8in,bottom=.8in]{geometry}
\usepackage{amsmath,amsfonts,amsthm,amssymb,graphicx,xcolor, authblk}  
\newtheorem{theorem}{Theorem}
\newtheorem{lemma}{Lemma}
\newtheorem{corollary}{Corollary}

\newtheorem{proposition}{Proposition}

\theoremstyle{remark}

\newtheorem{remark}{Remark}
\usepackage[utf8]{inputenc}
\usepackage{color}
\usepackage{hyperref}

\usepackage{mathtools}
\mathtoolsset{showonlyrefs=true} % show or hide unreferenced
\newcommand{\dualH}{\widetilde{H}}
\newcommand{\tildeR}{\widetilde{\mathcal{R}}}
\newcommand{\e}{\mathrm{e}}

\renewcommand{\d}{\mathrm{d}}

\newcommand{\R}{\mathbb{R}}

\newcommand{\eps}{\varepsilon}

\newcommand{\N}{\mathbb{N}}

\newcommand{\s}{\hspace{0.5pt}}
\newcommand{\supp}{\mathop{\rm supp}}

\newcommand{\M}{\mathcal{M}}
\renewcommand{\k}{{\bar{k}}}
\renewcommand{\hat}{\widehat}

\newcommand{\PSO}{\mathop{\rm PSO}}
\newcommand{\W}{W} % The set Q\Subset I^+ \cap I^-

\renewcommand{\tilde}{\widetilde}

\renewcommand{\;}{\, ; }
\newcommand{\ccdot}{\,\cdot\,}
\newcommand{\p}{\partial}

\newcommand{\norm}[1]{\lVert #1 \rVert}
\newcommand{\abs}[1]{\lvert #1 \rvert}

\title{Stability estimates for inverse problems for semi-linear wave equations on Lorentzian manifolds}

\date{}

\author[1]{Matti Lassas}
\author[2]{Tony Liimatainen}
\author[3]{Leyter Potenciano-Machado}
\author[4]{Teemu Tyni}
\affil[1, 2, 4]{Department of Mathematics and Statistics, University of Helsinki, Helsinki, Finland}
\affil[2,3]{Department of Mathematics and Statistics, University of Jyv\"askyl\"a, Jyv\"askyl\"a, Finland}

\begin{document}

\maketitle
\begin{abstract}
%TODO: write better...
This paper concerns an inverse boundary value problem of recovering a zeroth order time-dependent term of a semi-linear wave equation on a globally hyperbolic Lorentzian manifold.
We show that an unknown potential $q$ in the non-linear wave equation $\square_g u +q u^m=0$, $m\geq 4$, can be recovered in a H\"older stable way from the Dirichlet-to-Neumann map. Our proof is based on the higher order linearization method and the use of Gaussian beams. Unlike some related works, we do not assume that the boundary is convex or that pairs of lightlike geodesics can intersect only once. For this, we introduce some general constructions in Lorentzian geometry. We expect these constructions to be applicable to studies of related problems as well.  % using exponential concentration of Gaussian beams.
% %
% The proof is based on the higher-order linearization method.
% %
% Moreover, we do not assume that the manifold is convex or that the geodesics can intersect only once in the manifold.
% %
% The difficulties caused by possible multi-intersections of geodesics are resolved by using solutions of the wave equation to separate the intersection points.
% %
% Several geometric lemmata regarding Lorentzian manifolds are also given, which may be interesting in their own right.
% %
\end{abstract}

% \begin{abstract}
% %TODO: write better...
% This paper concerns the inverse problem of recovering a zeroth order term of a semi-linear wave equation on a globally hyperbolic Lorentzian manifold.
% %
% We show that an unknown potential $q$ in the non-linear wave equation $\square_g u +q u^m=0$ can be recovered in a H\"older stable way from the Dirichlet-to-Neumann map using exponential concentration of Gaussian beams.
% %
% The proof is based on the higher-order linearization method.
% %
% Moreover, we do not assume that the manifold is convex or that the geodesics can intersect only once in the manifold.
% %
% The difficulties caused by possible multi-intersections of geodesics are resolved by using solutions of the wave equation to separate the intersection points.
% %
% Several geometric lemmata regarding Lorentzian manifolds are also given, which may be interesting in their own right.
% %
% \end{abstract}

\section{Introduction}
We consider the stability and uniqueness of an inverse problem for the non-linear wave equation on an $n+1$-dimensional, $n\geq 2$, globally hyperbolic Lorentzian manifold.
As is well known, any globally hyperbolic Lorentzian manifold $N$ is isometric to a product manifold $\R\times M$ equipped with the product metric
\begin{equation}\label{eq:g}
g = -\beta(t,x)dt^2 + h(t,x).
\end{equation}
Here $\beta>0$ is a smooth function and $h(t,\ccdot)$, $t\in \R$, is a smooth one-parameter family of Riemannian metrics on an $n$-dimensional manifold $M$, see e.g. \cite{BS2005}.
%
%We will abuse the notation slightly and use the letter $g$ to denote the metric in both settings.
%
Let $\Omega \subset M$ be a smooth submanifold of dimension $n$ with smooth boundary and let us denote the lateral boundary of $[0,T]\times \Omega\subset N$ by 
\[
\Sigma := [0,T]\times \p\Omega.
\]
%We shall also require a slightly smaller time-interval, for which we use the notation $\Sigma_1 = [\eps,T-\eps]\times \Omega$, $0<\eps\ll 1$.
%
% 
In local  coordinates $(x^a)$ the D'Alembertian wave operator $\square_g$ of $g$ has the form
\begin{equation}\label{eq:metric_local}
\square_g u = -\sum_{a,b=0}^n\frac{1}{\sqrt{|\det(g)|}}\frac{\p}{\p x^a}\left(
\sqrt{|\det(g)|}g^{ab} \frac{\partial u}{\partial x^b}
\right).%=-g^{ab}\p_{ab}u \pm \Gamma^a\p_au.
\end{equation}
Here we denote $(g^{-1})_{ab}= (g^{ab})$, $a,b=0,\ldots,n$, as usual. We consider the non-linear wave equation
\begin{equation}\label{eq:Main equation}
\begin{cases}
\square_g u(t,x) + q(t,x)u(t,x)^m=0,&\text{in } [0,T]\times \Omega,\\
u=f,&\text{on } [0,T]\times \p \Omega,\\
u(0,x)=\p_t u(0,x) = 0,&\text{on } \Omega,
\end{cases}
\end{equation}
where we assume that the exponent $m$ is an integer greater or equal than $4$.
The inverse problem we study is the stability of recovery of the potential $q$ from the Dirichlet-to-Neumann (DN) map 
\begin{align*}
\Lambda : H_0^{s+1}(\Sigma)&\to H^s(\Sigma),\\
f&\mapsto \p_\nu u_f\big|_\Sigma,
\end{align*}
where $u_f$ is the unique small solution of \eqref{eq:Main equation} and $\p_\nu$ is the normal derivative on $\Sigma$. Here also $H^s$ refers to a Sobolev space. See Section \ref{sec:preliminaries} for details about Sobolev spaces and  Section~\ref{sec:forward} for details about the well-posedness of the forward problem. We describe our main results in Section \ref{sec:main_results}. The present work is a continuation of the authors' earlier work \cite{LLPT}, which considered the stability of a recovery of the potential $q$ of \eqref{eq:Main equation} in Minkowski space of $\R^{n+1}$.

% 
% The inverse problem we consider is to stably recover $q(t,x)$ from the knowledge of the DN map. We describe our main results in Section \ref{sec:main_results}. The present work is a continuation of the authors' earlier work \cite{LLPT}, which considered the stability of a recovery of the potential $q$ of \eqref{eq:Main equation} in $\R^{n+1}$.

%

%%
%We will see in Lemma~\ref{lemma:nonlinear-solutions} that the equation \eqref{eq:Main equation} has a unique small solution when we are given small enough boundary values $f\in H^{s+1}(\Sigma)$.
%%
%This allows us to define the Dirichlet-to-Neumann map (DN map) as
%\begin{align*}
%\Lambda_q : H^{s+1}(\Sigma) &\to H^{s}(\Sigma),\\
%f &\mapsto \partial_\nu u_f\big|_{\Sigma},
%\end{align*}
%where $u_f$ is the unique small solution to \eqref{eq:Main equation} and $\p_\nu$ is the normal derivative on the boundary $\Sigma$.

%Let us denote the epsilons with $\eps_j$, $j=1,\ldots,m$ and $\vec\eps = (\eps_1,\ldots,\eps_m)$. The notation 
%$$
%D_{\vec\eps=0}^m
%$$
%will be used for the finite difference operator, see below.
%
%
%I propose using $\kappa$ again to control the boundary data, that is, the boundary data $f$ should be bounded by
%$$
%\norm{f}\leq \kappa.
%$$
%In particular, $\norm{\sum_{j=1}^m\eps_jf_j} \leq \kappa$. Since boundary data will depend on the scale of the Gaussian beam $\tau$ we will have something like $\eps \tau^K\leq \kappa$.
%%

Studies of uniqueness and stability of recovery of unknowns parameters in inverse problems are motivated by practical applications.
Let us mention some results on inverse problems for linear wave type equations.
First results to this direction for the linear wave equation with vanishing initial data were obtained by Belishev and Kurylev~\cite{Be87,BeKu92}.
Their approach is called the boundary control method and it combines both the wave propagation and controllability results \cite{KKL01}.
The boundary control method  allows also an effective numerical algorithm \cite{deHoop}.
Recently, there has been several results on determining a Riemannian manifold from partial data boundary measurements for the linear wave equation and related equations such as the ones in \cite{AKKLT,Helin,Isozaki,KKLO,KrKL,KOP,Lassas,LO}.
However, the boundary control method has been applicable only in the cases where the coefficients of the equation are time-independent, or when the lower order terms are real analytic in time variable \cite{Esk}.
In a geometric setting it has been studied if it is possible to recover a Riemannian metric $g$ from the Dirichlet-to-Neumann map of the equation $(\partial_t^2-\Delta_g)u=0$ in a stable way.
Earlier results for recovery of the metric are based on Tataru's unique continuation principle, which yields stability estimates of logarithmic type, see e.g.~\cite{BoKuLa17}.
Later these results have been improved by using different techniques and different assumptions. For example, in \cite{StUh} it was shown that a simple Riemannian metric $g$ can be recovered in a H\"older stable way from the DN map.
For examples of instability of inverse problems for a wide class of equations, see \cite{KRS}.

%The stability of recovery of the potentials $a$ and $b$ in the linear wave equation $\square u + b\p_t u + au=0$ was considered for instance in \cite{IsSun}, where a local H\"older stablity estimate was given for partial data measurements on a subset of the boundary.
%%
%
Concerning the unique recovery of potentials in a linear counterpart of \eqref{eq:Main equation} with lower order terms we mention the works \cite{FIKO,St,SY}. 
These works make use of propagation of singularities along bicharacteristics to determine integrals of the unknown coefficients along light rays.
%
%This is also our approach to answer the questions about stability and uniqueness for the non-linear equation \eqref{eq:Main equation}.
%
In these results, the Dirichlet-to-Neumann or scattering operator needs to be known over all of the lateral boundary $\Sigma$.
Moving on to inverse problems for non-linear wave equations, Kurylev, Lassas and Uhlmann \cite{KLU18} observed that non-linearity  % in the studied equation gives rise to new singularities which 
can be used as a beneficial tool in inverse problems for nonlinear wave equations.
By exploiting the non-linearity, some still unsolved inverse problems for linear hyperbolic equations have recently been solved for their non-linear counterparts.
The first results in \cite{KLU18}, for the scalar wave equation with a quadratic non-linearity, already showed that local measurements of solutions of the non-linear wave equation determine the global topology, differentiable structure and the conformal class of the metric $g$ on a globally hyperbolic $3+1$-dimensional Lorentzian manifold.
The results of \cite{KLU18} use the so-called \emph{higher order linearization method}, which has made inverse problems for non-linear equations more approachable. 
The method has given rise to many new results on inverse problems for non-linear equations. We will explain the 
method later in this introduction.

The authors of~\cite{LUW18} studied inverse problems for general semi-linear wave equations on Lorentzian manifolds, and in \cite{LUW17} they studied analogous problem for the Einstein-Maxwell equations. The papers \cite{HUZ20,HUZ21} are closely related to this work. They use higher-order linearization method to study uniqueness for the inverse problem of \eqref{eq:Main equation}. However, these works have additional assumptions that the domain $\Omega$ of the time cylinder $[0,T]\times \Omega$ is convex and that lightlike geodesics can only intersect once. Our results will in particular improve results in \cite{HUZ20}.

The research of inverse problems for non-linear equations is expanding fast. By using the higher-order linearization, inverse problems for nonlinear models have been studied for example in~\cite{BKLT20, CaNaVa19, Chen2019,CLOP2020,dH2019, dH2020,FeOk20,FO19, KaNa02, KrUh19, KrUh20,KLOU2014, LaUhYa20, LLLS19a, LLLS19b, OSSU, SunUh97, UhWa18,WZ2019}.

\subsection{Main results}\label{sec:main_results}
The present work is a continuation of the work \cite{LLPT} to the setting of globally hyperbolic Lorentzian manifolds. The work \cite{LLPT} considered a stability result for a recovery of the potential $q$ of \eqref{eq:Main equation} in $\R^{n+1}$. We denote by $(N,g)$ a globally hyperbolic manifold. We assume that the dimension of $N$ is $n+1$, where $n\geq 2$. As explained earlier, we view $N$ as the product manifold $\R\times M$ equipped with the product metric \eqref{eq:g} and where $M$ is an $n$-dimensional manifold. We fix a time-interval $[0,T]$.  We assume that $\Omega \subset M$ is an $n$-dimensional submanifold of $M$ and that $\Omega$ has a smooth boundary $\p \Omega$. 
%
%We do not assume that the manifold $(N,g)$ is simple and in particular we allow geodesics to have multiple intersection points.
%
%To uniquely recover the potential $q$, it would be sufficient to recover $q$ pointwise.
%%
%However, to consider the stability of the recovery of the potential $q$ we need to keep track of many quantities and estimates.
%%
%%
%Our plan is to use the exponential concentration of Gaussian beams near geodesics and the nonlinearity of the wave equation to produce artificial point sources (in our case, approximate delta functions).
%

The finite propagation speed of solutions to the wave equation and the causal structure of $(N,g)$ cause natural limitations on the parts of $[0,T]\times \Omega$ where we can obtain information about the potential in the inverse problem.
Let $\W$ be a compact set belonging to both the causal future and past of the lateral boundary $\Sigma=[0,T]\times \p\Omega$:
\begin{equation}\label{eq:recovery_set}
\W\subset I^-(\Sigma) \cap I^+(\Sigma)\cap ([0,T] \times \Omega).
\end{equation}
This is the domain which can be reached by sending waves from $\Sigma$ so that the possible signals generated by a nonlinear interaction of the waves can also be detected on $\Sigma$. We do not assume that $[0,T]\times \p \Omega$ is convex or that lightlike geodesics of $(N,g)$ can only intersect once.

% {\color{gray}
% %
% We define admissible potentials as those functions, that are supported both in the causal past and the causal future of the lateral boundary:
% \begin{definition}[Admissible potentials]\f{Do we need this definition anymore? If not, need to add a bound on norms of $q$ to statements of theorems.}
% Given $s\geq 0$ and $C_0>0$, the class of admissible potentials is defined as the set of all functions $q\in C_c^{s+1}(\R\times M)$ satisfying
% \begin{align*}
% & \Vert q\Vert_{C^{s+1}} \leq C_0,\\
% & \mathrm{supp}(q)\subset \R\times \Omega.
% \end{align*}
% \end{definition}
% %
% }
Below we use the notation $H_0^s$ for the closure of the space of compactly supported smooth functions, with respect to the Sobolev $H^s$ norm.
The main result of this work is the following:
\begin{theorem}[Stability estimate]\label{thm:stability}
Suppose $(N,g)$, $N=\R\times M$, is an $n+1$-dimensional globally hyperbolic Lorentzian manifold.
Let $T>0$ and let $\Omega\subset M$ be a submanifold with smooth non-empty boundary. Let $m\geq 4$ be an integer, $s\in\N$ with $s+1>(n+1)/2$ and $r\in \R$ with $r\leq s$. Assume that $q_1, q_2\in C^{s+1}(\R\times\Omega )$ satisfy $\Vert q_j \Vert_{C^{s+1}} \leq c$, $j=1,2$, for some $c>0$. Let $\Lambda_1,\s\Lambda_2 :H_0^{s+1}(\Sigma)\to H^r(\Sigma)$ be the corresponding Dirichlet-to-Neumann maps of the non-linear wave equation~\eqref{eq:Main equation}.

Let $\eps_0>0$, $L>0$ and $\delta\in (0,L)$ be such that
 \begin{equation}
 \Vert \Lambda_1(f)-\Lambda_2 (f)\Vert_{H^r(\Sigma)} \leq \delta
\end{equation}
%\begin{equation}
% \abs{\langle \psi, \Lambda_1(f)-\Lambda_2 (f)\rangle_{L^2(\widetilde{\Sigma})}} \leq \delta
%\end{equation}
for all $f\in H_0^{s+1}(\Sigma)$ with $\Vert f\Vert_{H^{s+1}(\Sigma)}\leq \eps_0$. 
Then there exists a constant $C>0$, independent of $q_1,q_2$ and $\delta>0$, such that
\begin{equation}\label{eq:esimate_for_potential_difference}
%  \norm{q_1-q_2}_{L^\infty( [0,T]\times\Omega)}\leq C \delta^{\sigma(s,m)},
  \norm{q_1-q_2}_{L^\infty( \W )}\leq C \delta^{\sigma(s,m)},
\end{equation}
where
\[
\sigma(s,m) = \frac{8(m-1)}{2m(m-1)(8s-n+13)+2m-1}.
\]
\end{theorem}

%We do not claim that the H\"older exponent given by Theorem~\ref{thm:stability} is optimal --- most likely it is not.
%One can check that $0<\sigma\leq\frac{24}{487}<0.05$ for any admissible choice of $s,n$ and $m$.
%

A corollary of the theorem is a uniqueness result, which improves the main result of \cite{HUZ20} to the case of possibly non-convex boundary and  where lightlike geodesics can intersect more than once.
\begin{corollary}[Uniqueness]\label{corz:stability_estimate_z}
Adopt the notation and assumptions of Theorem~\ref{thm:stability}.
Then the Dirichlet-to-Neumann map $\Lambda$ uniquely determines the potential $q$ within the set $\W$.
\end{corollary}
We only consider the case $m\geq 4$ in this work, because the other natural cases, where $m$ is either $2$ or $3$, would lead to additional technicalities. Especially, the case where $m$ is $2$ is special and would need somewhat different techniques. In fact, the authors of \cite{HUZ20} used different types of solutions in their uniqueness proof when $m$ is $2$. We consider these two special cases in a future work. 

We explain next how these results are proved and how we are able to consider non-convex boundaries and the case where lightlike geodesics can intersect more than once.

\subsection{Sketch of the proof of Theorem \ref{thm:stability}}
Let us discuss the main ideas behind the proof of Theorem \ref{thm:stability}. We first discuss how to recover $q$ uniquely from the DN map $\Lambda$ associated with equation~\eqref{eq:Main equation}. To avoid technical details, the presentation here is slightly formal. We also only consider here the case $m=4$ for simplicity, while the case $m>4$ is similar. %This is to avoid technical dea%we concentrate on explaining the main ideas behind the proof of Corollary \ref{corz:stability_estimate_z} in the case $m=4$. 
%The general case follows from similar arguments.  

Consider $f_j\in H_0^{s+1}(\Sigma)$, $j=1,2,3,4$, with $\norm{f_j}_{H^{s+1}(\Sigma)}\leq c_0$ for some constant $c_0>0$. Let us denote by $u_{\eps_1f_1+ \cdots + \eps_4f_4}$ the solution to~\eqref{eq:Main equation} with boundary data $\eps_1f_1+ \cdots +\eps_4f_4$, where $\eps_j>0$ are sufficiently small parameters. We abbreviate the notation by writing $\vec{\eps}=0$ when referring to $\eps_1 = \cdots= \eps_4=0$. By taking the mixed derivative $\p_{\eps_1\cdots\eps_4}^4|_{\vec{\eps}=0}$ of the solution $u_{\eps_1f_1+ \cdots + \eps_4f_4}$ to the equation~\eqref{eq:Main equation} with respect to the parameters $\eps_1,\ldots,\eps_4$, %and of the solution $u_{\eps_1f_1+ \ldots + \eps_4f_4}$, 
we see that the function
\[
 w:=\frac{\p}{\p \eps_1} \cdots\frac{\p}{\p \eps_4}\Big|_{\vec{\eps}=0}\s\s u_{\eps_1f_1+ \ldots + \eps_4f_4}
\]
solves the equation
\begin{equation}\label{eq:second_deriv}
 \square_g \s w   = -16\s  q\s  v_{1}\s v_2 \s v_3 \s v_4, \quad\text{in} \,\,  [0,T]\times \Omega
\end{equation}
 with vanishing Cauchy and boundary data. 
 %Here $v_1$ and $v_2$ are solutions to the linearized equation $\square v=0$ with $v_1|_\Sigma=f_1$ and $v_2|_\Sigma=f_2$. 
 Here the functions $v_j$, $j=1,\ldots,4$, satisfy
 \begin{equation} \label{lin:delta_z}
 \begin{cases}
 \square_g\s v_j = 0, &\text{in } [0,T]\times \Omega,\\ 
 v_j = f_j, &\text{on }[0,T] \times\p\Omega,\\
 v_j\big|_{t=0} = \p_t v_j\big|_{t=0} = 0, &\text{in } \Omega.
 \end{cases}
 \end{equation}
This %\emph{higher order linearization method} produced 
way we have produced new linear equations from the non-linear equation~\eqref{eq:Main equation}. % Studying these new equations in inverse problems for non-linear equations is known as the higher order linearization method.
If the DN map $\Lambda$ is known, then the normal derivative of $w$ is also known on $\Sigma$. This is true, because %Taking the mixed derivative of the DN map yields
  \[
   \p_\nu w =\p^4_{\eps_1 \cdots \eps_4}|_{\vec{\eps}=0} \s\s \Lambda(\eps_1f_1+ \cdots + \eps_4f_4).
 \]
% \[
%  \p_\nu w 
% \]
% is also known on $\Sigma$ as it is the mixed derivative $\p^2_{\eps_1\eps_2}\Lambda(\eps_1f_1+\eps_2f_2)$ at $\eps_1=\eps_2=0$. 
Let $v_0$ be an auxiliary smooth function solving $\square_g \s v=0$ in $[0,T]\times\Omega$ with $v_0|_{t=T} =\p_t v_0|_{t=T} = 0$ in  $\Omega$. The function $v_0$ will compensate the fact that $\p_\nu w$ is known only on the lateral boundary $\Sigma$, but not on $\{t=T\}$. The normal derivative $\p_\nu w$ is known on $\{t=0\}$ due to the initial conditions. Multiplying~\eqref{eq:second_deriv} by $v_0$ and integrating by parts on $[0,T]\times \Omega$, we arrive at the following integral identity
 \begin{equation}\label{eq:integral_identity_derivitve}
 \begin{aligned}
  \int_{\Sigma}v_0\s \p^4_{\eps_1 \cdots \eps_4}|_{\vec{\eps}=0} \s \s \Lambda(\eps_1f_1+\cdots + \eps_4f_4)\s \d S& =\int_{[0,T]\times \Omega} \s v_0\s \s \square_g\s  w \s \d V_g  \\
  & =-16\int_{[0,T]\times \Omega}q\s v_0\s v_1\s v_2\s v_3 \s v_4 \s \d V_g,
  \end{aligned}
 \end{equation}
 which we will find to be useful.
This means that the quantity
 \begin{equation}\label{eq:integral_density}
  \int_{[0,T]\times \Omega} q\s v_0\s v_1\s v_2 \s v_3 \s v_4 \s \d V_g  
 \end{equation}
is known from the knowledge of the DN map $\Lambda$. Since the functions $v_j$, $j=1,\ldots,4$, were arbitrary solutions to \eqref{lin:delta_z}, we are able to choose suitable solutions $v_j$ so that the products of the form $v_0\s v_1 \s v_2\s  v_3 \s v_4$ become dense in $L^1([0,T]\times \Omega)$. This recovers the potential $q$ uniquely. The procedure we have now explained  obtains new equations, and an integral identity relating the DN map and the unknown $q$, by differentiating solutions to the nonlinear equation \eqref{eq:Main equation} depending on several parameters. This procedure in general is called the \emph{higher order linearization method}.
 
The earlier work \cite{LLPT} by the authors studied an analogous stability problem in the Minkowski space. There $v_j$ were chosen to be approximate plane waves so that the product $v_1\s v_2 \s v_3 \s v_4$ in the integral \eqref{eq:integral_density} essentially becomes a delta function of a hyperplane. Hence the integral \eqref{eq:integral_density} in \cite{LLPT} became the Radon transformation of $qv_0$ in $\R^{n}$. Since the Radon transformation is invertible, this recovered $q$. 
%The main idea of this approach is the fact that in $1+1$ dimensions, the $L^2$-product of the distributions $\omega_1:=\delta((x-x_0)-(t-t_0))$ and $\omega_2:=\delta((x-x_0)+(t-t_0))$ and the function $q \s v_0$ concentrates at $(q\s v_0)(x_0,t_0)$ for a given $(x_0,t_0)\in \R^{1+1}$. 
In $1+1$ dimensions, the integral \eqref{eq:integral_density} becomes an integral of $qv_0$ against a delta distribution, in which case the recovery of pointwise values of $qv_0$ is trivial. The auxiliary function $v_0$ in $qv_0$ can be eliminated by choosing $v_0$ suitably. 
%One might be tempted to choose $v_1(t,x)=\delta((x-x_0)-(t-t_0))$ and $v_2(t,x)=\delta((x-x_0)+(t-t_0))$. However, these distributions do not work directly, because they are not smooth solutions of the linear wave equation \eqref{lin:delta_z}. This issue can be overcome by choosing $v_1$ and $v_2$ to be approximations of delta distributions by smooth functions.
% The choice of the arguments in the distributions $\omega_1$ and $\omega_2$ is not a coincidence. They were deliberately chosen since the unique solution of the system in $\R^{1+1}$
% \[
% x-x_0=t-t_0, \quad x-x_0= - (t-t_0)
% \]
% is $(x_0,t_0)\in \R^{1+1}$. At this point, one can note that the unique solution of above system is just the intersection of two geodesics in the Minkowski space $\R^{1+1}$.
 
Motivated by the above explanation, in the present work we shall consider the so called \emph{Gaussian beam} solutions $v_j$ to  \eqref{lin:delta_z}. 
One can think of Gaussian beams as wave packets travelling on lightlike geodesics.
In Sections \ref{sec:Gaussian_beams} and~\ref{sec:proof_of_stabilit_estimate} we will show that by using the non-linearity of \eqref{eq:Main equation} and Gaussian beams, one can produce \emph{approximate delta distributions} from the product $v_1v_2v_3v_4$ in \eqref{eq:integral_density}.
This uses the fact that Gaussian beams are solutions to the linear wave equation \eqref{lin:delta_z}  with  exponential concentration to a neighbourhood of a  given lightlike geodesics up to a small error term. Thus, if two different geodesics intersect, then the product of the corresponding Gaussian beams concentrate near the intersection points of the geodesics. The product of four, instead of two, Gaussian beams is required to cancel oscillations of the product of the solutions. (If oscillations would not be cancelled, one would expect not to be able to recover $q$ due to nonstationary phase.)

Let us explain how we use four Gaussian beams in \eqref{eq:integral_identity_derivitve} in more detail.
%
%The non-linearity of \eqref{eq:Main equation} allows one to build products of Gaussian beams  that are exponentially concentrated at specific points.
%
Let us consider $p_0\in W\subset I^-(\Sigma) \cap I^+(\Sigma)\cap ([0,T] \times \Omega)$. We show that there exist two different geodesics $\gamma_1$ and $\gamma_2$ that pass throughout $p_0$ and that intersect $\Sigma$ in a suitable manner. We distinguish two cases depending on whether $\gamma_1$ and $\gamma_2$ intersect only once or multiple times.
% The easiest case is when geodesics intersect only once. 
%The multiple intersections case is more involved than the previous case, but it can follow by using similar arguments. 
%For this reason, we explain first how we recover $q$ uniquely when the geodesics intersect only once.
Let us explain first the simpler case, where the geodesics $\gamma_1$ and $\gamma_2$ intersect only at the point $p_0$. Let $v_1$ and $v_2$ be Gaussian beam solutions to \eqref{lin:delta_z} with respect to $\gamma_1$ and $\gamma_2$. 
%We will see in Section \ref{sec:proof_of_stabilit_estimate} that $|v_1|^2|v_2|^2$ will act as approximate delta functions concentrating at $p_0$. This is mainly due to the growth properties (in certain directions) of Gaussian beam solutions, see Proposition \ref{Gaussian_beam_construction}. This motivates the choice $v_3= \overline{v}_1$ and $v_4=\overline{v}_2$ so that $v_1 \s v_2 \s v_3 \s v_4=|v_1|^2|v_2|^2$.
Making the choice $v_3= \overline{v}_1$ and $v_4=\overline{v}_2$ yields $v_1 \s v_2 \s v_3 \s v_4=|v_1|^2|v_2|^2$.
Evaluating this product, one finds that the product $|v_1|^2|v_2|^2$ is an approximation of the delta distribution concentrated at $p_0$. Therefore, by using the integral identity \eqref{eq:integral_identity_derivitve} for this specific product $v_1 \s v_2 \s v_3 \s v_4$, we can recover $q\s v_0$ at $p_0$. We take $v_0$ to be another Gaussian beam that is nonzero at $p_0$. This way we have recovered $q$ at $p_0$. Repeating the argument for all points of $W$ recovers $q$ on $W$. 

%A perturbation argument shows that the same holds for $q\s v_0$ in a neighbourhood of $p_0$. By repeating the argument for all points of $W$ Covering the compact set $W$ by finitely many open sets Finally, since 
%$\supp \s q$ is a compact set,
%the set $\W$ which we can probe by sending waves from the boundary is a compact set,
%the function $q\s v_0$ can be recovered in 
%$\supp \s q$
%$\W$ via a compactness argument.
%
%For a single intersection point of the geodesics $\gamma_1$ and $\gamma_2$, the function $v_0$ can be chosen so that it can be safely removed. This recovers $q$ uniquely when the geodesics intersect only at $p_0$.

Suppose next that $\gamma_1$ and $\gamma_2$ intersect at points $x_1\leq  \cdots \leq x_P$, $P\geq 2$. 
%One of them is of course $p_0$.
 Using similar arguments as above, the integral  \eqref{eq:integral_density} reduces to an integral of $qv_0$ against a sum of approximative delta functions located at the intersection points $x_1,\ldots,x_P$. That is, by using \eqref{eq:integral_identity_derivitve}, we know from the DN map the quantity 
 \begin{equation} \label{full:contribution_z}
 \sum_{k=1}^P (q\s v_0)(x_k)
 \end{equation}
 up to an error, which can be made arbitrary small by taking a parameter associated to the Gaussian beams large. The task is then to decouple the information of $q\s v_0$ at each single point $x_k$ from this quantity.  
 
 To decouple the information, the choice of $v_0$ plays a crucial role. Recall that the only requirement from $v_0$ was that it satisfies the wave equation $\square_g \s v_0=0$ with Cauchy data vanishing at $t=T$. We show that there is a family $(v_0^{(k)})_{k=1}^P$ of $P$ functions, satisfying the required conditions for $v_0$, with the property that the matrix
 \begin{equation*}%\label{separation_of_multiple_points_matrix}
 \mathcal{V}\s:=\begin{pmatrix}
        v_0^{(1)}(x_1) & v_0^{(1)}(x_2) &\cdots & v_0^{(1)}(x_P) \\
        v_0^{(2)}(x_1) & v_0^{(2)}(x_2) &\cdots & v_0^{(2)}(x_P) \\
        \vdots & & \ddots & \vdots \\
        v_0^{(P)}(x_1) & v_0^{(P)}(x_2) &\cdots & v_0^{(P)}(x_P) \\
    \end{pmatrix} 
%     \begin{pmatrix}
% q(x_1)\\
% q(x_2)\\
% \vdots\\
% q(x_p)
% \end{pmatrix}:= \mathcal{V}\s  \mathcal{Q},
 \end{equation*}
 is invertible.
 Thus, by using \eqref{full:contribution_z} for each $v_0^{(k)}$ in place of $v_0$ separately  we know the  quantity
 \[
  \mathcal{V}\begin{pmatrix}
 q(x_1)\\
% q(x_2)\\
 \vdots\\
 q(x_P)
 \end{pmatrix}%:= \mathcal{V}\s  \mathcal{Q
% (q(x_1), \ldots, q(x_P))
 \]
 from the DN map $\Lambda$.
%  \begin{equation*}%\label{separation_of_multiple_points_matrix}
%  \begin{pmatrix}
%         w_{1}(x_1) & w_{1}(x_2) &\cdots & w_{1}(x_p) \\
%         w_{2}(x_1) & w_{2}(x_2) &\cdots & w_{2}(x_p) \\
%         \vdots & & \ddots & \vdots \\
%         w_{p}(x_1) & w_{p}(x_p) &\cdots & w_{p}(x_p) \\
%     \end{pmatrix} \begin{pmatrix}
% q(x_1)\\
% q(x_2)\\
% \vdots\\
% q(x_p)
% \end{pmatrix}:= \mathcal{V}\s  \mathcal{Q},
%  \end{equation*}
%  
%  
%  where $\mathcal{V}:=(v_0^{(k)}(x_l))_{k,l=1}^P$ and $\mathcal{Q}$ is the vector $(q(x_1), \ldots, q(x_p))$.  The important fact about this \emph{separation matrix} is that the functions $(w_k)_{k=1}^p$ can be chosen so that the matrix $ \mathcal{V}$ is invertible. 
% If $\mathcal{V}$ is the $p\times p$ identity matrix, then $\mathcal{V}\mathcal{Q}= \mathcal{Q}$ and we are done.
Since $\mathcal{V}$ is a known invertible matrix, this uniquely recovers the values of the unknown potential $q$ at the points $x_1,\ldots,x_P$. We shortly explain the idea how the separation matrix $\mathcal{V}$ is constructed in Section \ref{sec:Lorentzian_geom_tools}, while complete statements and proofs about the matter are in Section \ref{sec:multiple_intersections}.

%Using a perturbation argument, the recovery of $q$ can be extended to neighbourhoods of each $x_k$.
%Using the compactness of the recovery set $W$, we finally recover the potential $q$.
 %and later use a compactness argument.  
% One can think that when $\mathcal{V}$ is the $p\times p$ identity matrix, we are decoupling the information contained in $\mathcal{V}\mathcal{Q}$ by giving priority to the values $w_k(x_l)$ when $k=l$ and leaving out all values when $k\neq l$, the latest are actually zero. 
%Using this idea, we prove the existence of a family $v_0=(w_k)_{k=1}^p$ so that $\mathcal{V}$ is invertible. In this case, we essentially have $w_k(x_k)=1$ in the diagonal of $\mathcal{V}$, and the values $w_k(x_l)$ with $k\neq l$ are of order $\tau^{-\alpha}$, for some $\alpha>0$. Here $\tau>0$ is a large parameter involved in the Gaussian beam construction of $v_j$, $j=1,2,3,4$. Thus, making $\tau$ as large as possible, we deduce that $\mathcal{V}$ is invertible. In this way we recover $q$ uniquely in the general case. 
%Among other things, the novelty in our approach is proving the existence of the separation matrix $\mathcal{V}$, whose rigorous construction is given in Lemma \ref{lemma:separation_of_multiple_points}.

So far, we have sketched the proof of unique recovery of $q$ from the DN map $\Lambda$ of \eqref{eq:Main equation}. We briefly discuss how to quantify the uniqueness result and thus to prove a stability estimate. To obtain a stability estimate for $q$ in terms of $\Lambda$, instead of differentiating equation~\eqref{eq:Main equation}, we take the mixed finite difference 
  $D_{\eps_1 \s \cdots \s \eps_4}^4$ of $u_{\eps_1f_1+ \cdots +\eps_2f_4}$ at $\vec{\eps}=0$. In this case, we obtain a slightly different version of the integral identity \eqref{eq:integral_identity_derivitve} given by
\begin{equation}
\begin{aligned}
- 16   \int_{[0,T]\times\Omega} q\s v_0\s v_1\s v_2\s v_3\s v_4 \s \d V_g  &=\int_{\Sigma}v_0\s D_{\eps_1 \s \cdots \s \eps_4}^4\Big|_{\vec{\eps}=0}\Lambda(\eps_1f_1+ \cdots +\eps_4f_4) \s \d S\\
  &\qquad  + \frac{1}{\eps_1 \cdots \eps_4}\int_{[0,T]\times \Omega} \s v_0 \s \square_g \s \tildeR\s \d V_g.
  \end{aligned}
 \end{equation}
Here the second integral on the right is a small error term, where $\tildeR$ is of the size $O(\langle\eps_1, \ldots, \eps_4\rangle^7)$ in an energy space norm. For details, see \eqref{eq:energy_norm} and \eqref{eq:tildeR}--\eqref{est:square_tildeR}. Here we also denote by $\langle\eps_1, \ldots, \eps_4\rangle^7$ an unspecified homogeneous polynomial of order $7$ in $\eps_1,\ldots,\eps_4$. If $p_0\in W$ is fixed, a stability result for $q$ at $p_0$ follows by using Gaussian beams associated to the lightlike geodesics $\gamma_1$ and $\gamma_2$ described above, optimizing with respect to the parameters $\eps_1,\ldots,\eps_4$ and the parameters related to the Gaussian beams $v_1$, $v_2$, $v_3$ and $v_4$. The implied constant of the stability estimate at the fixed point estimate depends on $p_0$. To show that the constant can in fact be taken to be independent of $p_0$ we must vary the geodesics $\gamma_1$ and $\gamma_2$ and the corresponding Gaussian beams smoothly. This requires some work. See Section \ref{sec:Gaussian_beams} for details. In case lightlike geodesics intersect several times, we must also use different separation matrices for different points in $W$. We call a suitable finite collection of separation matrices a \emph{separation filter}. This concept is explained in the next section.

\subsection{Lorentzian geometry tools}\label{sec:Lorentzian_geom_tools}
To prove our main results, we make some constructions in Lorentzian geometry. The main constructions we develop are \emph{boundary optimal geodesics} and \emph{separation matrices}. We now explain briefly what these are. Since we expect the constructions to have applications in related inverse problems as well, and they might also be of interest in  Lorentzian geometry in general, this section is written to be independent of the inverse problem we consider. We follow the terminology of \cite{One83} while we have included the used concepts of causality in Section \ref{sec:preliminaries} for an easy access.

Let us first explain what is a boundary optimal geodesic. As before we consider the subset $[0,T]\times \Omega$ of a globally hyperbolic smooth Lorentzian manifold $\R\times M$, $\dim(M)=n\geq 2$, equipped with the metric \eqref{eq:g} and where $\Omega$ is a smooth submanifold of $M$ with boundary and of dimension $n$. The lateral boundary $\Sigma$ refers to the set $[0,T]\times \p \Omega$ as before. As is by now quite standard, see e.g.~\cite{KLU18,One83}, we say that a geodesic connecting the points $x,y\in N$, $x\leq y$, is optimal if the time separation function $\tau$ of these points vanishes, $\tau(x,y)=0$. An optimal geodesic is always lightlike. The time separation function is the supremum of lengths of piecewise smooth future-directed causal paths from $x$ to $y$, see \eqref{def:time_separation} or \cite{One83} for details.

Let us then consider a point $x\in I^{-}(\Sigma)\cap ([0,T]\times \Omega)$. In the inverse problem of this paper, we consider Gaussian beams that vanish on a neighbourhood of $\{t=T\}$. For this, it is required to find future-directed lightlike geodesics of $[0,T]\times \Omega$ from $x\in [0,T]\times \Omega$ to $\Sigma$, which do not intersect the set $\{t=T\}$. In Lemma \ref{optimal_geo}, we show that we may find a point $z_{\textrm{inf}}$ of the lateral boundary $\Sigma$ and an optimal future-directed geodesic $\gamma$ from $x$ to $z_{\textrm{inf}}$. The situation is illustrated in Figure \ref{pic:Special_optimal_geodesics}. 
\begin{figure}[ht!]
 \centering
  \includegraphics[scale=0.3]{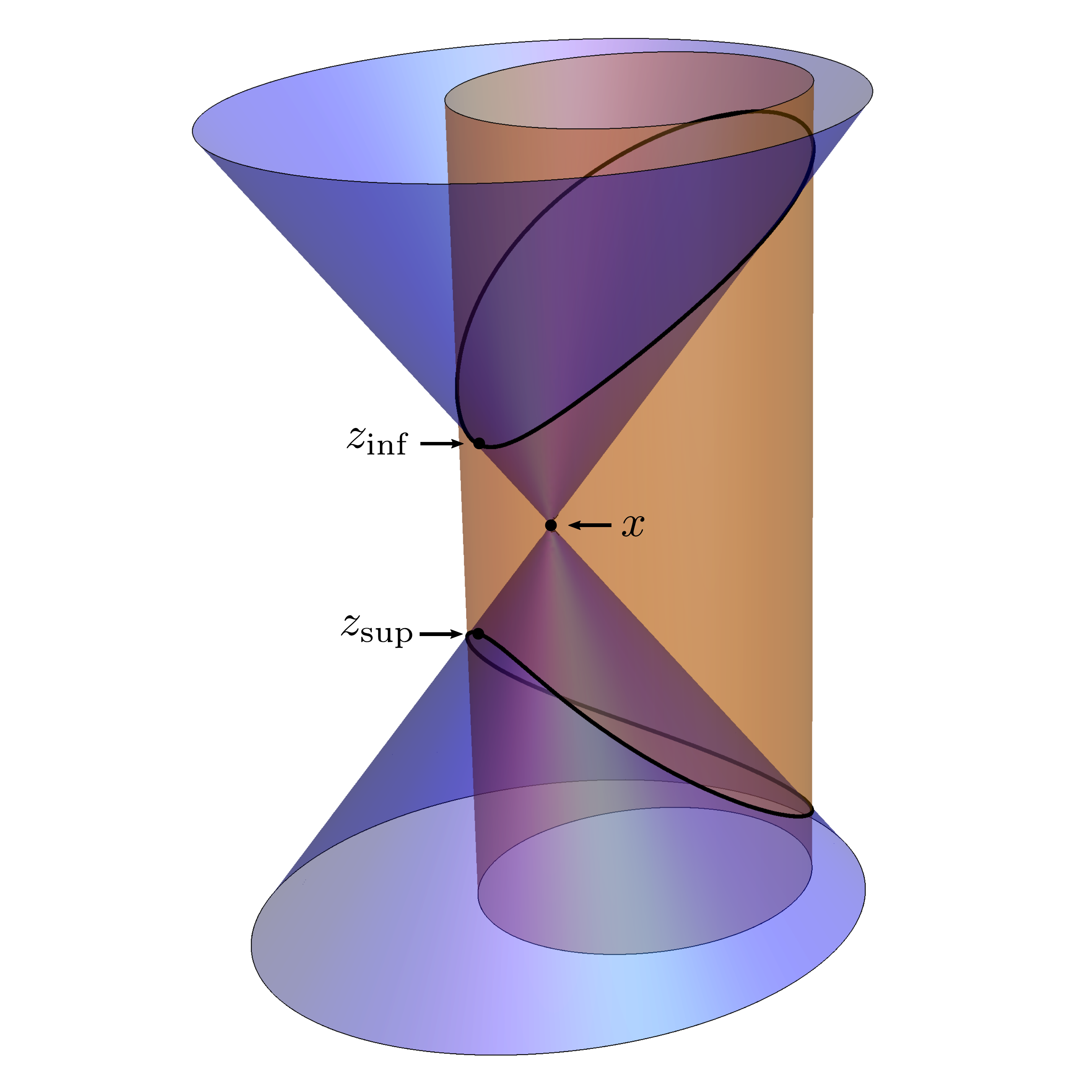}
 \caption{The lateral boundary $\Sigma$ (orange cylinder) intersects the lightcone (blue cone) of a point $x$ (apex of the cone) along the black curves.
 The point $z_\mathrm{sup}$ is the latest and $z_\mathrm{inf}$ the earliest point on $\Sigma$, which can be reached from $x$ by an optimal geodesics. We call these optimal geodesics boundary optimal geodesics. 
 %connecting the latest and earliest points $z_\mathrm{sup}$ and $z_\mathrm{inf}$ on $\Sigma$ to the point $x\in \Omega$ boundary optimal geodesics.
  \label{pic:Special_optimal_geodesics}}
 \end{figure}
In the figure, the point $z_{\textrm{inf}}\in \Sigma$, is the point which has the smallest time coordinate in the intersection of the lightlike future of $x$ (the upper cone) and $\Sigma$. The lightlike geodesic $\gamma$ from $x$ to $z_{\textrm{inf}}$ is not only optimal, i.e. $\tau(x,z_{\textrm{inf}})=0$, but it also necessarily intersects $\Sigma$ transversally even if $\Sigma$ would be nonconvex. We call the geodesic $\gamma$ a boundary optimal geodesic. Note that by deforming $\Sigma$ in the figure to a non-convex manifold, it is possible to find optimal geodesics from $x$ to points in $\Sigma$, which intersect $\Sigma$ tangentially. Therefore, not all optimal geodesics are boundary optimal geodesics. For $x\in I^{+}(\Sigma)\cap ([0,T]\times \Omega)$, we may similarly find a past-directed boundary optimal geodesic from $x$ to $z_{\textrm{sup}}\in \Sigma$ also presented in the figure.

Having explained what optimal geodesics and boundary optimal geodesics are, we are ready to present what a separation matrix  is and how it is constructed. In general, if $x_1,\ldots,x_P\in I^{-}(\Sigma)\cap ([0,T]\times \p\Omega)$ satisfy $x_1\leq \cdots \leq x_P$ we show in Lemma \ref{lemma:separation_of_multiple_points} that there are $P$ solutions $v_k$, $k=1,\ldots,P$, to the wave equation $\square_g v=0$ whose Cauchy data vanish on $\{t=T\}$ such that the matrix
\begin{equation}\label{separation_of_multiple_points_matrix_intro}
 \begin{pmatrix}
        v_{1}(x_1) & v_{2}(x_1) &\cdots & v_{P}(x_1) \\
        v_{1}(x_2) & v_{2}(x_2) &\cdots & v_{P}(x_2) \\
        \vdots & & \ddots & \vdots \\
        v_{1}(x_P) & v_{2}(x_P) &\cdots & v_{P}(x_P) \\
    \end{pmatrix}
 \end{equation}
%  
%  \begin{equation*}%\label{separation_of_multiple_points_matrix}
%  \mathcal{V}\s:=\begin{pmatrix}
%         v^{(1)}(x_1) & v^{(1)}(x_2) &\cdots & v^{(1)}(x_P) \\
%         v^{(2)}(x_1) & v^{(2)}(x_2) &\cdots & v^{(2)}(x_P) \\
%         \vdots & & \ddots & \vdots \\
%         v^{(P)}(x_1) & v^{(P)}(x_2) &\cdots & v^{(P)}(x_P) \\
%     \end{pmatrix} 
% %     \begin{pmatrix}
% % q(x_1)\\
% % q(x_2)\\
% % \vdots\\
% % q(x_p)
% % \end{pmatrix}:= \mathcal{V}\s  \mathcal{Q},
%  \end{equation*}
is invertible. We call the invertible matrix above a separation matrix. Let us consider here the simplest non-trivial case $P=2$ and assume that $x_1,x_2\in I^{-1}(\Sigma)\cap ([0,T]\times \p\Omega)$ satisfy $x_1\leq x_2$. To construct suitable solutions $v_1$ and $v_2$ in this case, we proceed by first choosing two lightlike geodesics as follows. The choice is illustrated in Figure \ref{pic:Geodesics}, where the points $x_1$ and $x_2$ are the intersection points of the black curves. (In our inverse problem the black curves are also geodesics, but that is not important for the present discussion.) By the discussion above, we may find a boundary optimal geodesic $\gamma_1$ from $x_1$ to $x_{1,\textrm{inf}}\in \Sigma$ and another boundary optimal geodesic $\gamma_2$ from $x_2$ to $\Sigma$.
\begin{figure}[ht!]
 \centering
  \includegraphics[scale=0.25]{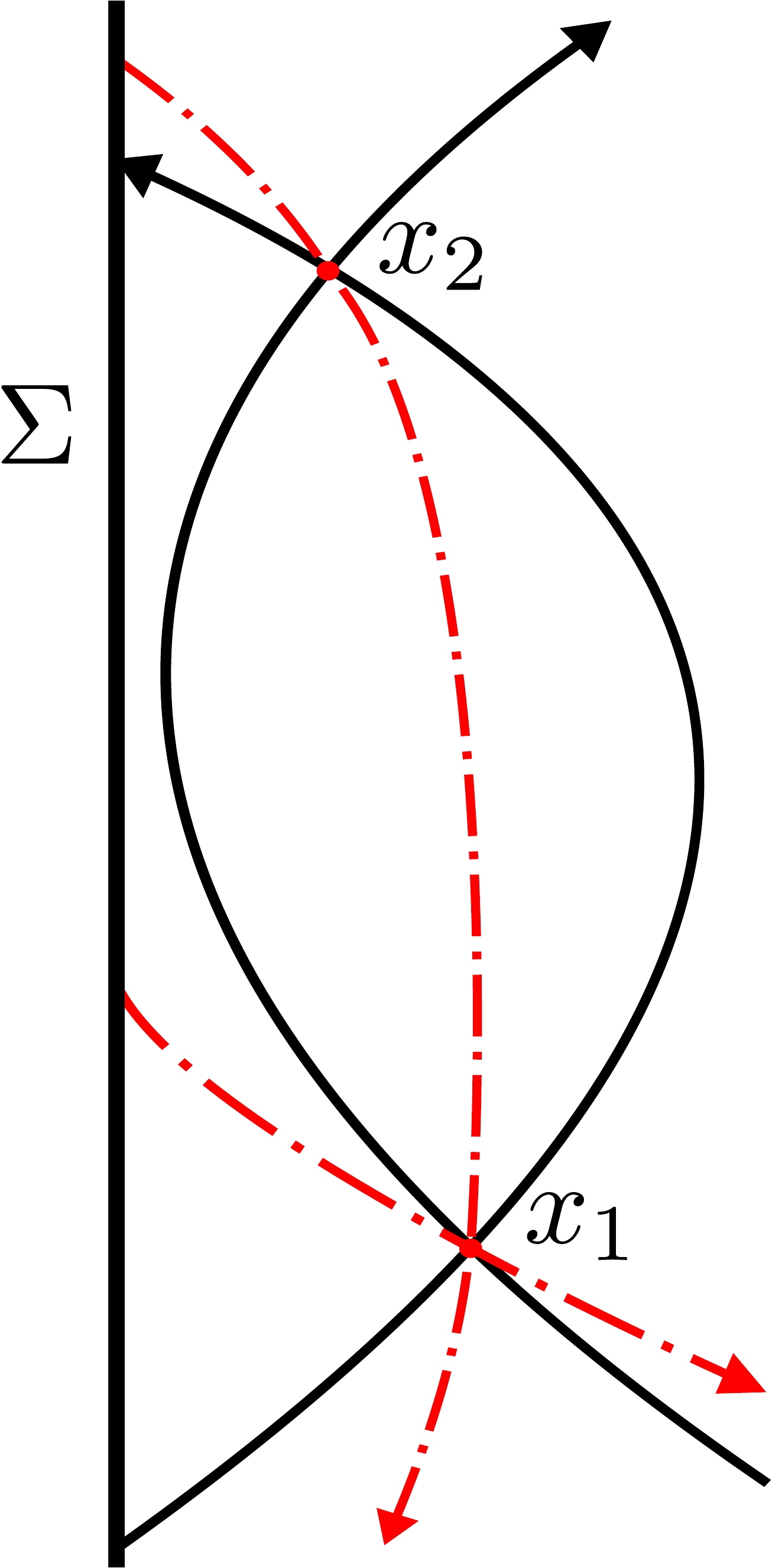}
 \caption{Past-directed lightlike geodesics (red dashed lines) that separate the intersection points $x_1$ and $x_2$ of future-directed lightlike geodesics (black). The geodesics in red and and black intersect $\Sigma$ at a time $t<T$ and $t>0$ respectively. \label{pic:Geodesics}}
 \end{figure}
 Next we note that if $\gamma_1$ also meets $x_2$, then we can perturb the initial direction of $\gamma_1$ at $x_1$ to have a new lightlike geodesic that does not meet $x_2$. Indeed, if the new geodesic would still meet $x_2$, then there would be a shortcut path from $x_2$ to $\Sigma$, which has positive length. This would contradict the condition $\tau(x_1,x_{1,\textrm{inf}})=0$. We refer to the proof of Lemma \ref{lemma:separation_of_multiple_points} for the details.
 
By the above discussion, we have the lightlike geodesic $\gamma_1$ from $x_1$ to $\Sigma$ which does not meet $x_2$ and another lightlike geodesic from $x_2$ to $\Sigma$. Corresponding to these two geodesics there are two Gaussian beams solutions $v_1$ and $v_2$ to $\square_gv=0$ with vanishing Cauchy data at $\{t=T\}$. By using the properties of Gaussian beams, we know that $v_1$ and $v_2$ are concentrated to small neighbourhoods of the corresponding geodesics, respectively. See Section \ref{sec:Gaussian_beams} for details. Thus we have for $k,l=1,2$ that 
\begin{align*}
&|v_k(x_l)|\approx 1,\quad k=l,\\
&|v_k(x_l)|\ll 1,\quad 
k>l,\\
&|v_k(x_l)|\leq c_0,\quad k<l,
\end{align*}
where $c_0>0$ is a constant.
Therefore the matrix $\mathcal{V}$ in \eqref{separation_of_multiple_points_matrix_intro} is approximatively a lower triangular matrix with ones on the diagonal. Thus $\mathcal{V}$ is invertible, hence it is a separation matrix in our terminology. Vaguely speaking, we can separate points by solutions to the wave equation $\square_g v=0$. We mention that a similar condition has been used in the study of inverse problems for elliptic equations in \cite{GST19,LLS19}.

Finally, we mention that when proving our stability result of this paper, we can only use finitely many separation matrices. For this, we show that there are finitely many solutions $v$ to $\square_gv=0$ with vanishing Cauchy data at $\{t=T\}$ such that the separation matrices made out of these solutions can separate any fixed number of points in $I^-(\Sigma)\cap ([0,T]\times \Omega)$ that are distinct in a precise sense. The set of all these solutions is called a \emph{separation filter} and it is denoted by $\mathcal{M}$. See Lemma \ref{lem:separation_filter} for details. A separation filter only depends on the geometry of $([0,T]\times \Omega,g)$.
% In fact, we show in Lemma~\ref{lem:separation_filter} that using only \emph{finitely many} functions $w_i$ suffices to separate all possible intersections of any two distinct light-like geodesics.
% \begin{figure}[ht!]
%  \centering
%   \includegraphics[scale=0.3]{CylinderConeIntersection8_Latex.pdf}
%  \caption{The lateral boundary $\Sigma$ (orange cylinder) intersects the lightcone (blue cone) of a point $x$ (apex of the cone) along the black curve.
%  The point $z_\mathrm{sup}$ is the latest and $z_\mathrm{inf}$ the earliest point on $\Sigma$, from which we reach $x$ by an optimal geodesic. 
%   \label{pic:Special_optimal_geodesics}}
%  \end{figure}

\subsection{Preliminary definitions}\label{sec:preliminaries}
The Sobolev spaces $H^s$ on a compact smooth manifold can be defined in several ways (up to equivalent norms). We define Sobolev spaces first on the manifold $N = \R\times M$
%
%We define Sobolev spaces
using partition of unity on charts, see e.g. \cite{Hormander3,Roe99,Taylor3}.
%Roe 1999: Definition 5.12
%Taylor vol III: Chapter 13.1
%Hörmander vol III: Appendix B
%
Sobolev spaces on the time-cylinder $[0,T]\times \Omega$ are then defined by restriction:
\[
H^s([0,T]\times \Omega) := \{ f\big|_{[0,T]\times\Omega} \mid
 f\in H^s(\R\times M)\}.
\]
%Following \cite[Section 2]{JK95}, 
As usual, the dual space of $H^{r}([0,T]\times \Omega)$, $r\geq 0$, is defined as
\[
\dualH^{-r}([0,T]\times\Omega):= \{ f \in H^{-r}(\R\times M) \mid
 \supp \s  f \s \subset \s [0,T]\times  \overline{\Omega} \}.
\]
It is endowed with the norm $\norm{g}_{\dualH^{-r}([0,T]\times\Omega)}:= {\sup}\, \, \s \s  \frac{|g(v)|}{ \norm{v}_{H^r([0,T]\times M)}}$, where the supremum is over all $v\in H^{r}([0,T]\times M)$ with   $\supp \s v \s \subset [0,T]\times \overline{\Omega}$.
By Riesz representation theorem, one can always find $f_0\in  H^{r}(\R\times M)$ so that for all $v\in H^r(\R\times M)$
\[
\norm{f}_{\tilde H^{-r}([0,T]\times \Omega)}= \norm{f_0}_{H^{r}(\R\times M)}, \qquad f(v)= \left \langle f_0 , v  \right \rangle. 
%, \quad v\in H^{r}(\R\times M).
\]
%It allows us to see the quantity $g(v)$ as the $L^2$-inner product $\left \langle g_0 , v  \right \rangle$. 
Additionally, if  $\supp \s v \s \subset  [0,T]\times \overline{\Omega}$, then we have for all $v\in  H^{r}([0,T]\times M)$ the estimate
\[
|f(v)|= |\left \langle f_0 , v  \right \rangle| \leq \norm{f}_{\dualH^{-r}([0,T]\times \Omega)} \norm{v}_{H^r([0,T]\times \Omega)}.
\]
%This fact will be used in the computations when deriving stability estimates. 
Sobolev spaces of the manifold $\Omega$ with boundary are defined similarly. 

%
%Another approach  would be to use an auxiliary Riemannian metric on $N$.
%%
%This Riemannian metric can be chosen so that it is complete, see \cite{NO61}.
%The metric could be constructed by changing the sign of the negative eiqenvalue of the metric $g_{ij}$ in local coordinates.
%
%We will use the notation $H^s_0$ to denote the closure of $C^\infty_0$ in the $H^s$ norm.

Next we recall some notations and definitions in time-oriented Lorentzian manifolds, see for example \cite{Beem,One83}.
A smooth path $\mu:(a,b)\to N$ is said to be time-like if $g(\dot\mu(s),\dot\mu(s))<0$ for all $s \in (a,b)$.
The path $\mu$ is causal if $g(\dot\mu(s),\dot\mu(s))\leq 0$ and $\dot\mu(s)\neq 0$ for all $s\in (a,b)$.
For $p,q\in N$ we denote $p\ll q$ if $p\neq q$ and there is a future-pointing time-like path from $p$ to $q$.
Similarly, $p<q$ if $p\neq q$ and there is a future-pointing causal path from $p$ to $q$, and $p\leq q$ when $p=q$ or $p<q$.
The chronological future of $p\in N$ is the set $I^+(p)=\{q\in N\mid p\ll q\}$ and the causal  future of $p$ is $J^+(p)=\{ q\in N\mid p\leq q\}$.
The chronological past $I^-(q)$ and causal past $J^-(q)$ of $q\in N$ are defined similarly.
If $A\subset N$, then we denote $J^\pm(A)=\cup_{p\in A}J^\pm(p)$.
%
%The diamond sets are denoted by $J(p,q)= J^+(p)\cap J^-(q)$ and $I(p,q)=I^+(p)\cap I^-(q)$.
%
%
%%
%A time-orientable Lorentzian manifold $(N,g)$ is globally hyperbolic if and only if there are no closed causal paths in $N$ and for $q_1,q_2\in N$ with $q_1<q_2$ the diamond $J(q_1,q_2)\subset N$ is compact.
%%
%Let $\mathcal L^+(p)$ be the boundary of the future set $J^+(p)$.
%%
%This is the set of points which can be reached from $p$ by travelling along lightlike geodesics.
%
The sets  $I^\pm(p)$ are always open and in globally hyperbolic manifolds the sets $J^\pm(p)$ %and $\mathcal{L}^+(p)$ 
are closed, see e.g. \cite[Lemma 14.22]{One83}.
%\cite[Lemma 14.22 and Corollary 14.27]{One83}
%
The sets $I^\pm(p)$ and $J^\pm(p)$ are related by $\mathrm{cl}(I^\pm(p))=J^\pm(p)$.
%
%%
%Let $L_pN=\{\xi\in T_pN\setminus\{0\}\mid g(\xi,\xi)=0\}$ be the set of light-like vectors in the tangent space $T_pN$.
%%
%Let also $L_p^+N$ and $L_p^-N$ denote the future and past light-like vectors in $T_pN$. 
%%
Finally, a geodesic from $p\in N$ with initial direction $\xi\in T_pN$ is denoted by $\gamma_{p,\xi}(t) = \exp_p(t\xi)$.

\subsection*{Structure of the paper}
This paper is organized as follows.
In Section~\ref{sec:main_results} we present our main results and explain briefly the structure of the proofs.
Section~\ref{sec:forward} studies the forward problem of the non-linear equation \eqref{eq:Main equation}. Most of the proofs of Section \ref{sec:forward} are included in the Appendix \ref{app:forward}.
%
%As most of the proofs related to the forward problem can be found in literature, we have opted to omit them or include them in the Appendix~\ref{app:forward}.
%
Section~\ref{sec:Gaussian_beams} concerns the construction of Gaussian beams in Lorentzian manifolds.
In Section~\ref{sec:separation_of_points} we construct the tools of Lorentzian geometry which we use in our inverse problem. This section in particularly shows it is possible distinguish different points of a Lorentzian manifold by using solutions to the wave equation. The section introduces the concepts of \emph{boundary optimal geodesics} and \emph{separation matrices}.
Finally, in Section~\ref{sec:proof_of_stabilit_estimate} we collect the results we have obtained until that point to give a proof for our main theorem.
For clarity, the proof is split into several parts.

\section{Well-posedness of the forward problem}\label{sec:forward}
%
% and let us denote for the sake of brevity
%\[
%X^{s}(\Omega):= C([0,T]\; H^{s}(\Omega))\cap C^{s}([0,T]\; L^2(\Omega)).
%\]
%If there is no danger of misunderstanding, we simply denote $X^s(\Omega)$ by $X^s$, or just by $X$ if the index $s$ is additionally known from the context. The norm of the Banach space $X^s(\Omega)$ is given by
%\[
%\Vert f \Vert_{X^s} := \sup_{0<t<T} \Big(
%\Vert f(\ccdot,t)\Vert_{H^{s+1}(\Omega)} + \Vert \partial_t f(\ccdot,t)\Vert_{H^s(\Omega)}
%\Big).
%\]
To prove existence of small solutions for the non-linear wave equation \eqref{eq:Main equation}, we start by recalling the corresponding results for the linear initial-boundary value problem
\begin{equation*}%\label{wave-eq_with_data}
\begin{cases}
\square_g u  = F, \quad \text{in } [0,T]\times\Omega,\\
u=f, \,\quad\ \text{ on } [0,T]\times\p\Omega ,\\
u\big|_{t=0} = u_0,\quad \partial_t u\big|_{t=0} = u_1,&\text{in }\Omega.
\end{cases}
\end{equation*}
% for the linear wave operator.
Let $s\in \N$. The convenient spaces for solutions of the wave equation are called \emph{energy spaces} $E^s$ (see e.g.~\cite[Definition 3.5 on page 596]{CB08}),
defined as
\[
 E^s=\bigcap_{0\leq k \leq s}C^k([0,T]\; H^{s-k}(\Omega)).
\]
These spaces are equipped with the norm
\begin{equation}\label{eq:energy_norm}
 \norm{u}_{E^s}=\sup_{0<t<T}\sum_{0\leq k \leq s}\norm{\p_t^ku(\ccdot,t)}_{H^{s-k}(\Omega)}.
\end{equation}
As is the case with the Sobolev spaces $H^s$, the space $E^s$ is an algebra if $s>(n+1)/2$ (see e.g.~\cite{CB08}) and we have the norm estimate
\begin{equation}\label{banach_algebra}
 \norm{u\s v}_{E^s}\leq C_s \norm{u}_{E^s}\norm{v}_{E^s}, \text{ for all } u,v\in E^s.
\end{equation}

\begin{remark} \label{sob:emb_z}
We note that $E^s\subset H^s([0,T]\times\Omega)$. Conversely, due to the standard Sobolev embedding $H^s([0,T]\times\Omega)\subset C^k([0,T]\times\Omega)$, when $s>k+\frac{n+1}{2}$, we have that $H^{s'}([0,T]\times\Omega)\subset E^s$, when $s'>s+\frac{n+1}{2}$. In particular,
\begin{equation}\label{ineq:sob_emb}
\norm{u}_{H^s([0,T]\times\Omega)} \lesssim \norm{u}_{E^s} \lesssim \norm{u}_{H^{s'}([0,T]\times\Omega)}.
\end{equation}
\end{remark}

For the wave equations we consider, we need to assume certain compatibility conditions between the boundary values and the initial data.
%
%The standard compatibility conditions of order $s\in \N$ for the equation \eqref{eq:Main equation} are given as follows.
The compatibility conditions for the equation \eqref{eq:Main equation} to order $2$ are given by
\begin{equation}\label{eq:compatibility}
\begin{aligned}
 f|_{t=0}&=u_0|_{\p \Omega}, \quad \p_tf|_{t=0}=\p_tu|_{\{0\}\times\p \Omega  }=u_1|_{\p \Omega}, \\
 \p_t^2f|_{t=0}&=\p_t^2u|_{\{0\}\times\p \Omega}=\beta^{-1}|_{\{0\}\times\p \Omega}\left(\Delta_h u_0|_{\p \Omega} +F|_{\{0\}\times\p \Omega}\right). %\quad \p_t^3f|_{t=0}=\cdots, \quad  \p_t^sf|_{t=0}=\cdots. % \text{ etc.}
 \end{aligned}
\end{equation}
The compatibility conditions up to general order $s$ are obtained by setting $\p_t^kf|_{t=0}=\p_t^ku|_{\{0\}\times\p \Omega}$, for $k=0,\ldots,s$, and then solving for $\p_t^ku|_{\{0\}\times\p \Omega}$ in terms of the initial data by using the equation $\square_g\s  u =F$.
%and similarly for the higher order derivatives up to order $s$. 
These conditions guarantee that at the boundary $\p \Omega$ the initial data $(u_0,u_1)$ is compatible with the corresponding boundary condition $f$. 
These conditions have been discussed for example in~\cite[Section 2.3.7]{KKL01} in the simpler case where the metric is time-independent.
%The case where the metric is time-independent these conditions have been discussed for example in~\cite[Section 2.3.7]{KKL01}.
%
Especially, if $\p_t^kf|_{t=0}=0$ for all $k=0,\ldots,s$, or if $f$ is supported away from the Cauchy surface $\{t=0\}$, and $F\equiv 0$ and $u_0\equiv u_1\equiv 0$, then the compatibility conditions of order $s$ hold. 
%
% In this paper, we have that the initial values satisfy $u_0,u_1\equiv 0$ and that the boundary values $f$ are supported away from the Cauchy surface $\{t=0\}$.
% %
% %Similarly, we assume that the unknown potential function $q$ is admissible.
% %
% This means that in our application the compatibility conditions are satisfied.

\begin{proposition}[Existence and estimates for linear equation, \cite{Ikawa,LLT86}]\label{thm:energy}
Assume that $(\R\times M, g)$ is a globally hyperbolic Lorentzian manifold as in \eqref{eq:g} and $\Omega\subset M$ is a compact submanifold with non-empty boundary.
Let $s\in\N$ be a positive integer and assume that $F\in E^s$, $f\in H^{s+1}(\Sigma)$, $u_0\in H^{s+1}(\Omega)$ and $u_1\in H^{s}(\Omega)$ satisfy the compatibility conditions.
Then the equation
\begin{equation}\label{wave-eq_with_data}
\begin{cases}
\square_{g} u = F, &\text{in } [0,T]\times \Omega,\\
u=f, &\text{on } \Sigma, \\
u=u_0,\, \p_t u = u_1,&\text{in } \{t=0\}\times \Omega.
\end{cases}
\end{equation}
has a unique solution $u\in E^{s+1}$ satisfying
\begin{equation}\label{energy_estimate}
\Vert u \Vert_{E^{s+1}} 
\leq
C\Big(
\Vert F\Vert_{E^s} + \Vert f\Vert_{H^{s+1}(\Sigma)} + \Vert u_0 \Vert_{H^{s+1}(\Omega)} + \Vert u_1 \Vert_{H^s(\Omega)}
\Big)
\end{equation}
and $\p_\nu u\big|_{\Sigma}\in H^s(\Sigma)$.
\end{proposition}
As we could not find a proof for Proposition~\ref{thm:energy} in general for globally hyperbolic Lorentzian manifolds, we have included one in Appendix~\ref{app:forward}.
The energy estimates of the linear problem \eqref{wave-eq_with_data} directly allow us to conclude that the non-linear problem \eqref{eq:Main equation} has a unique small solution. The proof of the following lemma is similar to the one in \cite[Proof of Lemma 1, Appendix A]{LLPT}. We omit the proof.
\begin{lemma}\label{lemma:nonlinear-solutions}
Let $m\geq 2$ be an integer and $\Omega\subset M$ be a compact submanifold, $\dim(\Omega)=\dim(M)$, with nonempty boundary. Assume $s\in \N$ is such that  $s+1>(n+1)/2$. Suppose  that $q\in C^{s+1}([0,T]\times \Omega)$ satisfies the a priori bound $\Vert q\Vert_{C^{s+1}}\leq c$.
%satisfies $\norm{q}_{C^s([0,T]\times \Omega)}\leq L$.
%
Then there is $\kappa>0$ and $\rho>0$ such that if $f\in H^{s+1}(\Sigma)$ satisfies $\Vert f\Vert_{H^{s+1}(\Sigma)} \leq \kappa$,
  and $\p_t^{\alpha}f|_{t=0}=0$ for all $ \alpha=0, \ldots, s$ on $[0,T]\times\p\Omega$,
then
there is a unique solution to 
\begin{equation}\label{eq:nonlinear_equation_lemma}
\begin{cases}
\square_g u + qu^m   =0, &\text{in } [0,T]\times \Omega,\\
u=f, &\text{on } [0,T]\times \p \Omega, \\
u\big|_{t=0} = \partial_t u\big|_{t=0} = 0,&\text{in } \Omega
\end{cases}
\end{equation}
in the ball
\[
B_\rho(0):= \{u \in E^{s+1} \mid \Vert u \Vert_{E^{s+1}} < \rho\}\subset E^{s+1}.
\]
Furthermore, the solution satisfies the estimate
\begin{equation}\label{eq:solution-estimate1}
\Vert u\Vert_{E^{s+1}} \leq C_0\s  \Vert f\Vert_{H^{s+1}(\Sigma)},
\end{equation}
where $C_0>0$ is a constant independent of $f$ and $q$.
\end{lemma}
If the boundary data of the non-linear equation~\eqref{eq:nonlinear_equation_lemma} depends on small parameters, we may expand the corresponding solution $u$ in terms of the small parameters. Indeed, let $\eps_1,\ldots, \eps_m>0$ be small parameters and denote 
$$
\vec\eps = (\eps_1,\ldots,\eps_m).
$$
Consider the following boundary value in \eqref{eq:nonlinear_equation_lemma}
$$
f(x) = \sum_{j=1}^m \eps_j f_j(x),
$$
where $f_j\in H^{s+1}(\Sigma)$, $j=1,\ldots,m$, satisfies the compatibility conditions to order $s$ and $\Vert f\Vert_{H^{s+1}(\Sigma)} \leq \kappa$ for some $\kappa>0$.
Let us denote in the usual multi-index notation
$$
\k = (k_1,\ldots,k_m),
$$
where $k_j\in \{0,\ldots,m\}$.
Then by repeating proof of Proposition 1 in~\cite{LLPT}, we find that $u$ can be expanded as
\begin{equation}\label{id:expam_epsilons}
u=\sum_{j=1}^m\epsilon_j v_j +\sum_{|\k|=m}\binom{m}{k_1,\ldots, k_m}\epsilon_1^{k_1}\cdots \epsilon_m^{k_m}w_{\k} 
+\mathcal{R}.
\end{equation}
The functions $v_j$, $j=1,\ldots,m$, satisfy
\begin{equation}\label{eq:norm_21}
\begin{cases}
\square_g v_j=0, &\text{in } [0,T]\times\Omega,\\
v_j=f_j,&\text{on }  [0,T]\times\p\Omega, \\
v_j\big|_{t=0} = 0,\quad \partial_t v_j\big|_{t=0} = 0,&\text{in } \Omega
\end{cases}
\end{equation}
and % for $k_j\in \{0,\ldots, m\}$
the functions $w_\k$ satisfy
\begin{equation}\label{eq:norm_2}
\begin{cases}
\square_g w_\k +  q\s  v_1^{k_1}\cdots v_m^{k_m}  = 0, &\text{in }  [0,T]\times\Omega,\\
w_\k=0, &\text{on } [0,T]\times\p\Omega, \\
w_\k\big|_{t=0} = 0,\quad \partial_t w_\k\big|_{t=0} = 0,&\text{in } \Omega.
\end{cases}
\end{equation}
The remainder $\mathcal{R}$ is bounded in the energy spaces as follows:
\begin{equation}\label{est:square_R}
\begin{split}
\norm{\mathcal{R}}_{E^{s+2}}&\leq c(s, T)\s   \norm{q}_{E^{s+1}}^2 \left\Vert \sum_{j=1}^m \eps_jf_j\right\Vert_{H^{s+1}(\Sigma)}^{2m-1},  \\
\norm{\square\, \mathcal{R}}_{E^{s+1}} &\leq C(s, T)\s   \norm{q}_{E^{s+1}}^2 \left\Vert \sum_{j=1}^m \eps_jf_j\right\Vert_{H^{s+1}(\Sigma)}^{2m-1}.
\end{split}
\end{equation}

By using the expansion formula \eqref{id:expam_epsilons}, we will next derive an integral equation, which relates the potential $q$ to the DN map.  In general, relating an unknown in an inverse problem for a non-linear equation to an formula for solutions is called a \emph{higher order linearization} method. See for example \cite{KLU18,LLLS19a,LUW18}, where solutions are differentiated with respect to small parameters. 
However, as we are interested in stability of our inverse problem, we need accurate control on the remainder terms.
For this reason, following \cite{LLPT}, instead of differentiating we use finite differences $D^m_{\vec\eps}$. 
The mixed finite difference of $u$ at $\vec\eps=0$, that is, $\eps_1=\cdots=\eps_m=0$, is defined by the formula
\begin{equation}\label{eq:fin_diff}
D_{\vec\eps}^m\big|_{\vec\eps=0} u_{\eps_1f_1+\cdots+\eps_mf_m}
=
\frac{1}{\eps_1\cdots\eps_m}\sum_{\sigma\in\{0,1\}^m}
(-1)^{|\sigma|+m}u_{\sigma_1\eps_1 f_1+\ldots+\sigma_m\eps_m f_m},
\end{equation}
where $u_{\eps_1f_1+\cdots+\eps_mf_m}$ is the unique solution to \eqref{eq:nonlinear_equation_lemma} with $f$ replaced by $ \varepsilon_1 f_1 +\cdots + \varepsilon_m f_m$.
Then the mixed finite difference $D^m_{\vec\eps}$ of the solution $u$ of~\eqref{eq:nonlinear_equation_lemma} takes the form
\begin{equation}\label{eq:mixed_difference}
 D^m_{\vec\eps}\big|_{\vec\eps=0} u=m!\s w_{1,1,\ldots,1}+ D^m_{\eps_1\eps_2\cdots\eps_m}\big|_{\vec\eps=0}\mathcal{R}.
\end{equation}
For more details about the finite differences of $u$, we refer the reader to \cite[Appendix C]{LLPT}.

Let $v_0$ be an auxiliary function solving $\square_g \s v_0=0$ with $v_0|_{t=T} =\p_t v_0|_{t=T} = 0$ in  $\Omega$. 
By multiplying the DN-map $\Lambda$ by $v_0$ and integrating by parts we obtain
 \begin{align*}
 &\int_{\Sigma}v_0\s D_{\vec\eps}^m\big|_{\vec\eps=0} \s \Lambda(\eps_1f_1+\cdots+\eps_mf_m) \s \d S =\int_{\Sigma}v_0\s D_{\vec\eps}^m\big|_{\vec\eps=0} \s \p_\nu u_{\eps_1f_1+\cdots+\eps_mf_m} \s \d S \\
 &\quad\quad=m!\int_{ [0,T]\times\Omega} v_0 \square_g\s w_{1,1,\ldots,1} \s \d V_g  
 +
  \frac{1}{\eps_1\cdots\eps_m}\int_{ [0,T]\times\Omega} v_0 \square_g\s\widetilde{\mathcal{R}}\s \d V_g . %\square\s w_{1,1,\ldots,1} \square D^m_{\eps_1\eps_2\cdots\eps_m}\big|_{\eps_1=\cdots=\eps_m=0}u=m!w_{1,1,\ldots,1}+ D^m_{\eps_1\eps_2\cdots\eps_m}\big|_{\eps_1=\cdots=\eps_m=0}\widetilde{\mathcal{R}},
\end{align*}
Here we denoted 
\begin{equation}\label{eq:tildeR}
\tildeR := \eps_1\eps_2 \ldots \eps_m \s D_{\vec\eps}^m\big|_{\vec\eps=0} \mathcal{R}
\end{equation} 
 and $\tildeR$ satisfies
\begin{equation}\label{est:square_tildeR}%removed factors (-1)^{|\sigma|+m} -TT
\begin{split}
\norm{\widetilde{\mathcal{R}}}_{E^{s+2}}&\leq c(s, T)\s   \norm{q}_{E^{s+1}}^2 \sum_{\sigma\in\{0,1\}^m}
\Vert \sigma_1\varepsilon_1 f_1 +\cdots + \sigma_m\varepsilon_m f_m\Vert_{H^{s+1}(\Sigma)}^{2m-1},  \\
\norm{\square\, \widetilde{\mathcal{R}}}_{E^{s+1}} &\leq C(s, T)\s   \norm{q}_{E^{s+1}}^2 \sum_{\sigma\in\{0,1\}^m}
\Vert \sigma_1\varepsilon_1 f_1 +\cdots + \sigma_m\varepsilon_m f_m\Vert_{H^{s+1}(\Sigma)}^{2m-1}.
\end{split}
\end{equation}
We have arrived to the following integral identity which connects the potential $q$ with the DN-map $\Lambda$.\\
\noindent\textbf{Integral identity:}
\begin{equation}\label{eq:integral_identity_finite_difference}
\begin{aligned}
  -m!\int_{ [0,T]\times\Omega} q\s v_0\s v_1\s v_2\cdots v_m \s \d V_g 
  &=
\int_{\Sigma}v_0\s D_{\vec\eps}^m\big|_{\vec\eps=0}\s \Lambda(\eps_1f_1+\cdots+\eps_mf_m)\s \d S\\&\qquad + \frac{1}{\eps_1\eps_2\cdots\eps_m}\int_{ [0,T]\times\Omega} v_0\square\s \widetilde{\mathcal{R}}\s \d V_g .
\end{aligned}
\end{equation}
Our analysis of the inverse problem is based on this formula.

\section{Gaussian beams}\label{sec:Gaussian_beams}   
In this section we record some facts about Gaussian beams. 
Gaussian beams on a Lorentzian manifold $(N,g)$, $\dim(N)=n+1\geq 3$, are approximate solutions to the equation $\square_gv=0$. If $s$ is a geodesic parameter of a lightlike geodesic $\gamma$ and $(s,y)$, $y=(y_1,\ldots,y_n)\in \R^n$, are suitable Fermi coordinates (see \eqref{Fermi_coords_def} below) on a neighbourhood of the graph of $\gamma$, then a Gaussian beam in the coordinates $(s,y)$ looks roughly like
\begin{equation}\label{Gaussian_beam_behvior}
 e^{iy_1\tau -a\tau |y|^2},
\end{equation}
up to a normalization. Here $a>0$ and $\tau$ is a large parameter. Therefore, the qualitative behavior of a Gaussian beam is oscillation in a direction $y_1$ transversal to the geodesic $\gamma$ and Gaussian concentration around the graph of $\gamma$. We denote the graph of $\gamma$ by $\Gamma$. 

The construction of Gaussian beams is well-known, see e.g.~\cite{BBM,FO19, Ralston}. We include details about the construction since we wish to keep track of the constants that will be implicit in our stability estimate of Theorem~\ref{thm:stability}.  Our presentation of the construction follows closely~\cite[Section 4]{FO19} to which we refer for omitted details. %We omit details that not crucial for our purposes. 
We mention here the recent work \cite{KTS20}, which constructs related Gaussian beam quasimodes in a Riemannian setting by using more sophisticated methods, which lead to better estimates.

Fermi coordinates are constructed by inverting the map
\begin{equation}\label{Fermi_coords_def}
 (s,y)\mapsto \text{exp}_{\gamma(s)}\big(\sum_{k=1}^{n}y^k\s e_k(s)\big) \in X.
\end{equation}
Here $e_k(s)$ are the parallel transportations along a lightlike geodesic $\gamma$ of the last $n$ vectors of a frame $\{e_0,e_1,\ldots, e_n\}$ of $T_{\gamma(0)}$ with 
\[
 e_0=\dot\gamma(0).
\]
The other vectors of the frame are chosen so that for $j,k=2,\ldots n$ hold
\begin{equation}\label{light_cone_frame}
  g(e_0,e_0)=0, \quad g(e_1,e_1)=0, \quad g(e_0,e_1)=-2, \quad g(e_j, e_k)=\delta_{jk}.
\end{equation}
The frame $\{e_0,e_1,\ldots, e_n\}$ is called a pseudo-orthonormal frame. (Due to relation to the usual light-cone coordinates, we could also call it a lightcone frame.)
Since the frame $\{e_0(s),e_1(s),\ldots, e_n(s)\}$ is the parallel transportation of $\{e_0,e_1,\ldots, e_n\}$ along  $\gamma$, the conditions~\eqref{light_cone_frame} hold for $e_j$, $j=0,\ldots, n$, replaced with $e_j(s)$ and $e_0(s)=\dot\gamma(s)$. 

We work in the Fermi coordinates described above. 
In the Fermi coordinates $(s,y)$, the geodesic $\gamma$ corresponds to $(s,0)$ and the coordinate representation $g|_{\gamma}=g(s,0)$ of the metric $g$ restricted to $\gamma$ satisfies
\begin{equation}\label{formula_for_g}
 g|_{\gamma}=-2ds\s dy_1+ \sum_{k=2}^n dy_k\s dy_k.
\end{equation}
%We also have that the first order derivatives $\p_jg_{lk}$, $j,k,l=0,\ldots, n$, of the metric $g_{lk}$ vanishes on $(s,0)$ in the coordinates on $\gamma$. 
%We remark for future reference that the inverse 

%We refer to~\cite{FO19}, where the proof is done in the Fermi coordinates described above. 
Gaussian beams are constructed by using a WKB ansatz $e^{i\tau \Theta(s,y)} a(s,y)$ to approximatively solve the equation $\square_gv=0$ in the Fermi coordinates $(s,y)$. We have
 \begin{equation}\label{box_of_WKB}
  \square_g(e^{i\tau \Theta}a)=e^{i\tau \Theta}(\tau^2g(d\Theta,d\Theta)-2i\s \tau\s g(d\Theta,da)+i\tau\s (\square_g \Theta)a+ \square_g a).
 \end{equation}
 We will choose a \emph{phase function} $\Theta$ and an \emph{amplitude function} $a$ so that the right hand side of~\eqref{box_of_WKB} is $O(\tau^{-K})$ in $H^k([0,T]\times \Omega)$ for given $k\geq 0$ and $K\in \N$.
 To do so,  we first approximatively solve the eikonal equation
 \begin{equation}\label{eikonal_equation}
  g(d\Theta,d\Theta)=0.
 \end{equation}
 %Typically one is satisfied with an approximate solution.
 After finding an (approximative) solution $\Theta$ to the eikonal equation, we equate the last three terms of~\eqref{box_of_WKB} by inserting $\Theta$ into
 \begin{equation}\label{transport_eq}
  -2i\s \tau\s g(d\Theta,da)+i\tau(\square_g \Theta)a+ \square_g a=0.
 \end{equation}
 By assuming an expansion 
 \begin{equation}\label{ansatz_for_a}
  a=a_0+\tau^{-1}a_{1}+\tau^{-2}a_{2}+\cdots +\tau^{-N}a_{N}
 \end{equation}
 for the amplitude $a$, where $N \in \mathbb{N}$ is to be chosen later, we are led by equating the powers of $\tau$ to a family of $N+1$ equations
 \begin{align}\label{transp_ord_0}
  &-2i\s g(d\Theta,da_0)+i(\square_g \Theta)a_0=0, \\
\label{transp_ord_higher}
  &-2i\s g(d\Theta,da_{j})+i(\square_g \Theta)a_{j}-\square_g a_{j-1}=0,
 \end{align}
 $j=1,\ldots, N$. We solve these equations approximatively and recursively in $j$ starting  from $a_0$. % then solved either precisely or approximately. 
% 
%  
%  having an expansion
%  \begin{equation}\label{ansatz_for_a}
%   a=a_0+\mu^{-1}a_{-1}+\mu^{-2}a_{-2}+\cdots +\mu^{-N}a_{-N}.
%  \end{equation}
 The equations~\eqref{transp_ord_0} and~\eqref{transp_ord_higher} are called transport equations.
 
 %For our purposes solving the eikonal equation and the transport equations approximately is enough. (An example where the equations are solved exactly is in~\cite{} in the elliptic setting.) 
 
 In what follows, we refer to~\cite{FO19} for omitted details. 
%  
%  By an approximative solution to an equation $F(s,y)=0$, $F$ smooth function, we understand  to the eikonal or any of the transport equations, we understand a function of $(s,y)$ so that the Taylor expansion of the equation in $y$ on the geodesic $(s,0)$ vanishes in the Fermi coordinates $(s,y)$. 
%  
 To solve the eikonal equation \eqref{eikonal_equation} approximatively, one sets
 \[
  \Theta=\sum_{j=0}^N\Theta_j(s,y),
 \]
 where $\Theta_j(s,y)$ is a homogeneous polynomial of order $j$ in $y\in \R^n$. We say that  $g(d\Theta,d\Theta)$ vanishes to order $N$ on $\Gamma$, or that $g(d\Theta,d\Theta)=0$ is satisfied to order $N$ on $\Gamma$,
 %or that $g(d\Theta,d\Theta)$ vanishes to order $N$ on $\Gamma$, 
 if 
  \[
   (\p_y^\alpha g(d\Theta,d\Theta))(s,0)=0,
  \]
  where $\alpha$ is any multi-index with $\abs{\alpha}\leq N$.
  %and $(s,0)$ is in the domain of the Fermi-coordinates. 
  %If this is the case, we also say that $\Theta$ solves the eikonal equation to order $N$ or approximatively. %(Recall that $(s,0)$ corresponds to the geodesic in the Fermi coordinates.) 
  We set
 \begin{equation}\label{Theta_first_ord}
  \Theta_0=0 \text{ and } \Theta_1=y_1.
 \end{equation}
 It follows that
 \[
  g(d\Theta,d\Theta)(s,0)=0 \text{ and } (\p_{y_l}g(d\Theta,d\Theta))(s,0)=0,
 \]
 where $l=1,\ldots,n$. That is, the eikonal equation~\eqref{eikonal_equation} is satisfied to order $1$ on $\Gamma$.  The conditions~\eqref{Theta_first_ord} imply the invariantly written conditions
 \[
\Theta(\gamma(s)) = 0 \text{ and }\nabla \Theta(\gamma(s)) = e_1(s).  
 \]
 %corresponds to the first two condition in~\eqref{Phi_prop_sec5}.
 
 %&\Theta(\gamma(s)) = 0, \quad \nabla \Theta(\gamma(s)) = e_1(s), \\ & \mathrm{Im}(\nabla^2 \Theta(\gamma(s))) \geq 0, \quad \mathrm{Im}(\nabla^2 \Theta)(\gamma(s))|_{\dot{\gamma}(s)^{\perp}} > 0.
 
 %More generally, ~\f{Why solving equation to high order, $N$ in this case, is important is that we are going to evaluate exponential integrals of the form $\int r^{N}e^{-r^2}$, which will be smaller when $N$ is higher. This is similar to what we did in proof of Lemma 6 in our earlier paper.}

  To have that $g(d\Theta,d\Theta)=0$ is satisfied to order $2$ on $\Gamma$ is more complicated. For this, one uses the quadratic ansatz
  \[
   \Theta_2(s,y)=y\cdot H(s)\s y,
  \]
  where $H(s)$ is a complex $n\times n$ matrix and $\cdot$ refers to the usual $\R^n$ inner product and $y\in \R^n$. This ansatz leads to the Riccati equation, which is a first order matrix valued ODE. For our purposes, the form of the Riccati equation is not important and we suffice to say that one can find a complex solution $H(s)$ to the equation with $\text{Im}(H(s))>0$. The conditions $\text{Im}(H(s))>0$ and $\Theta_0=0$ together imply the invariantly written conditions
  \[
   \mathrm{Im}(\nabla^2 \Theta(\gamma(s)))\geq 0 \text{ and } \mathrm{Im}(\nabla^2 \Theta)(\gamma(s))|_{\dot{\gamma}(s)^{\perp}} > 0.
  \]
Here we use the notation $\dot{\gamma}(s)^{\perp}$ to denote the algebraic complement to $\dot{\gamma}(s)$ in $T_{\gamma(s)}M$. That is $\mathbb{\R}\s \dot{\gamma}(s) \oplus \dot{\gamma}(s)^{\perp} = T_{\gamma(s)}N$.
%   
%   
%   . The third condition in~\eqref{Phi_prop_sec5} also follows from $\Theta_2(s,y)=y\cdot Y(s)\s y$ when combined with $\Theta_0(s,0)=0$.

  Solving the eikonal equation to order $2$ is enough to understand the qualitative properties of the phase function $\Theta$ needed in our inverse problem. However, we wish to have that 
  \begin{equation}\label{Nth_ord_sol}
   \square_g (e^{is \Theta(x)} a(x))=O_{H^k([0,T]\times \Omega)}(\tau^{-K}). 
  \end{equation}
% 
%   
%   the error of the ansatz $e^{is \Theta(x)} a(x)$ for being a solution to $\square_g v=0$ to be $\mathcal{O}(\tau^{-K})$ in $H^k$, 
  For this, we solve the eikonal equation to an order $N$, which depends on $k$ and $K$. This can be done by solving additional ODEs, but we omit the details. After finding $\Theta$ so that $g(d\Theta,d\Theta)$ vanishes to order $N$ on $\Gamma$, the term $\tau^2g(d\Theta,d\Theta)$ in the expansion~\eqref{box_of_WKB} of $\square_g(e^{i\tau\Theta}a)$ satisfies
  \begin{equation}\label{eikonal_approximation_error}
   \tau^2g(d\Theta,d\Theta)\leq C_0\tau^2|y|^{N+1}.
  \end{equation}
We choose a specific $N$ later. % that depends on $k$ and $K$ specifically later. % to satisfy $-2K=2k+4-n/2-1-N$ for a reason that will become clear later. %The denominator $2$ in $K=(N+1-k)/2-1$ is due to the fact that $H^k$ is $L^2$ based. Note that for higher $k, K$, one needs higher $N$.  %, by solving ODEs. We refer to~\cite{FO19} .
  %~\f{The  Also, the higher the Sobolev index $k$, the higher $N$ we need.} 
  
 Next we insert the phase function $\Theta$ that we have constructed into the transport equations~\eqref{transp_ord_0} and~\eqref{transp_ord_higher} to find an amplitude function $a$. %Using the ansatz~\eqref{ansatz_for_a} leads to a family of $N$ transport equations for the terms $a_{-j}$, $j=1,\ldots, N$. 
 To solve the transport equations, we write
 \begin{equation}\label{def:a_k}
  a_{k}=\chi\Big(\frac{|y|}{\delta'}\Big) b_{k},
 \end{equation}
 so that 
 \[
a=\chi(|y|/\delta')\sum_{k=0}^N \tau^{-k}b_{k}.  
 \]
 Here $\chi\in C_c^\infty(\R)$ is a fixed cutoff function, which is identically $1$ on a neighbourhood of $0\in \R$ and $\delta'>0$ is chosen small enough so that $\chi(|y|/\delta')$ is compactly supported in the domain of the Fermi coordinates. 
 
 We seek for each of the $b_{k}$, $k=1,\ldots, N$,  the form
 \begin{equation}\label{def:v_k}
 b_{k}=\sum_{j=0}^N b_{k,j}(s,y),
 \end{equation}
 where $b_{k,j}(s,y)$ is a complex valued homogeneous polynomial of order $j$ in $y$. We are interested in the specific form only of the leading term $v_{0,0}$. The transport equation concerning $b_0$ is $-2\s g(d\Theta,da_0)+(\square_g \Theta)a_0=0$, which is satisfied to order $0$ if % Since $\chi(r)$ is $1$ on a neighbourhood of $r=0$, we have that $v_{0,0}(s)$ needs to solve
 \[
  -2\s g(d\Theta,db_{0,0})(s,0)+(\square_g \Theta)b_{0,0}(s,0)=0.
 \]
 Here we used that $\chi(\abs{y}/\delta')=1$ to order $1$ at $y=0$. We have $d\Theta (s,0)=dy^{1}$ and $g^{01}(s,0)=-1$. It is calculated in~\cite[Section 4.2]{FO19} that $(\square_g \Theta)(s,0)=\frac{d}{ds}\log \det (Y(s))$, where $Y(s)$ is a one parameter non-degenerate matrix field, which solves an ODE with the initial condition $Y(0)=I_{n\times n}$. Thus we have that the equation for $b_{0,0}(s)$ is solved by 
\begin{equation}\label{eq:Ric_z_1}
 b_{0,0}(s) =\det(Y(s))^{-\frac{1}{2}},
 \end{equation}
 with
 \begin{equation}\label{eq:Ric_z_2}
  b_{0,0}(0)=1.
 \end{equation}
%  
%  
%  Since $d\Theta (s,0)=dy^{1}$, $g^{01}(s,0)=-1$ and by calculation $(\square_g \Theta)(s,0)=\frac{d}{ds}\log \det (Y(s))$, we have that the equation above for $v_{0,0}(s)$ is solved by 
% %  \[
% %   2\frac{d}{ds}v_{0,0}=-2\s g(d\Theta,dv_{0,0})(s,0)= -(\square_g \Theta)v_{0,0}= -\frac{d}{ds}\log \det (Y(s))v_{0,0}.
% %  \]
% %  Thus we may set
% %  
% \[
%  v_{0,0}(s) =\det(Y(s))^{-\frac{1}{2}}.
%  \]
 
 The terms $b_{0,j}$, $j=1,\ldots, N$, are constructed by solving linear ODEs so that $-2\s g(d\Theta,da_0)+(\square_g \Theta)a_0=0$ is satisfied to order $N$. The higher order transport equations~\eqref{transp_ord_higher} concerning $b_{k}$, $k\geq 1$, can be solved recursively to order $N$ by using similar arguments. We omit the details, and only conclude that there is $C_1>0$ so that
 \begin{align}\label{transport_approximation_error}
  \abs{-2i \s g(d\Theta,da_0)+i(\square_g \Theta)a_0}&\leq C_1 |y|^{N+1}, \\
%\label{transp_ord_higher}
  \abs{-2i\s g(d\Theta,da_{k})+i(\square_g \Theta)a_{k}-\square_g a_{k-1}}&\leq  C_1  |y|^{N+1},\nonumber
 \end{align}
 $k=1,\ldots, N$. 
 %Recall from~\eqref{eikonal_approximation_error} that $\tau^2g(d\Theta,d\Theta)\leq C_0\tau^2|y|^{N+1}$. 
 %Finally, we note that 
 Since $a=a_0+\tau^{-1}a_{1}+\tau^{-2}a_{2}+\cdots +\tau^{-N}a_{N}$, we have that 
 \begin{align*}
  &-2i\s \tau\s g(d\Theta,da)+i\tau\s (\square_g \Theta)a+ \square_g a \\
  &= \tau \sum_{k=0}^N\tau^{-k}(-2i\s \s g(d\Theta,da_{k})+i\s (\square_g \Theta)a_{k})+\sum_{k=0}^N\tau^{-k}\square_g a_{k} \\
  &=\tau \sum_{k=1}^N\tau^{-k}(-2i\s g(d\Theta,da_{k})+i\s (\square_g \Theta)a_{k}+\square_g a_{k-1}) \\
  &\quad  + \tau(-2i  \s g(d\Theta,da_0)+i(\square_g \Theta)a_0)+ \tau^{-N}\square_g a_{N} \\
  &= \tau \s O_{L^\infty}(|y|^{N+1}) + O(\tau^{-N}). % \leq C_1 \tau |y|^{N+1} + C_2\tau^{N+1}.
 \end{align*}
 By additionally recalling from~\eqref{eikonal_approximation_error} that $\tau^2g(d\Theta,d\Theta)\leq C_0\tau^2|y|^{N+1}$, we have
%  \[
%   \tau^2g(d\Theta,d\Theta)-2i\s \tau\s g(d\Theta,da)+i\tau\s (\square_g \Theta)a+ \square_g a\leq  C_0\tau^2|y|^{N+1}+C_1\tau |y|^{N+1} + C_2\tau^{-N}. % \leq C 
%  \]
  \begin{align}\label{full_estimate}
  e^{-i\tau \Theta}\square_g (e^{i\tau \Theta} a)&=\tau^2g(d\Theta,d\Theta)-2i\s \tau\s g(d\Theta,da)+i\tau\s (\square_g \Theta)a+ \square_g a \nonumber \\
  &\leq  C_0\tau^2|y|^{N+1}+C_1\tau |y|^{N+1} + C_2\tau^{-N}. % \leq C 
 \end{align}

 By redefining $\delta'>0$ smaller, if necessary, we have that
 \[
  |e^{i\tau \Theta(s,y)}|\leq C e^{-c\tau \abs{y}^2}
 \]
 for $(s,y)$ in the support of $a$. Recall that our aim is to show that 
 \begin{equation}\label{final_estim}
\norm{\square_g (e^{i\tau \Theta(s,y)} a(s,y))}_{H^k([0,T]\times \Omega)}=O(\tau^{-K}).
 \end{equation}
 Taking $k$ derivatives of $\square_g (e^{i\tau \Theta(s,y)} a(s,y))$ %brings at most $k$ powers of $\tau$ to the front of the expression, or reduces the degree of vanishing of $h_j$ on $\Gamma$ by at most $k$. This 
 gives 
	\begin{equation}\label{error_for_derivatives}
	 \abs{\nabla^k \square_g (e^{i\tau \Theta(s,y)} a(s,y))}\leq C_3 e^{-\tau c \abs{y}^2}\sum_{l=0}^k \tau^{k-l} (\tau^2 \abs{y}^{N+1-l}+\tau \abs{y}^{N+1-l}+\tau^{-N}).
	\end{equation}
	%(Here we implicitly used that $\chi(|y|/\delta')$ equals $1$ to infinite order on $\Gamma$.)
% % % % % % Taking $k$ derivatives of $\square_g (e^{i\tau \Theta(s,y)} a(s,y))$ brings at most $k$ powers of $\tau$ to the front expression, or reduces the degree the degree the eikonal and transport equations are solved by at most $k$. This gives 
% % % % % % %  
% % % % % % % by solving $\Theta$ and $v_{j}$ to order $N$ on the geodesic, we obtain that~\f{Katso ellitic paperi.} 
% % % % % %  \begin{equation}\label{derivative_estimate}
% % % % % %   \abs{\nabla^k \square_g (e^{i\tau \Theta(s,y)} a(s,y))}\leq \tau^k|e^{i\tau \Theta}|(C_{0,\alpha}\tau^2|y|^{N+1}+C_{1,\alpha}\tau |y|^{N+1}+C_{2,\alpha}\tau^{-N}).
% % % % % %  \end{equation}
% % % % % %  %where $\alpha$ is any multi-index of $0,1,\ldots, n$ with $\abs{\alpha}\leq N$. (The index $0$ corresponds to $s$-variable.)
% % % % % %  
% % % % % %  Recall that our aim is to show that 
% % % % % %  \begin{equation}\label{final_estim}
% % % % % % \norm{\square_g (e^{i\tau \Theta(s,y)} a(s,y))}_{H^k(M)}=\mathcal{O}(\tau^{-K}).  
% % % % % %  \end{equation}
% % % % % %  For this, we calculate the $L^2$ norms of~\eqref{derivative_estimate}. We have (by possibly redefining $\delta'>0$ smaller) that
% % % % % %  \[
% % % % % %   |e^{i\tau \Theta}|\leq C e^{-c\tau \abs{y}^2}
% % % % % %  \]
% % % % % %  for $(s,y)$ in the support of $a$. 
 %We calculate the $L^2(M)$ norm estimate ~\eqref{error_for_derivatives} 
 We calculate the integral of~\eqref{error_for_derivatives} squared using polar coordinates for the $y$-variable and the standard formula 
%  
%  By using $K=(N+1-k)/2-1$, polar coordinates for $y$-variable and 
%  
 $\int_0^\infty r^le^{-\tau cr^2}dr\sim \tau^{-\frac{l+1}{2}}$ for $l\geq 0$. Note that since the lightlike geodesic $\gamma$ of $(N,g)$ is causal, $[0,T]\times \Omega$ compact and $(N,g)$ globally hyperbolic, $\gamma=\gamma(s)$ will lie in $[0,T]\times \Omega$ only for a finite parameter values. Thus the integration in the coordinate $s$ will be over a finite interval. We obtain the estimate
 %by calculating the integral of~\eqref{error_for_derivatives} squared that
 \begin{align*}
  \norm{\square_g &(e^{i\tau \Theta(s,y)} a(s,y))}_{H^k(M)}^2\lesssim\sum_{l=0}^k \tau^{2(k-l)}\Big( \int_{0}^re^{-2\tau c r^2}r^{n-1}(\tau^4r^{2N+2-2l}dr +\tau^{-2N})\Big) \\
  &\lesssim \sum_{l=0}^k \tau^{2(k-l)}(\tau^4\tau^{-(n+2N+2-2l)/2}+\tau^{-2N-n/2})\lesssim \tau^{2k+4-(n+2+2N)/2}=\tau^{2k+3-n/2-N}
 \end{align*}
 for $\tau$ and $N$ large enough. (Here we have relaxed the notation and denoted by $A \lesssim B$ if there is a constant $\tilde{C}$ independent of $\tau$ such that $A\leq \tilde{C} B$.) 
 If $p>1$, we may $L^p$-normalize $e^{i\tau \Theta}a$ so that 
 \[
  \int_M |\tau^{\frac{n}{2p}}e^{i\tau \Theta}a|^p\lesssim \tau^{\frac{n}{2}}\int_0^\infty r^{n-1}e^{-\tau c r^2}\lesssim 1,
 \]
  in which case we also have 
 \[
  \int_M \abs{\nabla^l (\tau^{\frac{n}{2p}}e^{i\tau \Theta}a)}^2\lesssim \tau^{\frac{n}{p}}\tau^{2l}\tau^{-\frac{n}{2}}.%=C''\tau^{\frac{n}{p}}\tau^{2l}\tau^{-\frac{n}{2}}
 \]
 Therefore, if we define $N=N(n,k,K,p)$ so that it satisfies
 \begin{equation}\label{def:k_K}
  -2K=2k+3-n/2-N+n/p,
 \end{equation}
 we have~\eqref{final_estim}. (If $N$ above is not an integer, we redefine it as $\lfloor N+1\rfloor$.) 
 
By collecting the details of the construction and by defining
\[
 v_\tau=\tau^{\frac{n}{2p}}e^{i\tau \Theta}a
\]
we have:

\begin{proposition}[Gaussian beams]\label{Gaussian_beam_construction}
 Let $(N,g)$ be a globally hyperbolic Lorentzian manifold, $N=\R\times M$ and $\dim(N)=n+1\geq 3$. Let $\Omega$ be a compact submanifold of $M$ with boundary, $\dim(\Omega)=n$. Let $T>0$ and let $\gamma$ be a lightlike geodesic of $(N,g)$.
% ~\f{I think we need to allow also intersection of top or bottom of the timecylinder in the inverse problem. Add later, it does not change the construction of Gaussian beams.} 
 Let 
 %$\lambda\in \mathbb{R}$ and 
 $k,\s K, \s l \in \N$ and $p\geq 2$. There is $\tau_0\geq 1$ and a family of functions $(v_\tau)\subset C^\infty([0,T]\times \Omega)$ such that for $\tau\geq \tau_0$
 \begin{equation}\label{Hk_and_L4}
\begin{split}
  \norm{\square_gv_\tau}_{H^k( [0,T]\times\Omega)}&=O(\tau^{-K}), \\
\norm{v_\tau}_{L^p( [0,T]\times\Omega)}&=O(1), \\
\norm {v_\tau}_{H^l( [0,T]\times\Omega)}&=O(\tau^{\frac{n}{2p}-\frac{n}{4}+l}) \\
\end{split}
\end{equation}
 as $\tau \to \infty$. 
 The function $v_\tau$ is called a Gaussian beam and it has the form
	\begin{equation}\label{form_of_gb}
	v_\tau=  \tau^{\frac{n}{2p}} e^{i\tau\s \Theta(x)} a(x),
	\end{equation}
	where $\Theta$ is a smooth complex function (independent of $\tau$) on a neighbourhood of $\gamma([0,L])$ satisfying 
	\begin{align}\label{Phi_prop_sec5}
	\begin{split}
	&\Theta(\gamma(s)) = 0, \quad \nabla \Theta(\gamma(s)) = e_1(s), \\ & \mathrm{Im}(\nabla^2 \Theta(\gamma(s))) \geq 0, \quad \mathrm{Im}(\nabla^2 \Theta)(\gamma(s))|_{\dot{\gamma}(s)^{\perp}} > 0.
	\end{split}
	\end{align}
	Here also $a(\gamma(s)) =  (a_0(\gamma(s)) + O(\tau^{-1}))$, where 
	\begin{equation*}
	 a_0(\gamma(s))=\det(Y(s))^{-1/2}
	\end{equation*}
	is nonvanishing and independent of $\tau$. 
	Here $Y(s)$ is a nondegenerate $n\times n$ matrix valued function. 
% 	Here $Y(s)$ is an $n\times n$ complex matrix and $\textrm{Im}\s(Y(s))>0$ for all $s\in [0,L]$. 
	The support of $a$ can be taken to be in any small neighbourhood $U$ of $\gamma([0,L])$ chosen beforehand. If $s_0\in [0,L]$ we may arrange so that $a_0(\gamma(s_0))=1$. 
	\end{proposition}

The Gaussian beams can be corrected to be exact solutions to $\square_g\s  v = 0$.
%We next correct the Gaussian beams to be exact solutions to $\square v=0$. %See Corollary~\ref{correction_to_solution} below.

\begin{corollary}\label{correction_to_solution}
Let us adopt assumptions and notation of Proposition~\ref{Gaussian_beam_construction}. Assume in addition that the lightlike geodesic $\gamma$ does not intersect $\{t=0\}$. Let also $l',K\in \N$. Then there are Gaussian beams $v_\tau$ satisfying the conditions of Proposition~\ref{Gaussian_beam_construction} and functions $r_\tau\in C^\infty([0,T]\times \Omega)$ such that
 \[
  v:=v_\tau+r_\tau
 \]
 is a solution to 
  \begin{equation}
\begin{cases}\label{equation_for_v}
\square_g v =0, &\text{in } [0,T]\times \Omega,\\
 v = v_\tau, &\text{on } [0,T]\times \p \Omega, \\
v\big|_{t=0} = \partial_t v\big|_{t=0} = 0,&\text{in } \Omega.
\end{cases} 
\end{equation}
 The functions $r_\tau$ satisfy
 \begin{equation}\label{estimate_for_the error_r}
  \norm{r_\tau}_{H^{l'}( [0,T]\times\Omega)}=O(\tau^{-K}).
 \end{equation}
%for $k>{l'}+\frac{n-1}{2}$, when $k$ is chosen large enough in Proposition~\ref{Gaussian_beam_construction}.
\end{corollary}
\begin{proof}
By assumption the graph of $\gamma$ has a neighbourhood $U$, which does not intersect a neighbourhood of $\{t=0\}$. Let $v_\tau$ be Gaussian beams, which are supported in $U$ and satisfy the conditions of Proposition~\ref{Gaussian_beam_construction}. 
 By Proposition~\ref{thm:energy}, there exists a solution to
 \begin{equation}
\begin{cases}
\square_g r_\tau =-\square_g v_\tau, &\text{in } [0,T]\times \Omega,\\
 r_\tau = 0, &\text{on } [0,T]\times \p \Omega, \\
r_\tau\big|_{t=0} = \partial_t r_\tau\big|_{t=0} = 0,&\text{in } \Omega.
\end{cases}
\end{equation}
Then $v=v_\tau+r_\tau$ solves~\eqref{equation_for_v}.
% Here we used the fact that the lightlike geodesic $\gamma$ corresponding to $v_\tau$ and chose a neighbourhood of the graph of $\gamma$  the lateral boundary transverally. It follows that by choosing the amplitude to be supported in a small enough neighbourhood $U$ of $\gamma$, we have that $v_\mu=0$ on a neighbourhood of the initial-data surface $\{t=0\}$. By Proposition~\ref{Gaussian_beam_construction}, we have that $\norm{\square_gv_\mu}_{H^k(\Omega \times [0,T])}=O(\tau^{-K})$ so that the claim follows from the energy estimates~\eqref{}.

By Proposition~\ref{Gaussian_beam_construction} we have that $\norm{\square_g\s  v_\tau}_{H^k([0,T]\times\Omega)} = O(\tau^{-K})$, where $k,K$ can be chosen freely. By Remark \ref{sob:emb_z} for
$k>{l'}-1+\frac{n+1}{2}$ it holds that $H^k([0,T]\times\Omega) \subset E^{{l'}-1}$.
%
%This can be arranged by taking $k>0$ large enough in Proposition~\ref{Gaussian_beam_construction}.
%
Choosing $k>{l'}-1+\frac{n+1}{2}$ in Proposition~\ref{thm:energy} and using \eqref{ineq:sob_emb} shows that 
%$r_\tau$ belongs to $E^{l'}$, which combined with \eqref{ineq:sob_emb} implies
$$
\norm{r_\tau}_{H^{l'}([0,T]\times\Omega)} 
\lesssim \norm{r_\tau}_{E^{l'}}
\lesssim
\norm{\square_g\s v_\tau}_{E^{{l'}-1}} 
\lesssim\norm{\square_g\s  v_\tau}_{H^k([0,T]\times\Omega)}
= O(\tau^{-K})
$$
as claimed.
\end{proof}

\begin{remark}\label{rem:remark_time_inversion}
We shall also need solutions to the wave equation
  \begin{equation}\label{eq:backwards}
\begin{cases}
\square_g v =0, &\text{in } [0,T]\times \Omega,\\
 v = f, &\text{on } [0,T]\times \p \Omega, \\
v\big|_{t=T} = \partial_t v\big|_{t=T} = 0,&\text{in } \Omega
\end{cases}
\end{equation}
where the Cauchy data of $v$ vanishes at the top of the time-cylinder. Solutions to \eqref{eq:backwards} can be found as follows.
Consider the isometry $h$ given by $t\mapsto T-t$
%\begin{align*}
%&h:([0,T]\times M,g) \to ([0,T]\times M,\tilde g),\\
%&h(t,x) = (T-t,x),
%\end{align*}
and let $\tilde g = h^*g$.
%
%One can verify that in coordinates
%$$
%\tilde g = -\beta(T-t,x) \d t^2 + h(T-t,x),
%$$
%where $\beta(T-t,x)>0$ and $h(T-t,\cdot)$ is a smooth one-parameter family of Riemannian metrics on $M$.
%
Let $\tilde f = f(T-t,x)$ and let $\tilde v$ be the unique solution to
\begin{equation*}
\begin{cases}
\square_{\tilde g}\tilde v =0, &\text{in } [0,T]\times \Omega,\\
 \tilde v = \tilde f, &\text{on } [0,T]\times \p \Omega, \\
\tilde v\big|_{t=0} = \partial_t \tilde v\big|_{t=0} = 0,&\text{in } \Omega.
\end{cases}
\end{equation*}
Because wave operator is invariant under isometries we have
$$
h^*(\square_{\tilde g} \tilde v) = \square_g (h^*\tilde v),
$$ 
whence $v(t,x) := (h^*\tilde v)(t,x)=\tilde v(T-t,x)$ solves equation \eqref{eq:backwards}.
\end{remark}

We next vary the initial point and direction of a lightlike geodesic to construct a family of Gaussian beams. The Gaussian beams will be constructed so that the implied constants of the family of Gaussian beams are uniformly bounded. This uniformity of constants is essential when proving stability estimates. We mention here a similar consideration in the Riemannian setting~\cite[Section 4.1]{FKLLS20}. 

To obtain such Gaussian beams, we start with a lemma. We define the set $\PSO(N)$ of pseudo-orthonormal frames as 
 \begin{align*}
  \PSO(N):=\{ (e_0,\ldots,e_{n})\in  (TN)^{n+1}\mid&  \ g(e_0,e_0)=0, \ g(e_1,e_1)=0, \ g(e_0,e_1)=-2 \\ 
  & \quad \qquad g(e_j, e_k)=\delta_{jk},\text{ for } j,k=2,3,\ldots,n \}.
\end{align*}
The lemma especially says that on a neighbourhood of any point of $N$ there is local pseudo-orthonormal frame.
\begin{lemma}\label{section_of_PSO}
 Let $z_0\in N$ and let $V_0\in T_{z_0}N$ be a lightlike vector. The set of pseudo-orthonormal frames admits a local section $E:\mathcal{U}\to \PSO(N)$ such that the first component $(E(z_0))_0$ of $E$ at $z_0$ is $V_0$. Here $\mathcal{U}$ is an open neighbourhood of $z_0$.
\end{lemma}
\begin{proof}
 The existence of a pseudo-orthonormal frame $e=(e_0,e_1,\ldots,e_n)$ of the tangent space $T_{z_0}N$ over the single point $z_0$ with $e_0=V_0$ was shown in~\cite{FO19}. By using local coordinates $(x^k)$ on a neighbourhood $\mathcal{U}\subset M$ of $z_0$ let us define the mapping
\[
 F(x,E):x(\mathcal{U})\times \R^{(n+1)\times (n+1)} \to \R^{(n+1)\times (n+1)}, 
\]
where $x(\mathcal{U})\subset \R^{n+1}$, by the conditions 
% \begin{align*}
%  F(x,E)_{00}&=g(E_0,E_0), \ F(x,E)_{11}=g(E_1,E_1), \\
%  &\quad \quad F(x,E)_{10}=F(x,E)_{01}=g(E_1,E_0)+2, \ F(x,E)_{jk}= g(E_j,E_k)-\delta_{jk}.
% \end{align*}
% \begin{align*}
%  F(x,E)_{00}&=g(E_0,E_0), \quad \ \ \ F(x,E)_{11}=g(E_1,E_1), \\
%  F(x,E)_{10}&=g(E_1,E_0)+2, \ F(x,E)_{01}=g(e_0,E_1)+2 \\
%  &F(x,E)_{kk}= g(E_k,E_k)-1 \quad k=2,\ldots,n \\
%  &F(x,E)_{jk}= g(E_j,E_k) \text{ if } j>k \\
%  &F(x,E)_{jk}= g(e_j,E_k) \text{ if } j<k.
% \end{align*}
\begin{align*}
 F(x,E)_{jk}&=g_x(E_j,E_k)-g_{z_0}(e_j,e_k) \s\text{ if } j\geq k, \\
 F(x,E)_{jk}&=g_x(e_j,E_k)-g_{z_0}(e_j,e_k) \,\,\text{ if } j<k.
\end{align*}
Here $E_j$ is the $j^{\textrm{th}}$ column vector of the $(n+1)\times (n+1)$ matrix $E$. Here also $g_x(E_j,E_k)=\langle E_j,g(x)E_k\rangle$ and $g_x(e_j,E_k)=\langle e_j,g(x)E_k\rangle$, where $g(x)$ is the coordinate representation matrix of $g$ in the coordinates $(x^k)$. The perhaps ad hoc looking conditions for $F(x,E)_{jk}$ for $j<k$ are related to the fact local sections $E$ of $\PSO(M)$ satisfying $(E(z_0))_0=V_0$ (should they exist) are not unique without additional conditions. The conditions for $F(x,E)_{jk}$ for $j<k$ removes this ambiguity.

We apply the implicit function theorem (see e.g.~\cite[Theorem 10.6]{RR06}) to show that there is a smooth mapping $x\mapsto E(x)$ such that $F(x,E(x))=0$. In this case $E$ is a smooth section of $\PSO(N)$ by the conditions for $F(x,E)_{jk}$ for $j\geq k$ and by the symmetry of $g$. To apply the implicit function theorem, note that $F(z_0,e)=0$ and that
\begin{align}
 (D_EF|_{x=z_0, E=e}(v))_{jk}&=g_{z_0}(v_j,e_k)+g_{z_0}(e_j,v_k) \quad \text{ if } j\geq k, \label{jgeqk}\\
 (D_EF|_{x=z_0, E=e}(v))_{jk}&=g_{z_0}(e_j,v_k) \qquad \qquad \qquad\ \, \text{ if } j<k, \label{jleqk}
\end{align}
% 
% 
% \[
%  (D_EF|_{x=z, E=e}(v))^{jk}=g(e_j,v_k),
% \]
where $j,k=0,1,\ldots, n$ and $v=(v_0,v_1,\ldots,v_n)\in \R^{(n+1)\times (n+1)}$. Assume that $(D_EF|_{x=z, E=e}(v))=0$. Since $g$ is symmetric, the condition~\eqref{jleqk} implies that $g_{z_0}(v_j,e_k)=0$ for $j>k$. Substituting this into~\eqref{jgeqk} then implies that $g_{z_0}(e_j,v_k)=0$ for $j\geq k$. Thus we actually have that $g_{z_0}(e_j,v_k)=0$ for all $j,k=0,1,\ldots n$. Since $g$ is non-degenerate and $e$ is a frame, it follows that each $v_k\in \R^{n+1}$ is the $0$ vector of $\R^{n+1}$. Thus $D_EF|_{x=z_0, E=e}$ is injective, and also surjective by dimensionality. Thus, by the implicit function theorem, and by redefining $\mathcal{U}$ smaller if necessary, there is a smooth mapping $E:\mathcal{U} \to \PSO(N)$. This is our desired section. 
\end{proof}
We remark that it is likely that another proof of the above lemma can be obtained by generalizing the Gram-Schmidt procedure to the current situation. We also mention the similar construction~\cite[Lemma 6.1]{FKLLS20} in the Riemannian setting. 

In the next result $|V_0-\dot\gamma_x(s_0)|$ is defined by using local coordinates.

\begin{corollary}\label{uniform_family_of_Gaussian_beams}
 Let $\gamma$ be a lightlike geodesic of $(N,g)$. Assume as in Proposition~\ref{Gaussian_beam_construction} and adopt its notation. Let $s_0$ be in the domain of $\gamma$ and let us denote $\gamma(s_0)=z_0$ and $\dot\gamma(t_0)=V_0$. Let also $\delta>0$. Then there is $\tau_0\geq 1$ and a neighbourhood $U$ of $z_0$ and a family of Gaussian beams 
 \[
  v_\tau(x,\ccdot) %., \quad x\in v_\tau, y\in M \text{ and } \tau \geq\tau_0
 \]
 solving $\square_g v_\tau(x,\ccdot)=0$ in $[0,T]\times \Omega$ (including the correction term) parametrized by $x\in \mathcal{U}$.  Here $`` \ccdot " $ refers to points in $M$ and $\tau \geq \tau_0$. The geodesics $\gamma_x$ corresponding to the Gaussian beams $v_\tau(x,\ccdot)$ satisfy $|V_0-\dot \gamma_x(s_0)|\leq \delta$ and the implied constants of $v_\tau(x,\ccdot)$ in Proposition \ref{Gaussian_beam_construction} and Corollary \ref{correction_to_solution} are uniformly bounded in $x$.
%  
%  mapping $x\mapsto (\gamma_x,v_x)$, where $v_x$ is a Gaussian beam as in Corollary~\ref{} such that all the implied constants are uniformly boudned in $x$ and $\gamma_x$ is a non-tangential geodesic such that $|V_0-\dot \gamma_x(0)|\leq \delta$. 
%  and is a Gaussian beam 
%  with the following property: If $(x,V)\in \mathcal{U}$ and if $\gamma_{(x,V)}$ is a geodesic with the initial data $\gamma_{(x,V)}(0)=x$ and $\dot \gamma_{(x,V)}(0)=V$, then the corresponding Gaussian beam satisfies~\eqref{} and the all the implied constants are uniform in $x$.
% 
\end{corollary}
\begin{proof}
 The proof is based on inspecting the construction of the Gaussian beams at the beginning of this section that lead to Proposition~\ref{Gaussian_beam_construction}, and by using Corollary~\ref{correction_to_solution} and Lemma~\ref{section_of_PSO}.
 
 Let $v_\tau$ be a Gaussian beam without the error term corresponding to the geodesic $\gamma$  as in Proposition~\ref{Gaussian_beam_construction}. Note that this implies that we have chosen initial data for the certain ODEs used in the construction (such as the Riccati equation). Let us record these initial data and also define 
 \[
  v_\tau(z_0,\ccdot):=v_\tau(\ccdot).
 \]
% of the corresponding phase function $\Theta$ and components $a_k$ of the amplitude $a$. 

 By Lemma~\ref{section_of_PSO} there is a local section $E$ of $\PSO(M)$ such that $(E(z_0))_0=V_0$. We define a local vector field $V$ by 
 \[
  V(x)=(E(x))_0.
 \]
 By redefining the domain of $E$ smaller, if necessary, we have that $|V(x)-\dot\gamma(0)|<\delta$. The section $E$ also defines a family of Fermi-coordinates by the formula~\eqref{Fermi_coords_def} parametrized by $x$. Since $E$ is smooth, the corresponding Fermi-coordinates depend smoothly on $x$ (say in any $C^k$-norm in the Frech\'et sense). Also the domain of the Fermi-coordinates is uniformly bounded by the same reason.%~\f{Maybe a short explanation?}
%  
%  Let $v_\tau$ by a Gaussian beam corresponding to $\gamma$ (without the error term) as in Proposition~\ref{}. Recall from the beginning of this section that when we constructed Gaussian beams we chose initial data for certain ODEs such as the Riccati equation. Let us record these initial datum.
%  %used in the construction of the corresponding phase function $\Theta$ and components $a_k$ of the amplitude $a$. 
%  

 Let $x\in \mathcal{U}$ and let us pass to the Fermi-coordinates determined by $E(x)$. We construct a Gaussian beam 
 \[
  v_\tau(x,\ccdot)
 \]
 with the following properties: (a) It corresponds to the geodesic $\gamma_{x,V(x)}$ with inital data $x\in M$ and $V(x)\in T_xM$. (b) It is constructed by the exact same method described in the beginning of this section by using the same initial data for the corresponding ODEs that we used for $v_\tau$. Since the coefficients of the ODEs are determined by the smooth metric $g$ and the initial data are the same as for $v_\tau$, the Gaussian beam $v_\tau(x, \ccdot)$ differ boundedly and uniformly in $x$ from $v_\tau(\ccdot)$ (say in any $C^k(M)$ norm.) Especially the implied constants in Proposition \ref{Gaussian_beam_construction} are uniform in $x$.
 
 Finally, we use Corollary~\ref{correction_to_solution} to find correction terms for $v_\tau(x, \ccdot)$ such that the implied constants in~\eqref{estimate_for_the error_r} are uniform in $x$. This concludes the proof. 
\end{proof}

%The paper~\cite{HUZ20} uses four geodesics and their corresponding Gaussian beams. It is even more simple to use a pair of geodesic as we discussed: 
\section{Separation of points}\label{sec:separation_of_points}
In this section $(N,g)$ is a globally hyperbolic smooth Lorentzian manifold without boundary.
The length of a piecewise smooth causal path $\alpha : [a,b] \rightarrow N$ is defined as%~\f{Maybe move to intro?} %be defined by 
\begin{equation}\label{length_def}
l(\alpha) := \sum_{j=0}^{k-1} \int_{a_j}^{a_{j+1}} \sqrt{ -g( \dot\alpha(s), \dot\alpha(s)) } ds,
\end{equation}
where $a_0<a_1< \cdots<a_{k-1} < a_k$ are chosen such that $\alpha$ is smooth on each interval $(a_j,a_{j+1})$ for $j=0,\dots,k-1$. The time separation function, see e.g.~\cite{One83}, is denoted by $\tau : N \times N \rightarrow [0, \infty)$ and defined as 
% \[
% \tau(x,y) :=\begin{cases}  \sup \int_0^1 \sqrt{ -g( \dot\alpha (s) , \dot\alpha (s))  } ds, & x<y \\
% 0, & \text{otherwise.}
% \end{cases}
% \] 
\begin{equation}\label{def:time_separation}
\tau(x,y) :=\begin{cases}  \sup l(\alpha) , &  y\in J^+(x) \\
0, & y\notin J^+(x),
\end{cases}
\end{equation}
where the supremum is taken over all piecewise smooth future-directed causal curves $\alpha : [0,1] \rightarrow N$ that satisfy $\alpha(0)  = x$ and $\alpha(1) = y$. By~\cite[Ch. 14, Lemma 16]{One83}, we have that
\begin{equation}\label{timeseparation}
 \tau(x,z)>0 \text{ if and only if } x\ll z.
\end{equation}

%where $a_0<a_1< \cdots<a_{m-1} < a_m$ are chosen such that 

As before, we view $N$ as the product manifold $\R\times M$ and assume that $\Omega\subset M$, $\dim(\Omega)=\dim(M)$, is a smooth compact manifold with boundary. For $T>0$ we let $\Sigma^T=\Sigma$ denote the lateral boundary of $[0,T]\times \Omega$. Let us consider $x\in I^+(\Sigma)\cap I^-(\Sigma)$.
We say that $\gamma_1:[0,1]\to [0,T]\times \Omega$ is a future-directed \emph{optimal geodesic} connecting $\Sigma$ to $x$ if there is 
\[
  z_1\in J^-(x)\cap \Sigma \text{ such that } \gamma_1(0)=z_1, \ \gamma_1(1)=x \text{ and } \tau(z_1,x)=0.
\]
Similarly, we say that $\gamma_2:[0,1]\to [0,T]\times \Omega$ is a past-directed optimal geodesic connecting $\Sigma$ to $x$ if there is
\[
   z_2\in J^+(x)\cap \Sigma \text{ such that } \gamma_2(0)=z_2, \ \gamma_2(1)=x \text{ and } \tau(x,z_2)=0.
\]
We always understand optimal geodesics as their maximal extensions. Note that by definition future/past-directed optimal geodesics are always lightlike. The next lemma says that such optimal geodesics always exist. We assume the notations and assumptions used earlier in this Section~\ref{sec:separation_of_points}. The situation of the lemma is illustrated in Figure \ref{pic:Special_optimal_geodesics}, which can be found from Section \ref{sec:Lorentzian_geom_tools} in the Introduction.

In the lemma we consider intersection times of geodesics and $\Sigma$. This means that if the geodesic is denoted by $\gamma:[0,1]\to [0,T]\times \Omega$, then the first intersection time is the smallest $s_0\in [0,1]$ such that $\gamma(s_0)\in \Sigma$. Typically $s_0$ will be $0$. That the intersection in the lemma is transverse means that $\dot\gamma(s_0)$ is transversal to the tangent space $T_{\gamma(s_0)}\Sigma$. We do not claim anything about possible other intersections of $\gamma$ and $\Sigma$.

\begin{lemma}[Boundary optimal geodesics]\label{optimal_geo}
 Let $(N,g)$ be globally hyperbolic, $N=\R\times M$. If $x\in I^+(\Sigma)\cap ([0,T]\times \Omega)$, there exists a future-directed optimal geodesic $\gamma:[0,1]\to[0,T]\times \Omega$ from $\Sigma$ to $x$ 
 and the first intersection of $\gamma$ and $\Sigma$ is transverse.
 %that intersects $\Sigma$ for the first time and the intersection is transverse.
Similarly, if $x\in I^-(\Sigma)\cap ([0,T]\times \Omega)$, there exists a past-directed optimal geodesic $\gamma:[0,1]\to[0,T]\times \Omega$ from $\Sigma$ to $x$ 
%that intersects $\Sigma$ for the first time and the intersection is transverse.  
and the first intersection of $\gamma$ and $\Sigma$ is transverse.
\end{lemma}
\begin{proof}
\textbf{Existence.}
 Let us first consider the claim about the existence of future-directed optimal geodesic.  For this, let us define
\begin{equation}\label{tsup}
 t_{\textrm{sup}}=\sup\s\{\tilde{t}\in [0,T]\mid \text{ there is } \tilde{z}\in \Sigma \text{ such that } t(\tilde z)=\tilde{t} \text{ and } \tau(\tilde z,x)>0\}.
\end{equation}
Here $\tau$ is defined on $N\times N$. 
The number $t_{\textrm{sup}}$ will be the time coordinate of $z_{\textrm{sup}}$ in  Figure~\ref{pic:Special_optimal_geodesics}.
By assumption $x\in I^+(\Sigma)$ and thus there is $\tilde z\in \Sigma$ such that $x\in I^+(\tilde z)$ with $\tau(\tilde z,x)>0$ by~\eqref{timeseparation}. We also have $t(\tilde z)\in [0,T]$. Consequently the supremum in~\eqref{tsup} exists and $t_{\textrm{sup}}\in [0,T]$. Let $z_k\in \Sigma$ and $t(z_k)=t_k$ be such that $t_k\to t_{\textrm{sup}}$ as $k\to \infty$. Since $z_k\in \Sigma$ and $\Sigma$ is compact, we may pass to a subsequence so that $z_k\to z_{\textrm{sup}}\in \Sigma$. We also have $t(z_{\textrm{sup}})=t_{\textrm{sup}}$ by continuity of the time function $t$.

% Since $\tau(z_k,x)>0$, we have $z_k\in I^-(x)$ by~\eqref{timeseparation}. Since $I^-(x)\subset J^-(x)$, $\Sigma$ is compact and $J^-(x)$ is closed~\f{Add reference}, we may pass to a subsequence so that $z_k\to z_{\textrm{sup}}\in J^-(x)\cap \Sigma$. 

%Similarly, there is a sequences $t_k'\to t_{\textrm{inf}}$ and $z_k'\to z_{\textrm{inf}}\in J^+(x)\cap \Sigma$ with $t(z_k')=t_k'$. 

We claim that $\tau(z_{\textrm{sup}},x)=0$. %there is a future-directed lightlike geodesic $\gamma_1$ from $z_{\textrm{sup}}$ to $x$ of length $0$ and $\tau(z_{\textrm{sup}},x)=0$. 
We argue by contradiction and assume the opposite that $\tau(z_{\textrm{sup}},x)>0$. Then there is a timelike future-directed path $\eta:[0,1]\to N$ connecting $z_{\textrm{sup}}$ to $x$  by~\eqref{timeseparation}. 
%\tbl{Note that $\eta$ is actually a path $[0,1]\to [0,T]\times \Omega$. Indeed, if $\eta$ meets the complement of $[0,T]\times \Omega$, then $\eta$ intersects $\Sigma$ at a parameter time $s_0<1$ before it meets $z_{\textrm{sup}}$ at $s=1$. Since, $\eta$ is timelike, it follows that $\eta(s_0)\in \Sigma$, $t(\eta(s_0))>t_{\textrm{sup}}$ and $\tau(\eta(s_0),x)>0$. This contradicts the definition of $t_{\textrm{sup}}$.}  
%It can be that $\eta$ intersects $\Sigma$ before it meets $z_{\textrm{sup}}$. However, since 
Since $\eta$ is timelike and $I^-(x)$ is open, we may deform $\eta$ slightly on a neighbourhood of $z_{\textrm{sup}}$ to a future-directed timelike path that connects $z'\in \Sigma$ to $x$ so that $t(z')>t_{\textrm{sup}}$. Thus $x\in I^+(z')$ and we still have $\tau(z',x)>0$ by~\eqref{timeseparation}. This is a contradiction to the definition of $t_{\textrm{sup}}$. We conclude that $\tau(z,x)=0$. Since $(N,g)$ is globally hyperbolic, there is a future-directed lightlike geodesic $\gamma_1:[0,1]\to N$ from $z_{\textrm{sup}}$ to $x$ of length $\tau(z_{\textrm{sup}},x)=0$, see~\cite[Ch. 14, Prop. 19]{One83}. 

We note that $\gamma_1$ is actually a path $[0,1]\to [0,T]\times \Omega$. Indeed, if $\gamma_1$ meets the complement of $[0,T]\times \Omega$, then $\gamma_1$ necessarily intersects $\Sigma$ at a parameter time $s_0<1$ before it meets $z_{\textrm{sup}}$ at the parameter time $1$. Since $\gamma_1$ is causal, it follows that $t(\gamma_1(s_0))>t_{\textrm{sup}}=t(z_{\textrm{sup}})$, where  $\gamma_1(s_0)\in \Sigma$. Since $\Sigma$ is timelike, there is point $\hat z\in \Sigma$ with $t_{\textrm{sup}}<t(\hat z)<t(\gamma_1(s_0))$ and a future-directed timelike path $\hat\eta$ connecting $\hat z$ to $\gamma_1(s_0)$. Thus, a path achieved by composing the paths $\hat \eta$ and $\gamma_1$ has positive length by the definition \eqref{length_def}. It follows that $\tau(\hat z,x)>0$ by the definition \eqref{def:time_separation}.   %We may reparametrize $\gamma_1$ and deform it slightly to have a causal path $\hat \gamma_1$, which connects $\hat \gamma_1(0)\in \Sigma$ to $x=\hat \gamma_1(1)$.
%
%But then $\hat \gamma_1$ has positive length and also satisfies $t(\hat\gamma_1(0))>t_{\textrm{sup}}$. It follows that $\tau(\hat\gamma_1(0),x)>0$. 
We have arrived to a contradiction with the definition of $z_{\textrm{sup}}$, since  $t(\hat z)>t_{\textrm{sup}}$.
%We also have that $\tau(\gamma(s_0),x)=0$, otherwise there would be a shortcut path from $x$ to $z_{\textrm{sup}}$ with positive length contradicting $\tau(z_{\textrm{sup}},x)=0$. 

\textbf{Transversality:}
We next show that the optimal geodesic $\gamma$ constructed above intersects the lateral boundary $\Sigma$ transversally.
Assume that $\gamma$ is parametrised so that $\gamma(0)=z_\mathrm{sup}$.
Let $S_{t_\mathrm{sup}} = \{t_\mathrm{sup}\}\times M$ be the Cauchy level surface at $t=t_\mathrm{sup}$.
%
%Let $\{T,\nu\}$, where 
Let $T=(T_1,\ldots,T_{n-1})$ be a basis for the tangent space $T_{z_\mathrm{sup}}\p\Omega$. Then $\{T,\nu\}$, where  $\nu$ is the normal vector to $\p\Omega$ at $z_\mathrm{sup}$ in $S_{t_\mathrm{sup}}$, is a basis for $T_{z_\mathrm{sup}} S_{t_\mathrm{sup}}$.
%, such that $T_l\in T_{z_\mathrm{sup}}\p\Omega$, $l=1,\ldots,n-1$. Let us also denote by $\nu$ the normal vector to $\p\Omega$ at $z_\mathrm{sup}$ in $S_{t_\mathrm{sup}}$.
%
Consequently, the tangent space $T_{z_\mathrm{sup}} N$ is spanned by $\{\p_t,T,\nu\}$, where $\p_t$ is the coordinate vector of $[0,T]$.
%
%Let 
%\[
% x=(x^t,x^T,x^\nu)
%\]
%be coordinates on a neighbourhood of $z_{\textrm{sup}}$ such that: a) $x^t$ is a coordinate for $[0,T]$ and $x^t(z_{\textrm{sup}})=0$ b) $x^\nu=0$ if and only if $x\in \Sigma$, c) $\nu$ is orthogonal to $\Sigma_{t_0}=\{t_0\}\times \p \Omega$ with respect to $h(t_{\textrm{sup}},\ccdot)$ d) $g(0,0,0)=\eta$. Such coordinates may be achieved by using boundary normal coordinates on $S_{t_{\textrm{sup}}}$ and by doing a linear change of variables in the $t$-direction.~\f{Check this carefully, or use less restrictive coordinate
Let us write $\dot\gamma(0) \in T_{z_\mathrm{sup}} N$ in the form
\[
 \dot\gamma(0)=(\dot\gamma^t(0),\dot\gamma^T(0),\dot\gamma^\nu(0)).
\]
%Note that $\gamma$ can be tangential to $\Sigma$ only if $\dot\gamma^T(0)\neq 0$ since $T_{z_\textrm{sup}}\Sigma=\text{span}(\p_t,T)$.\f{This needs additionally, that $\gamma$ is lightlike.}

%If $\gamma$ intersects $\Sigma$ transversally at $z_\mathrm{sup}$, then in local coordinates it must have a component $0\neq \dot\gamma^\nu(0)\in T_{z_\mathrm{sup}} N$.
%
Suppose now to the contrary that $\gamma$ does not intersect $\Sigma$ transversally.
Then it follows that $\dot\gamma^\nu(0)=0$. Indeed, if this is not the case, then $T_{z_\mathrm{sup}} \Sigma + T_{z_\mathrm{sup}} \mathrm{graph}(\gamma)$ would be equal to $T_{z_\mathrm{sup}}N$.
Let us check whether $\dot\gamma(0)$ is normal to $\Sigma_{t_\mathrm{sup}}:=\Sigma\cap \{t=t_\mathrm{sup}\}$.
%belongs to $N_{z_\mathrm{sup}}\Sigma_{t_\mathrm{sup}}$.
%
%Since $\Sigma_{t_\mathrm{sup}} = \Sigma \cap S_{t_\mathrm{sup}}$ and the intersection of $\Sigma$ and $S_{t_\mathrm{sup}}$ is transversal,
Since $\Sigma_{t_\mathrm{sup}}$ is space-like, the normal space 
$$
N_{z_\mathrm{sup}}\Sigma_{t_\mathrm{sup}} := \{ v\in T_{z_\mathrm{sup}}N \mid \langle v,w\rangle_g = 0\text{ for all } w\in T_{z_\mathrm{sup}}\Sigma_{t_\mathrm{sup}} \}
$$
(see \cite[p.98 or p.198]{One83}) is spanned by $\p_t$ and $\nu$.
To see this, note that a vector $X\in T_{z_\mathrm{sup}}N$, $X = a\p_t + b\cdot T + c\nu$, is in $N_{z_\mathrm{sup}}\Sigma_{t_\mathrm{sup}}$ if and only if $b\in\R^{n-1}$ is zero.
Because $\dot\gamma^\nu(0)=0$, then if $\dot\gamma(0)\in N_{z_\mathrm{sup}}\Sigma_{t_\mathrm{sup}}$, we must have $\dot\gamma(0) = (\dot\gamma^t(0),0,0)$.
But this is not possible, since $\gamma$ is light-like.
So $\dot\gamma(0)$ is not normal to $\Sigma_{t_\mathrm{sup}}$ and by \cite[Ch. 10, Lemma 50]{One83} there exists a time-like curve $\sigma$ from $x$ to $\Sigma_{t_\mathrm{sup}}$.
By slightly deforming $\sigma$ we obtain another time-like curve $\tilde\sigma$ connecting $x$ to $z'\in\Sigma$ with $t(z')>t_\mathrm{sup}$.
This contradicts the definition of $t_\mathrm{sup}$.

The claim about the past-directed optimal geodesic follows by defining 
\[
 t_{\textrm{inf}}=\inf\s\{\tilde{t}\in [0,T]\mid \text{ there is } \tilde{z}\in \Sigma \text{ such that } t(\tilde z)=\tilde{t} \text{ and } \tau(x,\tilde z)>0\}.
\]
and by using arguments analogous to the ones above to find $z_{\textrm{inf}}\in \Sigma$ with $\tau(x,z_{\textrm{inf}})=0$.
\end{proof}

By using boundary optimal geodesics and related Gaussian beams we may separate points of $[0,T]\times \Omega$ by solutions to $\square_g v=0$. We mention here that separation of points by solutions has been beneficial in the study of inverse problems for elliptic equations \cite{LLS19, GST19}.

\begin{proposition}[Separation of points]\label{prop:separation_of_points}

Let $(N,g)$ be globally hyperbolic, $N= \R\times M$. Let $x\in I^-(\Sigma)\cap ([0,T]\times \Omega)$ and $y\in N$ be such that
% $y$ is not in the past of $x$, 
 $y\notin J^-(x)$.
 Denote by $v_f$ the solution to $\square_g v=0$ in $N$ with $v|_{\Sigma}=f$ and whose Cauchy data vanishes at $t=T$. Then there %is a dense subset of points $x\in U$ such that there 
 is $f\in C^\infty(\Sigma)$ such that
 \begin{equation}\label{separation_of_points}
  v_f(x)\neq v_f(y).
 \end{equation}
%If the boundary of $M$ is convex, the dense subset of $U$ can be taken to be $U$ itself. 
If $x\in I^+(\Sigma)\cap ([0,T]\times \Omega)$ and $x\notin J^-(y)$, we have the same claim for solutions of $\square_g v=0$ in $N$ with $v|_{\Sigma}=f$ whose Cauchy data instead vanishes at $t=0$.
% If instead $x>y$, or if additionally $y\in U$, we alternatively have the claim for a solution whose Cauchy data vanishes at $t=0$.
\end{proposition}
\begin{proof}
%Let $\gamma_1$ and $\gamma_2$ be future- and past-directed optimal geodesics connecting $\Sigma$ to $x$ respectively. Let $B_x$ be a neighbourhood of $x$. Let $\eta_1$ and $\eta_2$ be future- and past-directed optimal geodesics connecting $\Sigma$ to $y$ respectively. Let $B_y$ be a neighbourhood of $y$. 

We first claim that there is a past-directed lightlike geodesic from $\Sigma$ that meets the point $x$ but not $y$. We argue by contradiction and assume the opposite that all past-directed lightlike geodesics from $\Sigma$ to $x$ meet both $x$ and $y$. 
Since $x\in I^-(\Sigma)\cap ([0,T]\times \Omega)$, by Lemma~\ref{optimal_geo} we have that there is a past-directed boundary optimal geodesic $\gamma_1:[0,1]\to [0,T]\times \Omega$ with $\gamma_1(0)=z\in\Sigma$ and $\gamma_1(1)=x$. The first intersection of $\gamma_1$ with $\Sigma$ is transverse. If $x\notin J^-(y)$, then by the assumption $y\notin J^-(x)$ we have that $x$ and $y$ are not causally connected. Thus $\gamma_1$ can not pass trough $y$ and we have found our lightlike geodesic. Therefore, we may assume that $y\geq x$.

%Since $\Sigma$ is convex, we have that $\gamma_1$ is non-tangential to $\Sigma$.~\f{Reference?} 
% given by Lemma~\ref{optimal_geo}. 
%Let $B_x$ be a neighbourhood of $x$. Let $\eta_1$ and $\eta_2$ be future- and past-directed optimal geodesics connecting $\Sigma$ to $y$ respectively. Let $B_y$ be a neighbourhood of $y$. 
Let $\tilde \gamma_1$ be a past-directed geodesic with $\tilde \gamma_1(0)\in \Sigma$, such that $\tilde \gamma_1=\tilde \gamma_1(s)$ intersects $\Sigma$ transversally at $s=0$, and which satisfies $\tilde \gamma_1(\tilde s)=x$ for some $\tilde s \geq 0$. The geodesic $\tilde \gamma_1$ can be obtained by perturbing the tangent vector of $\gamma_1$ at $\gamma_1(1)=x$ slightly. By assumption $\tilde \gamma_1$ meets $y$. In this case we have a shortcut path, which has timelike portion, obtained by traveling along $\tilde\gamma_1$ from $x$ to a point $y'$ close to $y$, doing a shortcut from $y'$ to $\gamma_1$ and then by continuing along $\gamma_1$ to $z$, see~\cite[Ch. 10, Prop. 46]{One83}. Since the shortcut path has timelike portion, it has positive length. Since $y\geq x$, the shortcut path is also future-directed. It follows that $\tau(x,z)>0$. This contradicts optimality of $\gamma_1$. We conclude that $\tilde\gamma_1$ is a past-directed lightlike geodesic from $\Sigma$ that meets $x$ but not $y$.

To conclude the proof, we use Proposition~\ref{Gaussian_beam_construction} and choose a Gaussian beam $v_\tau=e^{i\s \tau\Theta}a$ corresponding to $\tilde\gamma_1$ with $k>n$, $K=1$ and $p,l=2$. We also choose the support of the amplitude $a$ be so small that $y\notin \supp(a)$ and $\supp(a)\cap \{t=T\}=\emptyset$. 
We will use the Sobolev embedding $H^{l'}\subset L^\infty$, which hold for ${l'}>\frac{n+1}{2}$.
Since $k>n$, then $k-\frac{n-1}{2}>\frac{n+1}{2}$ which shows that we can take ${l'}$ such that $\frac{n+1}{2}<{l'} < k-\frac{n-1}{2}$.
Applying Corollary~\ref{correction_to_solution} with these $k$ and ${l'}$ shows that there is  $r_\tau\in C^\infty(N)$ such that
\[
v_f:=\tau^{-n/4} v=\tau^{-n/4}v_\tau +\tau^{-n/4}r_\tau
\]
satisfies $\square_g v_f=0$ and%~\f{The number 1 will be replaced by $\det(Y(s))^{1/2}$. Change later once the choice of initial data for the Riccati equation is specified.}
\[
 \tau^{-n/4}v_\tau(x)=1 \text{ and } \tau^{-n/4}v_\tau(y)=0 \text{ for all } \tau\geq \tau_0 
\]
and 
\[
 \norm{\tau^{-n/4}r_\tau}_{L^\infty(N)}\leq C \tau^{-n/4}\norm{r_\tau}_{H^{l'}(N)}=\tau^{-n/4}O(\tau^{-1}).
\]
We mention for future reference, that at any other point of $z\in [0,T]\times \Omega$, we have 
\begin{equation}\label{eq:other_points}
 \abs{v_f(z)}\leq \abs{\tau^{-n/4}v_\tau(z)}+ \abs{\tau^{-n/4}r_\tau(z)}\leq C' +\abs{\tau^{-n/4}r_\tau(z)}\leq C,
\end{equation}
for all $\tau$ large enough.
Here we used the Sobolev embedding. 
Taking $\tau$ large enough shows that 
\[
 v_f(x)\neq v_f(y). 
\]
The claim about the case $x\in I^+(\Sigma)$ and $x\notin J^-(y)$  is proven similarly.
\end{proof}
We next consider the case where we have multiple points of $[0,T]\times \Omega$, which we wish to separate by solutions of $\square_gv =0$. The points will correspond to the intersection points of pairs of geodesics we use for our inverse problem. The matrix \eqref{separation_of_multiple_points_matrix} is called a \emph{separation matrix}.

% \begin{figure}[ht!]
%  \centering
%   \includegraphics[scale=0.25]{Geodesics3.pdf}
%  \caption{Past-directed lightlike geodesics (red dashed lines) that separate the intersection points $z_1$ and $z_2$ of future-directed lightlike geodesics (black). The geodesics in red and and black intersect $\Sigma$ at a time $t<T$ and $t>0$ respectively. \label{pic:Geodesics}}
%  \end{figure}

\begin{lemma}[Separation matrix]\label{lemma:separation_of_multiple_points}
Assume that $(N,g)$ is as in Proposition~\ref{prop:separation_of_points}.
Let $x_1,\ldots,x_P\in I^-(\Sigma)\cap ([0,T]\times \Omega)$ be such that $x_1 < x_2< \cdots < x_P$. 
 Denote by $v_f$ the solution of $\square_g v=0$ in $[0,T]\times \Omega$ with $v|_{\Sigma}=f$ and whose Cauchy data vanishes at $t=T$. Then there %is a dense subset of points $x\in U$ such that there 
 are boundary values $f_k\in C^\infty(\Sigma)$ such that the matrix
 \begin{equation}\label{separation_of_multiple_points_matrix}
 \begin{pmatrix}
        v_{f_1}(x_1) & v_{f_2}(x_1) &\cdots & v_{f_P}(x_1) \\
        v_{f_1}(x_2) & v_{f_2}(x_2) &\cdots & v_{f_P}(x_2) \\
        \vdots & & \ddots & \vdots \\
        v_{f_1}(x_P) & v_{f_2}(x_P) &\cdots & v_{f_P}(x_P) \\
    \end{pmatrix}
 \end{equation}
 is invertible.
 
 If $x_k\in I^+(\Sigma)\cap ([0,T]\times \Omega)$, we have the similar claim for solutions of $\square_g v=0$ in $[0,T]\times \Omega$ with $v|_{\Sigma}=f$ whose Cauchy data instead vanishes at $t=0$.
\end{lemma}
\begin{proof}
 The proof is an iteration of the proof Proposition~\ref{prop:separation_of_points}. First we let $\gamma_1$ be a past-directed boundary optimal geodesic that connects a point $z\in \Sigma$ to the point $x_1$. By the shortcut argument in the proof of Proposition~\ref{prop:separation_of_points}, we deduce after possibly redefining $\gamma_1$ as its small perturbation, that $\gamma_1$ does not meet any of the other points $x_k$, $k=2,\ldots, P$. Let $v_{f_1}$ be a Gaussian beam solution (including the correction term) as in proof Proposition~\ref{prop:separation_of_points}, where $f_1\in C^\infty(\Sigma)$. Then there is $\tau_1>0$ such that for $\tau\geq \tau_1$, we have%~\f{The number 1 will be replaced by $\det(Y(s))^{1/2}$. Change later.}
 \[
  v_{f_1}(x_1)=1 \text{ and } v_{f_1}(x_k)=O(\tau^{-1-n/4}), \ k=2,\ldots, P.
 \]

 Next, let $\gamma_2$ be a past-directed boundary optimal geodesic that connects $z\in \Sigma$ to the point $x_2$. By repeating the above argument we find a boundary value $f_2\in C^\infty(\Sigma)$ and a solution $v_{f_2}$ such that
 \[
  v_{f_2}(x_2)=1 \text{ and } v_{f_2}(x_k)=O(\tau^{-1-n/4}), \ k=3,\ldots, P
 \]
 for all $\tau\geq \tau_2$.
 Note that we do not claim that we have much control on the value of $v_{f_2}$ at $x_1$ and it might be that $\gamma_2$ meets also the point $x_1$. However, by \eqref{eq:other_points} we know that $|v_{f_2}|$ at $x_1$ is bounded by $C$ (possibly by defining $\tau_2$ larger). This is illustrated in Figure~\ref{pic:Geodesics}, which can be found from the introduction, Section \ref{sec:Lorentzian_geom_tools}.
 By repeating the above arguments, we find other solutions $v_{f_k}$, $k=3,\ldots, P$ such that the matrix~\eqref{separation_of_multiple_points_matrix} becomes of the form
 \begin{equation}\label{invertible_matrix_z}
\mathcal{V}_\tau= \begin{pmatrix}
        1 & O(\tau^{-1-n/4}) & O(\tau^{-1-n/4}) \\
        \# &\ddots & O(\tau^{-1-n/4}) \\
        \#& \# & 1  \\
        \end{pmatrix}.
 \end{equation}
 Here $\#$ means unspecified complex numbers bounded by some fixed constant. The determinant of this matrix tends to $1$ as $\tau\to \infty$. Therefore, there is $\tau_0\geq 1$ such that the matrix~\eqref{separation_of_multiple_points_matrix} is invertible for all $\tau\geq \tau_0$. 
\end{proof}

The previous lemma shows that if we are given a set of points $x_1< \cdots< x_k$ one can find a set of Gaussian beams separating these points.
However, for the proof of the stability estimate in Theorem \ref{thm:stability}, we need a finite collection
\[
 \mathcal{M}
\]
of boundary values corresponding to Gaussian beams that suffice to separate any finite set of points.
The next lemma shows how to construct such collection. The finite collection $\mathcal{M}$ is called a \emph{separation filter}.

Let $\overline{g}$ be an auxiliary Riemannian metric on $\R\times M$.

\begin{lemma}[Separation filter]\label{lem:separation_filter}
Let $P\geq 1$ be an integer and let $\delta >0$. Suppose $K\subset I^-(\Sigma)\cap I^+(\Sigma)\cap([0,T]\times \Omega)$ is a compact set. There exists a finite collection $\M\subset C^\infty(\Sigma)$ of boundary values with the following properties: Assume that $x_1,\ldots, x_P\in K$ are any points such that $x_1<  x_2< \cdots < x_P$ and $d_{\overline{g}}(x_k,x_l)> \delta$ for $x_k\neq x_l$, $k,l=1,\ldots,P$. Then there are $f_1,\ldots, f_P\in\M\subset C^\infty(\Sigma)$  and corresponding solutions $v_{f_k}$ of $\square_g v_{f_k}=0$ with vanishing Cauchy data at $t=T$, such that the separation matrix $(v_{f_i}(x_j))_{i,j=1}^P$ in \eqref{separation_of_multiple_points_matrix} is invertible. 
\end{lemma}
\begin{proof}
\textbf{Case 1.} If $P=1$, then the situation is similar to the proof of Proposition~\ref{prop:separation_of_points}.
Applying Lemma~\ref{optimal_geo} to $x_1$, we find a past-directed boundary optimal geodesic $\gamma$ from  $\Sigma$ to $x_1$, whose first intersection with $\Sigma$ is transverse.
Using Corollary~\ref{correction_to_solution} we can then construct a Gaussian beam $v$ (including the correction term and with vanishing Cauchy data at $\{t=T\}$) corresponding to $\gamma$, such that
\[
v(x_1) = 1.
\]
By continuity of $v$, the point $x_1$ has a neighbourhood $B(x_1)$ such that
\[
|v(z)|>\frac{2}{3},\quad\text{for all } z\in B(x_1).
\]
Doing this for all points $x\in K$ we find an open cover of $K$ of the form
$$
\bigcup_{x\in K} B(x)
$$
and to each $B(x)$ the corresponding optimal geodesic and the respective Gaussian beam $v$.
Because $K$ is compact, there is a finite subcover
$$
\bigcup_{j=1}^R B(x^j)
$$
of $K$ and the corresponding finite collection of Gaussian beams.
Denoting the set of boundary values of these Gaussian beams by $\M$ completes the proof for $P=1$.

\textbf{Case 2.} Suppose now $P\geq 2$.
%
%If it happens that there are no points $x_2,\ldots, x_P\in K$ with $x_k\not\in B_\delta(x)$ and $x\leq x_k$, $k=2,\ldots,P$, then the claim reduces to Case 1. So we  assume that such points exist.
%
To begin, consider a complex matrix of the form
 \begin{equation}\label{Kmatrix}
	\begin{pmatrix}
        d_1 &  & \mathcal{O} \\
         & \ddots &  \\
        \#&  & d_P  \\
    \end{pmatrix},
 \end{equation}
where all entries $\#$ are bounded by a fixed constant $C>0$ and the diagonal entries satisfy $|d_j|>2/3$, $j=1,\ldots,P$.
When the elements of the upper triangular part $\mathcal{O}$ are of the size $\eps>0$, the determinant of the matrix in \eqref{Kmatrix} equals 
\[
d_1\cdots d_P+ O(\eps).
\]
This can be seen by considering the definition of the determinant in terms of minors.
Thus the matrix in \eqref{Kmatrix} is invertible when $\eps$ is  small enough.
%
%So $\mathcal{K}$ is invertible when $\eps$ is small enough.
%
%The parameter $\eps$ will be chosen later. 
%here will play a  role~\f{What is this?} later when choosing the  parameter $\tau>0$ for the Gaussian beams we will use.
%

We construct an open cover of $K$ as follows. 
Let $\tilde K \subset J^+(\Sigma) \cap J^-(\Sigma)\cap \big([0,T]\times\Omega\big)$ be an open neighbourhood of $K$.
Let us fix $x\in \tilde K$ and let $B_{\frac{\delta}{2}}(x)$ denote a $\frac{\delta}{2}$-radius ball centered at $x$ with respect to the metric $\overline{g}$. Let us also denote
\[
 \mathcal{V}(x):=\big(J^+(x)\setminus B_{\frac{\delta}{2}}(x)\big)\cap \big([0,T]\times\Omega\big).
\]
Since $J^+(x)$ is closed, the set $\mathcal{V}(x)$ is compact for all $x\in \tilde K$.
We define the subset of $\mathcal{V}(x)^{P-1}$ of ordered points by
\begin{align*}
 \mathcal{T}(x):=\{(x_2,\ldots,x_P)\in \mathcal{V}(x)^{P-1}: x\leq x_2&\leq \cdots \leq x_P\}.
 %&x_k\notin B_{\frac{\delta}{2}}(x), \ k=2,\ldots,P\}.
\end{align*}
%Then $\mathcal{T}(x)$ is compact,
Because the relation $\leq$ is closed (see e.g. \cite[Section 14, Lemma 22]{One83}), the set 
%and also the condition $d_{\overline{g}}(x_l,x_k)\geq \delta/2$, $k\neq l$, is closed, 
%the complement of $B_{\frac{\delta}{2}}(x)$ is closed, 
$\mathcal{T}(x)$ is compact as a closed subset of the compact set $\mathcal{V}(x)^{P-1}$.

% Let $X=(x,x_2,\ldots,x_P)\in \mathcal{T}(x)$. 
% %
% Since $x\in I^-(\Sigma)$, there is $f_X\in C^\infty(\Sigma)$ and a Gaussian beam $v_X$ (including the correction term) and $\tau_X>0$ such that there is a neighbourhood $B(x)\subset B_{\frac{\delta}{3}}(x)$ of $x$ and neighbourhoods $B_k(X)$ of $x_k$ such that
% \begin{equation}\label{eq:separate_first_point}
% \begin{split}
%  & \abs{v_{f_X}}\geq 2/3 \text{ on } B_1(x), \\
%  & \abs{v_{f_X}} < \eps \text{ on } B_{x_k}(X), \quad k=2,3,\ldots, P.
%  \end{split}
% \end{equation}

Let $\eps>0$ and let $X=(x_2,\ldots,x_P)\in \mathcal{T}(x)$. 
Recall that when constructing a Gaussian beam $v$, we can bound its size in absolute value by using the estimate \eqref{eq:other_points}.
%using the Sobolev embedding $H^s\subset L^\infty$ when $s > \frac{n+1}{2}$. 
%
Since $x\in I^-(\Sigma)\cap ([0,T]\times \Omega)$, there is $f_X\in C^\infty(\Sigma)$ and a Gaussian beam $v_X$ (including the correction term and with vanishing Cauchy data at $\{t=T\}$) and $\tau_X>0$ such that there is a neighbourhood $U_\eps(x)\subset B_{\frac{\delta}{3}}(x)$ of $x$ and neighbourhoods $B(x_k)$ of $x_k$ such that
\begin{equation}\label{eq:separate_first_point}
\begin{split}
 & \abs{v_{f_X}}\geq 2/3 \text{ on } U_\eps(x), \\
 & \abs{v_{f_X}} < \eps \text{ on } B(x_k), \quad k=2,3,\ldots, P, \\
 & \abs{v_{f_X}} \leq C \text{ on } [0,T]\times \Omega,
 \end{split}
\end{equation}
where $C>0$ is independent of $\eps>0$.
Here we have first normalized so that $v_{f_X}(x)=1$.
Then we have chosen the $\tau_X$ large enough,
% the neighbourhoods $B(x_k)$, $k=2,3,\ldots, P$
so that the condition $\abs{v_{f_X}} < \eps$ holds on $B(x_k)$, and that $\abs{v_{f_X}} \leq C \text{ on } [0,T]\times \Omega$. These conditions can be obtained since the correction term of a Gaussian beam can be made arbitrarily small by taking the corresponding $\tau$ large enough.  Then, by continuity of $v_{f_X}$ and $v_{f_X}(x)=1$, we have chosen the neighbourhood $U_\eps(x)$ so that $\abs{v_{f_X}}\geq 2/3$. Note that since here $\tau_X$ depends on $\eps$ and $v_{f_X}$ depends on $\tau_X$, the neighbourhood $U_\eps(x)$ depends on $\eps$ as indicated in the notation. See the argument in the proof of Proposition \ref{prop:separation_of_points} for more details.

We now modify the open sets $U_\eps(x)$ slightly.
Let us define
\[
\tilde{U}_\eps(x) := I^+(x) \cap U_\eps(x).
\]
We have that
\[
 \abs{v_{f_X}}\geq 2/3 \text{ on } \tilde U_\eps(x).
\]
Moreover, we have %if $z\in \tilde{U}_\eps(p_x)$, there is $p\in I^-(x)\cap U_\eps(x)$, such that
\begin{equation}\label{eq:Causality_Beps}
x \leq z \quad\text{for all} \quad z\in \tilde U_\eps(x).
\end{equation}
We then have an open cover of $\mathcal{T}(x)$ given by
 \[
 \bigcup_{X\in \mathcal{T}(x)} B(x_2)\times \cdots \times B(x_P).
\]
%Here $X=(x,x_2,\ldots,x_P)$.
Since $\mathcal{T}(x)$ is compact, we may pass to a finite open subcover 
\[
\bigcup_{X\in \mathcal{J_\eps}(x)} B(x_2)\times \cdots \times B(x_P),
\]
where $\mathcal{J_\eps}(x)$ is a finite subset of $\mathcal{T}(x)$ and which depends on $\eps$. Note that for each $X=(x_2,\ldots,x_P)\in \mathcal{J_\eps}(x)$ there are associated neighbourhoods $B(x_2),\ldots, B(x_P)$ of the points $x_2,\ldots,x_P$ and an open set $\tilde U_\eps(x)$.
This shows that to each point $x\in\tilde K$ we can attach a finite collection 
\[
\M_\eps(x) \subset C^\infty(\Sigma)
\]
of boundary values with the following property: For any $X\in \mathcal{T}(x)$ there is some $f_X\in \M_\eps(x)$ so that the corresponding solution $v_{f_X}$ is a Gaussian beam with the property \eqref{eq:separate_first_point} with $U_\eps(x)$ replaced by $\tilde U_\eps(x)$. %\abs{v_{f_X}}\geq 2/3 \text{ on } \tilde U_\eps(x) and also 
 
We repeat the above argument for all $x\in \tilde K$. Note that if $x\in K$, then there is $\tilde x\in \tilde K\cap J^-(x)$ so that $x\in \tilde U_\eps(\tilde x)$. Thus, our construction yields an open cover of $K\subset [0,T]\times \Omega$ by the sets $\tilde U_\eps(x)$ described above. By compactness, finitely many sets $\tilde U_\eps(x)$ suffice to cover $K$. Let $x^{(j)}\in [0,T]\times \Omega$ be the corresponding points, such that
\begin{equation}\label{eq:finite_sub_cover}
 \bigcup_{j=1}^{R_\eps} \tilde U_\eps(x^{(j)})
\end{equation}
is a finite subcover of $K$, where $R_\eps\in \N$.
To each of these finitely many points $x^{(j)}$ there is also attached a finite subset $\mathcal{J}_\eps(x^{(j)})\subset \mathcal{T}(x^{(j)})$, $j=1,\ldots,R_\eps$.
Corresponding to this finite cover, we take as the collection of boundary values $\M_\eps$ the set
\[
\mathcal{M_\eps}:=\bigcup_{j=1}^{R_\eps} \M_\eps(x^{(j)}).
\]

Let then $x_1,x_2,\ldots, x_P\in K$ with $x_1\leq x_2\leq \cdots \leq x_P$ and $d_{\overline{g}}(x_l,x_k)>\delta$ for $k\neq l$ with $k,l=1,\ldots,P$. 
Let us consider first the point $x_1\in K$.
Corresponding to $x_1$, there is an index $j_1\in \{1,\ldots, R_\eps\}$ and a neighbourhood $\tilde U_\eps(x^{(j_1)})$ of $x_1$, where $\tilde U_\eps(x^{(j_1)})$ belongs to the finite subcover \eqref{eq:finite_sub_cover} of $K$.
The radius of $\tilde U_\eps(x^{(j_1)})$ is less that $\frac{\delta}{3}$.
Note that $d_{\overline{g}}(x^{(j_1)},x_k)> {\frac{\delta}{2}}$ for $k=2,3,\ldots,P$.
Indeed, we have that
\begin{equation}\label{eq:dist_cal}
d_{\overline{g}}(x^{(j_1)},x_k)
\geq
d_{\overline{g}}(x_1,x_k)
-
d_{\overline{g}}(x^{(j_1)},x_1)
> \delta - \frac{\delta}{3}=\frac{2\delta}{3}>\frac{\delta}{2}.
\end{equation}
Moreover, \eqref{eq:Causality_Beps} implies $x^{(j_1)} \leq x_1$. Thus $x^{(j_1)}\leq x_2\leq x_3\leq\cdots\leq x_P$.
Using this and \eqref{eq:dist_cal}, we obtain
\[
 (x_2,x_3,\ldots,x_P)\in \mathcal{T}(x^{(j_1)}).
\]
Consequently, using the definition of $\mathcal{J}_\eps(x^{(j_1)})$, we find $X=(x_2^{(j_1)},\ldots,x_P^{(j_1)})\in \mathcal{J}_\eps(x^{(j_1)})$ with the  associated neighbourhoods %$\tilde U_\eps(p_{j_1})$ of $x_1$ and 
$B(x_k^{(j_1)})$ of $x_k$, $k=2,3,\ldots, P$, satisfying the following property: 
% Corresponding $x^{(j_1)}$ and the other points $x_2,x_3,\ldots,x_P$ there are neighbourhoods $B(x_k^{(j_1)})$ of $x_k$, $k=2,3,\ldots, P$, with the following property: 
%Corresponding $x^{(j_1)}$ there is $J(x^{(j_1)})$ and  neighbourhoods $B(x_k^{(j_1)})$ of $x_k$, $k=2,3,\ldots, x_P$ with the following property: 
There is a boundary value $f_1\in \M_\eps$ and the corresponding Gaussian beam $v_{f_1}$ solution to $\square v=0$, such that
\begin{align}\label{eq:properties_of_vf1}
 & \abs{v_{f_1}}\geq 2/3 \text{ on } \tilde U_\eps(x^{(j_1)})\nonumber\\ 
 & \abs{v_{f_1}}< \eps\text{ on } B(x_k^{(j_1)}), \quad k=2,3,\ldots, P, \\
  & \abs{v_{f_1}} \leq C \text{ on } [0,T]\times \Omega.\nonumber
\end{align}

 Let us then proceed to the point $x_2$. Similarly as above, regarding this point there is $j_2\in \{1,\ldots, R_\eps\}$, $x^{(j_2)}\in [0,T]\times \Omega$ and neighbourhoods $\tilde U_\eps(x^{(j_2)})$ of $x_2$ %, $B_\eps(x^{(j_2)})$ of $x^{(j_2)}$ and 
 %such that $x_2\in B_\eps(x^{(j_2)})$ and 
and neighbourhoods  $B(x_k^{(j_2)})$ of $x_k$, $k=3,4\ldots, x_P$, and a Gaussian beam $v_{f_2}$, such that
\begin{align*}
  &\abs{v_{f_2}}\geq 2/3 \text{ on } \tilde U_\eps(x^{(j_2)})\\ 
  &\abs{v_{f_2}} < \eps \text{ on } B(x_k^{(j_2)}), \quad k=3,4,\ldots, P, \\
   & \abs{v_{f_2}} \leq C \text{ on } [0,T]\times \Omega.
\end{align*}

Continuing in this manner, we have indices $j_1,j_2,\ldots,j_P$ and a set of Gaussian beams $v_{f_k}$, $k=1,\ldots, P$, such that $\abs{v_{f_k}}\geq 2/3$ on a neighbourhood $\tilde U_\eps(x^{(j_k)})$ of $x_k$ and $\abs{v_{f_k}} < \eps$ on a neighbourhood $B(x^{(j_k)}_l)$ of $x_l$ for $l>k$ and $\abs{v_{f_k}} < C$ on $[0,T]\times \Omega$. 

%Let us then consider the case $x^{(j_1)}\notin J^-(x_2)$.

The separation matrix \eqref{separation_of_multiple_points_matrix} corresponding to the functions $v_{f_k}$ and points $x_k$ is invertible for $\eps\leq \eps_0$ for $\eps_0$ small enough. We set $\M:=\M_{\eps_0}$. Finally, we note that the number of Gaussian beams used is
 \[
  \# \M=\# \left(\bigcup_{j=1}^{R_{\eps_0}} \M_{\eps_0}(x^{(j)})\right) \leq \sum_{j=1}^{R_{\eps_0}} \# \left(\M_{\eps_0}(x^{(j)})\right)=\sum_{j=1}^{R_{\eps_0}} \#\mathcal{J}_{\eps_0}(x^{(j)}),
 \]
which is finite.
\end{proof}

\begin{remark}
We will apply Lemma~\ref{lemma:separation_of_multiple_points} as follows.
Suppose the points $x_1< \cdots < x_P$ are the intersection points of two lightlike geodesics $\gamma_1$ and $\gamma_2$.
Lemma~\ref{lemma:separation_of_multiple_points} then shows that there is a choice of $P$ functions $f_1,\ldots,f_P$ from the finite collection $\M$ such that the corresponding Gaussian beams $v_{f_1},\ldots, v_{f_P}$ separate the points $x_1, \ldots, x_P$.
Moreover, the Gaussian beams constructed this way have zero Cauchy data at $t=T$. 

We also mention that we have a similar result as Lemma \ref{lem:separation_filter} for solutions that have vanishing Cauchy data at $\{t=0\}$. The result is obtained, for example, from Lemma \ref{lem:separation_filter} by considering the isometry $t\mapsto T-t$ as in Remark \ref{rem:remark_time_inversion}.
\end{remark}

\section{Proof of the stability estimate: Theorem \ref{thm:stability}}\label{sec:proof_of_stabilit_estimate}
Assume the conditions from Theorem~\ref{thm:stability}, especially that 
 \begin{equation*}
 \Vert \Lambda_1(f)-\Lambda_2 (f)\Vert_{H^r(\Sigma)} \leq \delta,
\end{equation*}
where $r\leq s+1$ and $s+1>(n+1)/2$, and $\delta>0$. Here $\Lambda_1$ and $\Lambda_2$ are the DN maps of the non-linear wave equation~\eqref{eq:Main equation} corresponding to the potentials $q_1$ and $q_2$, respectively. We show that 
% the We consider the difference $q_1-q_2$ and show that if the DN-maps $\Lambda_1$ and $\Lambda_2$ of the nonlinear wave equations~\eqref{eq:Main equation}, corresponding to the potentials $q_1$ and $q_2$ respectively, are sufficiently close to each other, then 
% 
we have explicit control on the $L^\infty$ norm of $q_1-q_2$ in terms of $\delta$.
The proof will be divided into several steps. 

\subsection{Step 1: Integral identity from finite differences}
Let  $j=1, \ldots, m$ and $\eps_j>0$ be small parameters. Let $\kappa$ be as in Lemma \ref{lemma:nonlinear-solutions}. Assume that $f_j\in H^{s+1}(\Sigma)$ is a family of functions satisfying $\p_t^{\alpha}f_j|_{t=0}=0$ on $[0,T]\times\p\Omega$, $\alpha=0, \ldots, s$, and that %.  Suppose also that 
\[
\norm{\eps_1 f_1 +\cdots+\eps_mf_m}_{H^{s+1}( [0,T]\times\Omega)} \leq \kappa.
\] 
%for some $\kappa>0$ small enough, as in Lemma  \ref{lemma:nonlinear-solutions}. 
%The boundary value $\eps_1 f_1 +\cdots+\eps_mf_m$ satisfies the compatibility conditions  \eqref{eq:epsilons_z}.
For $l=1,2$, we have that the boundary value problems
 \begin{equation}\label{eq:epsilons_z}
\begin{cases}
\square_g u_{l} + q_l\s \s u_l^m   =0, &\text{in }  [0,T]\times\Omega,\\
u_l=\eps_1f_1+\cdots+\eps_mf_m, &\text{on }  [0,T]\times\p\Omega,\\
u_l\big|_{t=0} = 0,\quad \s \partial_t u_l\big|_{t=0} = 0, &\text{in } \Omega
\end{cases}
\end{equation}
have unique small solutions $u_l={u}_{\s \eps_1 f_1 + \cdots +\eps_mf_m}$ as described in Lemma~\ref{lemma:nonlinear-solutions}. %In what follows, we shall use $u$ to denote a solution to \eqref{eq:epsilons_z} instead of using ${u}_{\s \eps_1 f_1 + \cdots \eps_mf_m}$. 
According to \eqref{id:expam_epsilons}, the solutions $u_l$ have expansions of the form 
\begin{equation}\label{eq:expansion_ul_proof}
\begin{aligned}
 u_{l}&=\eps_1v_{l,1}+\cdots+\eps_mv_{l,m} \\
 & \qquad  \qquad  \qquad  +\sum_{|\vec{k}|=m}\binom{m}{k_1,\ldots, k_m}\epsilon_1^{k_1} \s \cdots \s \epsilon_m^{k_m} \s w_{l, \vec k} %}\epsilon_2u_{(1,1)}+\frac{1}{2} \epsilon_1^2 u_{(2,0)}+\frac{1}{2} \epsilon_2^2 u_{(0,2)}
+\mathcal{R}_l,%(\langle \eps_1,\eps_2\rangle^3),
\end{aligned}
\end{equation}
%\begin{equation}\label{eq:expansion_ul_proof}
% u_{l}=\eps_1v_{l,1}+\eps_2v_{l,2} +\eps_3v_{l,3}+\eps_4v_{l,4} + 2 \eps_1\eps_2 w_{l,(1,1)} +\eps_1^2w_{l,(2,0)}+\eps_2^2w_{l,(0,2)} +\mathcal{R}_l,
%\end{equation}
where $v_{l,j}$ satisfy~\eqref{eq:norm_21} and $w_{l, \vec k}$ satisfy~\eqref{eq:norm_2} with $q$ replaced by $q_l$. We also used the notation $\vec k=(k_1,\ldots,k_m)$. In particular, we know by \eqref{eq:norm_2} that 
\[
w_{l, \vec{1}}:=w_{l,(1,\ldots, 1)} 
\]
satisfy
\begin{equation}\label{eq:norm_2_1_4}
\begin{cases}
\square_g \s w_{l, \vec{1}}+  q_l\s \s  v_{l,1} \cdots v_{l,m} = 0, &\text{in }  [0,T]\times\Omega,\\
w_{l, \vec{1}}=0, &\text{on } [0,T]\times\p\Omega, \\
w_{l, \vec{1}}\big|_{t=0} = 0,\quad \partial_t \s w_{l, \vec{1}}\big|_{t=0} = 0,&\text{in } \Omega.
\end{cases}
\end{equation}
Note that since the equations \eqref{eq:norm_21} for $v_{l,j}$ are independent of $q_l$, we have by the uniqueness of solutions that
\begin{equation} \label{eq:norm_2_1_4_z}
 v_{1,j}=v_{2,j}=:v_j,\quad j=1,\ldots, m.%=v_{3,j}= v_{4,j}=:v_j, \quad j=1,2, 3, 4.
\end{equation}
Moreover, according to \eqref{est:square_R}, the correction terms $\mathcal{R}_l$ for $l=1,2$ satisfy 
\begin{equation}\label{eq:estim_for_Rl}
\norm{\mathcal{R}_l}_{E^{s+2}} +  \norm{\square_g\, \mathcal{R}_l}_{E^{s+1}} \leq C(s, T)\s   \norm{q_l}_{E^{s+1}}^2 \Vert \varepsilon_1 f_1 + \cdots+ \varepsilon_m f_m\Vert_{H^{s+1}(\Sigma)}^{2m-1}. %, \quad l=1,2.
\end{equation}

We apply the finite difference operator $D^m_{\vec\eps}\big|_{\vec\eps=0}$ of order $m$ to $u_l$. %to obtain a useful identity that relates the solutions $u_1$ and $u_2$ with the unknown potentials $q_1$ and $q_2$. 
The finite difference operator was defined in \eqref{eq:fin_diff}. 
%obtain a useful identity, see \eqref{eq:difference_finite} below, that relates the solutions $u_l$ with the unknown potentials $q_1$ and $q_2$, we apply the finite difference operator $D^m_{\vec\eps}\big|_{\vec\eps=0}$ of order $m$ w.r.t $\eps_j$ defined by \eqref{eq:fin_diff} to $u_l$. 
By \eqref{eq:mixed_difference}, we have
\[
 D^m_{\vec\eps}\big|_{\vec\eps=0}u_l=m\s! \s\s w_{l, \vec{1}}+\frac{1}{\eps_1\cdots \eps_m}\mathcal{R}_l.
\]
Consequently, by taking into account \eqref{eq:norm_2_1_4} and \eqref{eq:norm_2_1_4_z} we obtain
\begin{equation}\label{eq:difference_finite}%\label{eq:equation_for_box_Rl}
 \square_g\s  D^m_{\vec\eps}\big|_{\vec\eps=0} \s u_l=- m\s! \s\s q_l \s \s v_{1}\cdots v_{m}+\frac{1}{\eps_1\cdots \eps_m}\square_g\s \tildeR_l,
\end{equation}
where $\tildeR_l = \eps_1\cdots \eps_m  \s \s  D^m_{\vec\eps}\big|_{\vec\eps=0}\mathcal{R}_l$, $l=1,2$.

We manipulate the integral identity \eqref{eq:integral_identity_finite_difference} to relate the difference of the DN maps $\Lambda_1$ and $\Lambda_2$ to the difference of the unknown potentials $q_1$ and $q_2$ in terms of $v_j$.
%with respect to the pair $(q_1, \Lambda_1)$ and $(q_2, \Lambda_2)$. 
% In this way, we shall arrive to an integral identity which relates the difference of the DN maps $\Lambda_1$ and $\Lambda_2$ %, and henceforth boundary information in $\Sigma$ associated to $u^{(1)}$ and $u^{(2)}$, 
% to the difference of $q_1$ and $q_2$ in $[0,T]\times\Omega$ in terms of $v_j$, which solve $\square_g v_j=0$. See \eqref{id:integral_identity} below. 
For this, consider an auxiliary function $v_0$, which satisfies $\square_g\s v_0=0$ in $[0,T]\times \Omega$ with $v_0|_{t=T} =\p_t v_0|_{t=T} = 0$ in $\Omega$.  Applying \eqref{eq:integral_identity_finite_difference}, to the difference of the DN maps yields  % with \eqref{eq:difference_finite}, we obtain 
\begin{equation}\label{id:integral_identity}
\begin{aligned}
&  -m\s\s !\int_{ [0,T]\times\Omega} (q_1-q_2)\s \s v_0\s v_1 \s \cdots\s v_m \s \d V_g = \frac{1}{\eps_1\cdots\eps_m}\int_{ [0,T]\times\Omega} v_0 \s \square_g\s (\widetilde{\mathcal{R}}_1-\widetilde{\mathcal{R}}_2)\s \d V_g\\
 &\qquad \qquad \qquad \qquad   + \int_{\Sigma}v_0\s D_{\vec\eps}^m\big|_{\vec\eps=0}\s (\Lambda_1- \Lambda_2)(\eps_1f_1+\cdots+ \eps_mf_m)\s \d S.
\end{aligned}
\end{equation}
The finite difference $D_{\vec\eps}^m\big|_{\vec\eps=0}$ of $u_l$ is a sum of $2^m$ terms. By using \eqref{id:integral_identity}, we calculate
\begin{equation}\label{eq:estimate_for_optimization_z}
\begin{aligned}
%\begin{aligned}
&m\s! \left| \langle v_0(q_1-q_2), v_1\s \cdots\s v_m \rangle_{L^2([0,T]\times\Omega)} \right|\\
&\,  \leq  \left| \langle v_0,   D_{\vec\eps=0}^m\left[(\Lambda_1- \Lambda_2)(\eps_1f_1+ \cdots +\eps_mf_m)\right]\rangle_{L^2(\Sigma)} \right| \\
& \qquad  +  \,(\eps_1\s \cdots \s \eps_m)^{-1} \,\left| \langle v_0, \square_g\s(\tilde R_1-\tilde R_2)\rangle_{L^2([0,T]\times\Omega)} \right|\\
 &\,  \leq  2^m\, (\eps_1\s \cdots\s \eps_m)^{-1}\,\left| \langle v_0, (\Lambda_1-\Lambda_2) \left( \eps_1 f_1+ \cdots+  \eps_m f_m \right) \rangle_{L^2(\Sigma)} \right| \\
 &\qquad  +  \,(\eps_1\s \cdots \s \eps_m)^{-1} \,\left| \langle v_0, \square_g\s(\tilde R_1-\tilde R_2)\rangle_{L^2([0,T]\times\Omega)} \right|\\
& \, \leq 2^m\, \delta \,(\eps_1 \s\cdots \s\eps_m)^{-1}\, \norm{v_0}_{\dualH^{-r}(\Sigma)} \\
& \qquad +  \,(\eps_1 \s\cdots \s\eps_m)^{-1}\, \norm{ \square_g\s(\tilde R_1-\tilde R_2)}_{H^{s+1}([0,T]\times\Omega)}\,\norm{v_0}_{\dualH^{-(s+1)}([0,T]\times\Omega)}\\
& \, \leq 2^m\, \delta \,(\eps_1 \s\cdots \s\eps_m)^{-1}\, \norm{v_0}_{\dualH^{-r}(\Sigma)} \\
& \qquad +  C_{s+1}\,(\eps_1 \s\cdots \s\eps_m)^{-1}\, \norm{ \square_g\s(\tilde R_1-\tilde R_2)}_{E^{s+1}}\,\norm{v_0}_{\dualH^{-(s+1)}([0,T]\times\Omega)}\\
%& \, \leq 4\, \delta \,\eps_1^{-1}\, \eps_2^{-1} \,\norm{v_0}_{\dualH^{-(s+1)}([0,T]\times\Omega)} + \eps_1^{-1}\, \eps_2^{-1}\, \norm{ \square\s(\tildeR_1-\tildeR_2)}_{E^{s+1}}\,\norm{v_0}_{\dualH^{-(s+1)}([0,T]\times\Omega)}\\
& \, \leq C_{m,s+1} \,(\eps_1 \s\cdots \s\eps_m)^{-1}\, \,(\norm{v_0}_{\dualH^{-r}(\Sigma)}+ \norm{v_0}_{\dualH^{-(s+1)}([0,T]\times\Omega)})\\
& \qquad  \times  \left[2^m\, \delta + C(s,T) \left(\norm{q_1}_{E^{s+1}}^2+\norm{q_2}_{E^{s+1}}^2\right) \left( \sum_{j=1}^m \varepsilon_j \Vert  f_j\Vert_{H^{s+1}(\Sigma)}\right)^{2m-1}\right]\\
&  \leq C  \,(\eps_1 \s\cdots \s\eps_m)^{-1}\, \left[ \delta  + \left( \sum_{j=1}^m \varepsilon_j \Vert  f_j\Vert_{H^{s+1}(\Sigma)}\right)^{2m-1} \right].
%\end{aligned}
\end{aligned}
\end{equation}
Here we used the assumption $\norm{\Lambda_1(f)-\Lambda_2(f)}_{H^r(\Sigma)}\leq \delta$ for $f=\eps_1f_1+ \s \cdots \s + \eps_mf_m$. We also used that the norm in $H^{s+1}([0,T]\times\Omega)$ is bounded by the norm in $E^{s+1}$ up to a multiplicative factor $C_{s+1}$ as noticed in Remark \ref{sob:emb_z}.
The final constant $C$ is given by
\begin{equation}
C=  \max  \left\{ C_{m,s+1}, C(s,T) (\norm{q_1}_{E^{s+1}}^2+\norm{q_2}_{E^{s+1}}^2)  \right\}  \,\left(\norm{v_0}_{H^{-r}(\Sigma)}+ \norm{v_0}_{H^{-(s+1)}([0,T]\times\Omega)}\right).
\end{equation}
Here we have respectively denoted by $\dualH^{-r}(\Sigma)$ and $\dualH^{-(s+1)}( [0,T]\times\Omega)$ the dual spaces of $H^{r}(\Sigma)$ and $H^{s+1}( [0,T]\times\Omega)$.
% 
% endowed with the following norms, see e.g. \cite{AF03},\footnote{Triebel, Hans Theory of function spaces. III. Monographs in Mathematics, 100. Birkhäuser Verlag, Basel, 2006. xii+426 pp. and https://arxiv.org/pdf/1107.3826.pdf}
% \begin{align*}
% \norm{w}_{\dualH^{-r}(\Sigma)}&:= \underset{v\in H^{r}(\Sigma),\, \norm{v}_{H^r(\Sigma)}\,\leq \,1}{\sup}\, \, |\langle v, w\rangle_{L^2(\Sigma)}|,\\
% \norm{w}_{\dualH^{-(s+1)}( [0,T]\times\Omega)}&:= \underset{v\in H^{s+1}( [0,T]\times\Omega),\, \norm{v}_{H^{s+1}( [0,T]\times\Omega)}\, \leq \,1}{\sup}\,\, |\langle v, w\rangle_{L^2(  [0,T]\times\Omega)}|.
% \end{align*}

\subsection{Step 2: Approximation of a delta distribution by a product of Gaussian beams} Recall that $(v_j)_{j=1}^m$ is a family solutions to $\square_g v_j=0$ as in~\eqref{eq:norm_21}. The second step of the proof of Theorem~\ref{thm:stability} is to choose the solutions $v_j$ so that they allow us to obtain information about $q_1-q_2$ on the left-hand side of the estimate~\eqref{eq:estimate_for_optimization_z}. The boundary values corresponding to $v_j$ will be denoted by $f_j$. We use the Gaussian beam construction of Section \ref{sec:Gaussian_beams} to produce approximate delta functions from products of Gaussian beams. We shall need the following elementary result. %, which is an elementary statement about uniform approximation of a $C^1$-function using convolution.

\begin{lemma}\label{est:tau_z}
Let $d\in \N$, $\tau>0$ and $b\in C^{1}_c(\mathbb{R}^{d})$. The following estimate  
\[
\left| b(z_0)-  \left(\frac{\tau}{\pi}\right)^{d/2} \int_{\mathbb{R}^{d}}b(z)\e^{-\tau |z-z_0|^2} \d z \right| \leq c_d \left\|b \right\|_{C^1}\tau^{-1/2}
\]
holds true for all $z_0\in \mathbb{R}^{d}$. In particular, the integral on the left converges uniformly to $b(z_0)$ when $\tau \to \infty$. Here $c_d:=\Gamma\left( \frac{d+1}{2}\right) / \Gamma \left( \frac{d}{2} \right)$.
\end{lemma}
\begin{proof}
Without loss of generality, we prove the estimate when $z_0=0$, because the estimate can be later applied to $b(z+z_0)$ in place of $b(z)$. We have the standard formulas
\[
\int_{\mathbb{R}^d}\e^{-|z|^2} \d z = \pi^{d/2}, \quad \int_{\mathbb{R}^d}\s |z|\,\e^{-|z|^2} \d z =c_d \s\s  \pi^{d/2},
\]
which can be obtained by using polar coordinates.
%\[
%\int_{\mathbb{R}^d}\e^{-|z|^2} \d z = \pi^{d/2}, \qquad \int_{\mathbb{R}^d}2\, |z|\,\e^{-|z|^2} \d z =c_d \s\s  \pi^{d/2}, \quad c_d:=\Gamma\left( \frac{d+1}{2}\right) / \Gamma \left( \frac{d}{2}+1 \right).
%\]
On the other hand, note that 
\[
 \left|  b(0)- b(\tau^{-1/2}z)\right| \leq \left\|b \right\|_{C^1}\tau^{-1/2}\left| z\right|,  \quad z\in \mathbb{R}^d,
 \]
which implies
\begin{align*}
&\left| b(0)-   \left( \frac{\tau}{\pi} \right)^{d/2} \int_{\mathbb{R}^d}b(z)\e^{-\tau|z|^2} \d z \right|= \left| b(0)-   \frac{1}{\pi^{d/2}} \int_{\mathbb{R}^d}b(\tau^{-1/2}z)\e^{-|z|^2} \d z\right| \\
&\qquad  = \left|  \frac{1}{\pi^{d/2}} \int_{\mathbb{R}^d}\left(b(0)-b(\tau^{-1/2}z)\right)\e^{-|z|^2} \d z\right|\\
&\qquad  \leq \frac{\tau^{-1/2}}{\pi^{d/2}}  \left\|b \right\|_{C^1} \int_{\R^d} |z|e^{-|z|^2} \d z   = c_d  \left\|b \right\|_{C^1}\tau^{-1/2}
\end{align*}
as claimed.
\end{proof}

In what follows, we shall apply Lemma \ref{est:tau_z} with $d=n+1$. As the function $b$ in the lemma we will have a function related to $q_1-q_2$. To achieve the factor $\tau^{d/2}=\tau^{(n+1)/2}$ appearing in Lemma~\ref{est:tau_z}, we use the solutions of Corollary \ref{correction_to_solution} with $p=4$ and scale them by a constant $\tau^{1/8}$. This change amounts to scaling the boundary values $f_j$ by $\tau^{1/8}$.
The estimates \eqref{Hk_and_L4} and \eqref{estimate_for_the error_r} still hold by taking $k$, $l$, $K$ and $N$ large enough.

Recall that Gaussian beams concentrate on lightlike geodesics. 
We show that at the intersection points of geodesics, the corresponding product of Gaussian beams approximate the delta function of the intersection point. Taking this approach, one can recover information about the difference of the unknown potentials $q_1$ and $q_2$ at points where the geodesics intersect. When the geodesics intersect only once, the proof is simpler and instructive. For this reason, we first analyze the case where the geodesics intersect only once and prove the general case after that.

\subsection{Proof in the case of a single intersection point} Let 
%$p_0\in \supp(q_1)\cup \supp(q_2)$.
$p_0\in W$, where $W$ is as in \eqref{eq:recovery_set}. In this case $p_0\in I^+(\Sigma)$ by assumption and by Lemma~\ref{optimal_geo} there is a future-directed optimal geodesic $\gamma_1$ from $\Sigma$ to $p_0$ that does not intersect $\{t=0\}$. By making a small perturbation of $\gamma_1$, we have another geodesic $\gamma_2$ that intersects $\gamma_1$ at $p_0$ and does not intersect $\{t=0\}$. Since the geodesics are causal, they exit the compact set $[0,T]\times \Omega$ in finite parameter time. By the assumption of this simplified case, $\gamma_1$ and $\gamma_2$ intersect only at $p_0$. Let $\delta'>0$ be small parameter so that the Fermi coordinates \eqref{Fermi_coords_def}, associated to $\gamma_1$ and $\gamma_2$, are defined for $|y|<\delta'$.  %and for the coordinates $s$ belonging to compact intervals. 

%We also fix $\tau_0>0$ so that the Gaussian beam construction in Proposition \ref{Gaussian_beam_construction} and Corollary \ref{correction_to_solution} are valid for $\tau\geq \tau_0$.

%Let $N\in \N$. 
By Proposition \ref{Gaussian_beam_construction} and Corollary \ref{correction_to_solution} there is $\tau_0>0$ such that for $j=1,2$ and $\tau\geq \tau_0$, we may choose
\begin{equation} \label{sol:decay_tau_z}
 v_{j}=\tau^{1/8}(v_{\tau, j}+r_j) \quad \text{and} \quad f_j= v_{j}|_{\Sigma}, \quad j=1,2, 
\end{equation}
%\begin{equation} \label{sol:decay_tau_z}
% v_{j}=\tau^{1/8}(v_{\tau, j}+r_j), \quad v_{2}=\tau^{1/8}(v_{\tau, 2}+r_2),
%\end{equation}
so that $ \square_g( v_{\tau, {j}}+r_{j})=0$ in $[0,T]\times \Omega$. % are the corresponding solutions to $\square_g\s v=0$ of Corollary~\ref{correction_to_solution} for $\tau\geq \tau_0$. 
Here the function $v_{\tau,{j}}$ stands for the Gaussian beam described in Section \ref{sec:Gaussian_beams} corresponding to the geodesic $\gamma_{j}$. We also have that the correction term $r_j$ satisfies
\begin{equation}\label{eq:r_vanish_on_Sigma}
 r_j|_\Sigma=0, \quad j=1,2.
\end{equation}
By \eqref{def:a_k} and \eqref{def:v_k} and Proposition~\ref{Gaussian_beam_construction} applied with $p=4$, we have for $\tau\geq \tau_0$
\begin{equation}\label{eq:cut_off_z}
\begin{aligned}
v_{\tau, {j}}(s,y)&= \tau^{\frac{n}{8}}\s e^{i\s \tau\s \Theta_{j} (s,y)} a^{(j)}(s,y),&\tau \geq \tau_0, & \\%j=1,2, \quad \tau\geq \tau_0, \\
a^{(j)}(s,y)&= \chi\Big(\frac{|y|}{\delta'}\Big) \sum_{k'=0}^N \tau^{-k'} b_{k'}^{(j)}(s,y),&\tau \geq \tau_0, & \\
b_{k'}^{(j)}(s,y)&= \sum_{k^{''}=0}^N b_{k', k^{''}}^{(j)}(s,y),& %j=1,2, \quad \tau\geq \tau_0,
\end{aligned}
\end{equation}
%\begin{equation}\label{eq:cut_off_z}
%\begin{aligned}
%v_{\tau, {j}}(s,y)&= \tau^{\frac{n}{8}}\s e^{i\s \tau\s \Theta_{j} (s,y)} a_{j}(s,y),&\tau \geq \tau_0, & \\%j=1,2, \quad \tau\geq \tau_0, \\
%a_j(s,y)&= \chi\Big(\frac{|y|}{\delta'}\Big) b_j,& \\
%b_{j}(s,y)&= \sum_{k'=0}^N \tau^{-k'}  \left[\sum_{k^{''}=0}^N b_{k', k^{''}}^{(j)}(s,y)\right],& \tau\geq \tau_0, %j=1,2, \quad \tau\geq \tau_0,
%\end{aligned}
%\end{equation}
where $b_{k', k^{''}}^{(j)}(s,y)$ is a family of complex-valued homogeneous polynomial of order $k^{''}$ in the variable $y$. We emphasize that all functions on the right hand sides of \eqref{eq:cut_off_z} are independent of $\tau$. Thanks to Proposition \ref{Gaussian_beam_construction}, see also \eqref{eq:Ric_z_1} and \eqref{eq:Ric_z_2}, we also have
\begin{equation}\label{first_term_b_0_0_z}
b_{k'}^{(j)}(0,0)=b_{0, 0}^{(j)}(0,0)=1, \quad j=1,2.
\end{equation}
In addition, by \eqref{def:k_K}, \eqref{Hk_and_L4} and \eqref{estimate_for_the error_r}, we get for $j=1,2$ and $k>l + (n-1)/2$
 \begin{equation}\label{est:ineq_z}
 \begin{aligned}
\norm {v_{\tau,{j}}}_{H^l( [0,T]\times\Omega)}&=O(\tau^{-\frac{n}{8}+l}),& \tau\geq \tau_0,  \\% \tau\geq \tau_0, \\
 \norm{r_{j}}_{H^l( [0,T]\times\Omega)}& =O(\tau^{-K}),& \tau\geq \tau_0,%\tau\geq \tau_0,  
\end{aligned}
 \end{equation}
if $N$ satisfies $K= (N+1-k)/2 -1$. (If $N$ defined this way is not an integer, we redefine it as $\lfloor N+1\rfloor$.)  We imposed the condition $k>l + (n-1)/2$ to embed the energy space $E^l$ into $H^{k}([0,T]\times \Omega)$, see Remark \eqref{sob:emb_z}.  This condition is needed to control certain certain Sobolev norms in the following computations. Furthermore, by \eqref{Hk_and_L4} and assuming that $l>(n+1)/4$ (to embed $H^l([0,T]\times \Omega)$ into $L^4([0,T]\times \Omega)$) we get 
\begin{equation}\label{v_tau_r_l_4}
\begin{aligned}
\norm{v_{\tau,j}}_{L^4([0,T]\times \Omega)}= O(1), & \qquad  j=1,2, \quad \tau\geq \tau_0,\\
\norm{r_{j}}_{L^4([0,T]\times \Omega)}=O(\tau^{-K}), & \qquad  j=1,2, \quad \tau\geq \tau_0.
\end{aligned}
\end{equation}
%Here $(N, k, l, K)$ is a set of positive numbers, whose values will be chosen later to be large enough so that we can use Sobolev embeddings. 
 %in order to control the $\tau$ behavior of certain Sobolev norms in what follows. 
 %The precise value of $(N, k, l, K)$ is not important.
 
Since $\square_g$ is linear, the complex conjugates of $v_1$ and $v_2$, denoted by $\overline{v}_1$ and $\overline{v}_2$, also solve $\square_g\s  v=0$. We set
\[
v_j:=\overline{v}_{j-2} \quad \text{and} \quad  f_j:= v_j|_{\Sigma}, \quad j=3,4.
\]
%Note that the complex conjugation preserves Sobolev norms. 
Combining the trace theorem  with \eqref{eq:r_vanish_on_Sigma} and \eqref{est:ineq_z} in the case $l=s+3/2$, we obtain an estimate for the boundary values $f_j$ for $j=1,2,3,4$ and $\tau\geq\tau_0$, as
\begin{equation}\label{f_j_restrcit_z}
\begin{split}
\norm{f_j}_{H^{s+1}(\Sigma)} &= \norm{v_j|_{\Sigma}}_{H^{s+1}(\Sigma)}= \tau^{1/8}\norm{(v_{\tau,j}+r_j)|_{\Sigma}}_{H^{s+1}(\Sigma)} \\
& \leq \tau^{1/8}\norm{v_{\tau,j}}_{H^{s+3/2}( [0,T]\times\Omega)} \leq C \tau^{s -\frac{n}{8} + \frac{13}{8}}.
\end{split}
\end{equation}
%\begin{equation}\label{f_j_restrcit_z}
%\begin{aligned}
%\norm{f_j}_{H^{s+1}(\Sigma)}& = \norm{v_j|_{\Sigma}}_{H^{s+1}(\Sigma)} = \norm{v_{\mu_j}|_{\Sigma}}_{H^{s+1}(\Sigma)}\\ & \quad \leq \norm{v_{\mu_j}}_{H^{s+3/2}(\Omega \times [0,T])} \lesssim \tau^{s -\frac{n}{8} + \frac{3}{2}}.
%\end{aligned}
%\end{equation}
For $j=5, \ldots, m$, we choose Gaussian beams at fixed $\tau=\tau_0$ as 
\begin{equation}\label{choice_j_5_m}
v_j =\tau_0^{-(n+1)/8} v_1|_{\tau=\tau_0} \quad \text{and} \quad  f_j=v_j|_{\Sigma}, \quad j=5,\dots, m.
\end{equation}
Let us write
\[
\hat{v}= v_5 \cdots v_m.
\]
\begin{remark} \label{remove_hat_v_z} %One can check that $\hat{v}$ is continuous. 
We remark that by making $\tau_0>0$ large enough there exists $c>0$ such that 
\begin{equation}\label{est:v_hat_z}
|\hat{v}(s,y)| >c
\end{equation}
in a neighbourhood of $(s,y)=(0,0)$. Indeed, by taking $l>(n+1)/2$ and combining Morrey's inequality with \eqref{est:ineq_z}, we deduce that both $v_{\tau, 1}$ and $r_1$ are continuous functions for $\tau\geq \tau_0$. In particular, the function $v_1$ is continuous according to \eqref{sol:decay_tau_z}. % and $\hat{v}$ so.
 Proposition \ref{Gaussian_beam_construction} ensures that $\Theta_1(0,0)=0$ and $b_{0,0}^{(1)}(0,0)=1$. Looking at \eqref{eq:cut_off_z} one has
\[
a_1(0,0)= 1+ O(\tau^{-1}), \quad \tau\geq \tau_0.
\]
Hence 
\[
\tau^{-(n+1)/8} v_1(0,0)=1+ \tau^{-n/8}\s \s r_1(0,0)=1+O(\tau^{-n/8}), \quad \tau\geq \tau_0,
\]
where in the last equality, we have used \eqref{est:ineq_z} to deduce $\norm{r_1}_{L^\infty( [0,T]\times\Omega)}= O(1)$.  Thus we have, by redefining $\tau_0$ if necessary, that
\[
 \abs{\hat v(0,0)}=\big(\tau^{-(n+1)/8} |v_1(0,0)|\big)^{m-4}>1/2
\]
for all $\tau\geq \tau_0$. By continuity of $\hat{v}$, we have \eqref{est:v_hat_z} on a neighbourhood of $(0,0)$.  
% 
% 
% Let $\hat{\tau}_0>0$ be chosen so that  
% \[
% \tau^{-(n+1)/8} |v_1(0,0)| > 1/2, \quad \tau\geq \hat{\tau}_0. 
% \]
% Redefining $\tau_0>0$ bigger than both the previous one and $\hat{\tau}_0$, we get for $j=5, \ldots, m$
% \[
% |v_j(0,0)|>1/2.
% \]
% Thus, inequality \eqref{est:v_hat_z} follows by the definition of $\hat{v}$ and by continuity. 
\end{remark}
%\begin{remark} \label{remove_hat_v_z}
%Let $j=5, \ldots, m$. Note that $v_j$ does not depend on $\tau$ and, therefore, neither does $\hat{v}$. By using both \eqref{est:ineq_z} with $l>(n+1)/2$ and Morrey's inequality, we deduce $\norm{\hat{v}}_{L^\infty( [0,T]\times\Omega)}= O(1)$ and $\norm{r_1}_{L^\infty( [0,T]\times\Omega)}= O(1)$. For subsequent computations, we claim that there is a uniform $c>0$ such that 
%\begin{equation}\label{est:v_hat_z}
%|\hat{v}(s,y)| >c, \quad |s|<\delta^\prime \quad \text{and} \quad |y|<\delta^\prime.
%\end{equation}
%By continuity, it is enough to prove the above estimate for $\hat{v}$ at $(0,0)$. Proposition \ref{Gaussian_beam_construction} ensures that $\Theta_1(0,0)=0$ and $b_{0,0}^{(1)}(0,0)=1$. Looking at \eqref{eq:cut_off_z} one has $a_1(0,0)= 1+ O(\tau_0^{-1})$. Hence 
%\begin{align*}
%v_j(0,0)&= \tau_0^{-(n+1)/8} \left[\tau_0^{(n+1)/8}e^{i\s \tau_0 \s \Theta_1(0,0)} a^{(1)}(0,0)+   \tau_0^{1/8} r_1(0,0) \right]= 1+ O(\tau_0^{-1}),
%\end{align*}
%which immediately implies \eqref{est:v_hat_z} at $(s,y)=(0,0)$. The claim follows by continuity.
%\end{remark}

%Thanks to \eqref{est:ineq_z} and \eqref{f_j_restrcit_z}, the function $v_j$ and its boundary value $f_j$, $j=5, \ldots, m$, remain bounded in higher Sobolev indexes and so $v^{\tau_0}$. Furthermore, taking $l$ large enough, e.g. $l>(n+1)/2$, we get from \eqref{est:ineq_z} 
%\begin{equation}
%\norm{v^{\tau_0}}_{L^{\infty}([0,T]\times \Omega)} \leq C(\tau_0). 
%\end{equation}

With these choices, we now analyze the left-hand side of \eqref{eq:estimate_for_optimization_z}. %, we now concentrate on obtaining a useful expression for $v_1 \ldots v_m$. 
We decompose the product $v_1 \cdots v_m$ as the sum of a leading term plus lower order terms. A straightforward computation holding for $\tau\geq \tau_0$ yields
\begin{equation}\label{mult:product_z}
\begin{aligned}
& v_1 \cdots v_m \\
 &\quad = |v_1|^2|v_2|^2 \hat{v}\\
 &\quad= \tau^{1/2}|v_{\tau,1}+ r_1|^2 |v_{\tau,2}+ r_2|^2 \s \hat{v}\\
 &\quad =   \tau^{1/2} \left( |v_{\tau,1}|^2 + v_{\tau,1} \s \overline{r}_1 + r_1 \s \overline{v}_{\tau,1}+|r_{1}|^2  \right)\left( |v_{\tau,2}|^2 + v_{\tau,2} \s \overline{r}_2 + r_2 \s \overline{v}_{\tau,2}+|r_{2}|^2  \right)  \s \hat{v}  \\
 %&=  \tau^{1/2}\tau_0^{(m-4)/8}|v_{\tau,1}+ r_1|^2 |v_{\tau,2}+ r_2|^2 \s \s \s  \displaystyle{\sum_{k'=0}^{m-4} \s \left(v_{\tau_0,1}\right)^{m-4-k'}  \s \left( r_{1|_{\tau=\tau_0}}\right)^{k'} }\\
 %&=  \tau^{1/2}\tau_0^{(m-4)/8}|v_{\tau,1}+ r_1|^2 |v_{\tau,2}+ r_2|^2 \s \s \left[v_{\tau_0,1}^{m-4} + \displaystyle{\sum_{k'=1}^{m-4} \s \left(v_{\tau_0,1}\right)^{m-4-k'}  \s \left( r_{1|_{\tau=\tau_0}}\right)^{k'} } \right]\\
 %&\quad =  \tau^{1/2}\tau_0^{(m-4)/8}\s \s v_{\tau_0,1}^{m-4}\s \s |v_{\tau, 1}|^2|v_{\tau, 2}|^2 + \mathcal{L}_m^1\\
 &\quad  =  \tau^{1/2}  \s |v_{\tau, 1}|^2|v_{\tau, 2}|^2  \s \hat{v}  + \mathcal{L}_1,
 \end{aligned}
\end{equation}
%\begin{equation}\label{mult:product_z}
%\begin{aligned}
% v_1 \cdots v_m& = |v_1|^2|v_2|^2 \hat{v}\\
% &  = \tau^{1/2}|v_{\tau,1}+ r_1|^2 |v_{\tau,2}+ r_2|^2 \s \hat{v}\\
% &=  \tau^{1/2}\tau_0^{(m-4)/8}|v_{\tau,1}+ r_1|^2 |v_{\tau,2}+ r_2|^2 \s \s \s  \displaystyle{\sum_{k'=0}^{m-4} \s \left(v_{\tau_0,1}\right)^{m-4-k'}  \s \left( r_{1|_{\tau=\tau_0}}\right)^{k'} }\\
% &=  \tau^{1/2}\tau_0^{(m-4)/8}|v_{\tau,1}+ r_1|^2 |v_{\tau,2}+ r_2|^2 \s \s \left[v_{\tau_0,1}^{m-4} + \displaystyle{\sum_{k'=1}^{m-4} \s \left(v_{\tau_0,1}\right)^{m-4-k'}  \s \left( r_{1|_{\tau=\tau_0}}\right)^{k'} } \right]\\
% %&\quad =  \tau^{1/2}\tau_0^{(m-4)/8}\s \s v_{\tau_0,1}^{m-4}\s \s |v_{\tau, 1}|^2|v_{\tau, 2}|^2 + \mathcal{L}_m^1\\
% &=  \tau^{1/2}\s \tau_0^{(m-4)/8} v_{\tau_0,1}^{m-4}  \s |v_{\tau, 1}|^2|v_{\tau, 2}|^2 + \mathcal{L}_m^1,
% \end{aligned}
%\end{equation}
where $\mathcal{L}_1$ is a sum of products of terms each containing $r_1$ or $r_2$, or their complex conjugates, as well as $\hat{v}$ as a factor. Consequently, we can choose  $(N, k, l, K)$ in \eqref{est:ineq_z} so that together with the Cauchy-Schwarz inequality, we obtain 
\begin{equation}\label{est:remainder_z}
 \norm{\mathcal{L}_1}_{L^1([0,T]\times  \Omega)}=O(\tau^{-R})
\end{equation}
for some $R>1$. Indeed, let us analyze one term of $\mathcal{L}_1$, say $ \tau^{1/2}v_{\tau, 1} \s  |v_{\tau, 2}|^2 \s \overline{r}_1\s \hat{v}$. %By Remark \ref{remove_hat_v_z}, the 
As $\hat{v}$ is continuous, it is bounded in $[0,T]\times \Omega$. Using \eqref{v_tau_r_l_4}, we have for $\tau\geq \tau_0$
\begin{align*}
& \tau^{1/2} \norm{v_{\tau, 1} \s  |v_{\tau, 2}|^2 \s \overline{r}_1\s \hat{v}}_{L^1([0,T]\times \Omega)}\\
& \lesssim  \tau^{1/2} \norm{v_{\tau, 1} \s  |v_{\tau, 2}|^2 \s \overline{r}_1}_{L^1([0,T]\times \Omega)} \\
%&\lesssim   \tau^{1/2} \norm{v_{\tau, 1} \s  |v_{\tau, 2}|^2 \s \overline{r}_1}_{L^1([0,T]\times \Omega)} \\
& \lesssim  \tau^{1/2} \s \norm{v_{\tau,1}}_{L^4([0,T]\times \Omega)} \s \norm{v_{\tau,2}}_{L^4([0,T]\times \Omega)}^2\s  \s \norm{r_{1}}_{L^4([0,T]\times \Omega)} = O(\tau^{1/2-K}).
\end{align*}
A similar analysis allows us to deduce that the $L^1([0,T]\times \Omega)$ norms of the other terms of $\mathcal{L}_1$ are $O(\tau^{1/2-K})$. Therefore	
\[
 \norm{\mathcal{L}_1}_{L^1([0,T]\times  \Omega)}=O(\tau^{1/2- K}), \quad \tau\geq \tau_0.
\]
Thus we can take $R= K-1/2$ in \eqref{est:remainder_z}. Note that we can always find suitable parameters $l$, $k$, $N$ and $K$ satisfying $K= (N+1-k)/2 -1$, $k>l + (n-1)/2$ and $l>(n+1)/2>(n+1)/4$. One possible choice is %, for instance
\[
l=n+1, \quad k=3n+1, \quad K=2, \quad N=3(n+1).
\]

Let us now analyze the leading term in the expansion \eqref{mult:product_z}: 
%\begin{equation}
%\begin{aligned}
%& \tau^{1/2}\s \tau_0^{(m-4)/8} v_{\tau_0,1}^{m-4}  \s |v_{\tau, 1}|^2|v_{\tau, 2}|^2\\
% &=\tau^{\frac{n+1}{2}}\s  \tau_0^{(n+1)(m-4)/8} \s  \s e^{i\tau_0(m-4) \Theta_1(x)}\s e^{i\tau \Theta_1(x)}e^{-i\tau \overline{\Theta}_1(x)}e^{i\tau \Theta_2(x)}e^{-i\tau \overline{\Theta}_2(x)} \s \left(a_1(x)|_{\tau=\tau_0} \right)^{m-4} \s |a_1(x)|^2|a_2(x)|^2,
%\end{aligned}
%\end{equation}
\begin{equation}
\begin{aligned}
 \tau^{1/2}  \s |v_{\tau, 1}|^2|v_{\tau, 2}|^2  \s \hat{v} =\tau^{\frac{n+1}{2}}\s e^{i\tau \Theta_1(x)}e^{-i\tau \overline{\Theta}_1(x)}e^{i\tau \Theta_2(x)}e^{-i\tau \overline{\Theta}_2(x)} \s |a^{(1)}(x)|^2|a^{(2)}(x)|^2 \s \s \hat{v}(x).
\end{aligned}
\end{equation}
%\begin{equation}
%\begin{aligned}
%& \tau^{1/2}\s \tau_0^{(m-4)/8} v_{\tau_0,1}^{m-4}  \s |v_{\tau, 1}|^2|v_{\tau, 2}|^2\\
% &=\tau^{\frac{n+1}{2}}\s e^{i\tau \Theta_1(x)}e^{-i\tau \overline{\Theta}_1(x)}e^{i\tau \Theta_2(x)}e^{-i\tau \overline{\Theta}_2(x)} \s |a_1(x)|^2|a_2(x)|^2\\
% & \qquad \qquad \qquad \qquad\qquad  \times  \tau_0^{(n+1)(m-4)/8} \s  \s e^{i\tau_0(m-4) \Theta_1(x)}\s \left(a_1(x)|_{\tau=\tau_0} \right)^{m-4}.
%\end{aligned}
%\end{equation}
%where we denoted $x=(s,y)$, the Fermi coordinates near the geodesics $\gamma_1$ and $\gamma_2$. 
For technical convenience, we consider a normal coordinate system $(x^a)_{a=0}^n$ centered at the point $p_0$, which is the unique intersection of the geodesics $\gamma_1$ and $\gamma_2$. At the point $p_0$ both the phase functions $\Theta_1$ and $\Theta_2$ vanish and their gradients are real. Using the properties~\eqref{Phi_prop_sec5}, we have the following Taylor expansion around $p_0$
\begin{align*}
 \Theta_1(x)-\overline{\Theta}_1(x)+\Theta_2(x)-\overline{\Theta}_2(x)=2 \s \s i \s\s  x\cdot \nabla^2\text{Im}(\Theta_1+\Theta_2)\big|_{x=0} x +O(|x|^3).
\end{align*}
Here $\nabla^2\text{Im}(\Theta_1+\Theta_2)$ is a positive definite matrix at $p_0$ (i.e., at $x=0$ in normal coordinates) by the last two conditions of~\eqref{Phi_prop_sec5}, because $\Theta_1$ and $\Theta_2$ are positive semi-definite and positive definite in directions transversal to $\dot \gamma_1$ and $\dot\gamma_2$ respectively.

Recall from~\eqref{eq:cut_off_z} that the amplitude $a^{(j)}$, $j=1,2$, has the cut-off function $\chi$ as a factor. Therefore, we may redefine $\delta'>0$ smaller, if necessary, so that at the intersection $U_1\cap U_2$ of the supports 
\[
U_j:=\supp(a^{(j)})= \supp(v_{j,\tau})
\]
we have $\text{Im}(\Theta_1+\Theta_2)>0$. Let us write
\begin{equation}\label{eq:mathcalH}
 \mathcal{H}:=2\nabla^2\text{Im}(\Theta_1+\Theta_2)\big|_{x=0}>0
\end{equation}
so that in the coordinates
\begin{equation}\label{eq:def_of_hatTheta}
 \Theta_1(x)-\overline{\Theta}_1(x)+\Theta_2(x)-\overline{\Theta}_2(x)=i \s x\cdot \mathcal{H}\s x + \hat{\Theta}(x),
\end{equation}
where %$\nabla^2\text{Im}(\Theta_1+\Theta_2)\big|_{x=0}$ is positive definite and 
$\hat{\Theta}(x) = O(|x|^3)$. 
%Let us write
%\[
% \mathcal{H}:=2\nabla^2\text{Im}(\Theta_1+\Theta_2)\big|_{x=0}.
%\]
Using the precise expressions in \eqref{eq:cut_off_z} for $a^{(j)}$, $j=1,2$, we see that
\begin{align*}
 |a^{(1)}(x)|^2|a^{(2)}(x)|^2 &=  |b_{0}^{(1)}(x)|^2\s |b_{0}^{(2)}(x)|^2 + \tau^{-1} \mathcal{L}_2(x),
\end{align*}
where 
\begin{equation}\label{est:remainder_z_z_z}
 \norm{\mathcal{L}_2}_{L^1([0,T]\times  \Omega)}=O(1).
 \end{equation}
Via a similar calculation as was done in deriving \eqref{mult:product_z}, we deduce in the coordinates $(x^a)_{a=0}^n$ that
%\begin{equation}\label{second:expansion_x_z}
%\begin{aligned}
%&\tau^{1/2}  \s |v_{\tau, 1}|^2|v_{\tau, 2}|^2  \s \hat{v} = \tau^{\frac{n+1}{2}}\s\left|\chi_1(x)\right|^2\s\left|\chi_2(x)\right|^2\s |b_{0,0}^{(1)}|^2\s |b_{0,0}^{(2)}|^2\s \hat{v}(x) \s  e^{i \tau \hat{\Theta}(x)} \s \s  e^{-\tau x\cdot \mathcal{H}\s x}+  \mathcal{L}_2,
%\end{aligned}
%\end{equation}
\begin{equation}\label{second:expansion_x_z}
\begin{aligned}
\tau^{1/2}  \s |v_{\tau, 1}|^2|v_{\tau, 2}|^2  \s \hat{v} &= \tau^{\frac{n+1}{2}}\s\left|\chi_1(x)\right|^2\s\left|\chi_2(x)\right|^2\s |b_{0}^{(1)}(x)|^2\s |b_{0}^{(2)}(x)|^2\s \hat{v}(x) \s  e^{i \tau \hat{\Theta}(x)} \s \s  e^{-\tau x\cdot \mathcal{H}\s x}\\
& \quad + \underset{:= \widehat{\mathcal{L}}_2(x)}{\underbrace{\tau^{-1}\tau^{\frac{n+1}{2}}\s\left|\chi_1(x)\right|^2\s\left|\chi_2(x)\right|^2\s \hat{v}(x) \s  e^{i \tau \hat{\Theta}(x)} \s \s  e^{-\tau x\cdot \mathcal{H}\s x} \mathcal{L}_2(x)}}.
\end{aligned}
\end{equation}
Here the functions $\chi_j$, $j=1,2$, stand for the normal coordinate representations of $\chi_j$, which in Fermi coordinates $(s,y)$ corresponding to the geodesics $\gamma_j$ take the form $\chi(\frac{|y|}{\delta'})$. Note that $\chi_j(0)=1$. Recall that $\hat{\Theta}(x) = O(|x|^3)$. By using \eqref{est:remainder_z_z_z},  making the change of variables $x\mapsto \tau^{-1/2}x$ and using the fact $\tau\s \hat{\Theta}(\tau^{-1/2}x)= \tau^{-1/2}O(|x|^3)= O(|x|^3)$
%and $\d x$ under the above change of variables reads $\tau^{-(n+1)/2}\d x$, 
one calculates that
\begin{equation}\label{est:remainder_z_z}
 \norm{\widehat{\mathcal{L}}_2}_{L^1([0,T]\times  \Omega)}=O(\tau^{-1}).
 \end{equation}
 (See \eqref{estz:reg_z} below for a similar calculation.)
%The last estimate is due to the fact that $a_j$, $j=1,2$, apart from the cut-off function $\chi$, only contains complex valued homogeneous polynomial of finite order in the variable $y$, each of them multiplied by some decay in $\tau$ as a factor except for the leading terms $b_{0,0}^{(1)}$ and $b_{0,0}^{(2)}$.\\

For the sake of brevity, we set 
\begin{equation}\label{eq:def_of_A}
q(x) = q_1(x)-q_2(x),  \quad A(x) = \left|\chi_1(x)\right|^2\s\left|\chi_2(x)\right|^2\s |b_{0}^{(1)}(x)|^2\s |b_{0}^{(2)}(x)|^2\s \hat{v}(x).
\end{equation}
%\begin{equation}
%q = q_1-q_2,  \quad A(x) =  \left[\chi\Big(\frac{|y|}{\delta'}\Big)\right]^4\s |b_{0,0}^{(1)}|^2\s |b_{0,0}^{(2)}|^2\s \hat{v}(x).
%\end{equation}
%\begin{equation}
%q = v_0 (q_1-q_2), \quad A(x)=  \tau_0^{(n+1)(m-4)/8} \s  \s e^{i\tau_0(m-4) \Theta_1(x)}\s(v_{0,0}^{(1)})^{m-4} \s\chi^m\Big(\frac{|y|}{\delta'}\Big)\s |v_{0,0}^{(1)}|^2\s |v_{0,0}^{(2)}|^2.
%\end{equation}
By Proposition \ref{Gaussian_beam_construction}, see also \eqref{first_term_b_0_0_z}, we have in the normal coordinates that $\Theta_j(0)=0$ and $b_{0}^{(j)}(0)=1$, $j=1,2$. Note also that $\hat{\Theta}(0)=0$. Thus we have
\begin{equation}\label{q_0_z_v}
%q(0) = (q_1-q_2)(0) \text{ and }  
A(0)=\hat{v}(0).
\end{equation}
Integrating in the  normal coordinates, and combining \eqref{mult:product_z} and \eqref{second:expansion_x_z}, we find
\begin{equation}
\begin{aligned}\label{id:gaussian_z_z}
& \int_{ [0,T]\times\Omega} v_0\s(q_1-q_2)\s v_1\s \cdots\s v_m \s \d V_{g} \\
& =\tau^{\frac{n+1}{2}}\s \int_{B(p_0)} v_0(x)\s q(x) A(x)  \s  e^{i \tau \hat{\Theta}(x)} \s e^{-\tau x\cdot \mathcal{H}\s x} \d x   + \int_{B(p_0)} v_0(x)\s q(x) \left( \mathcal{L}_1(x) +  \widehat{\mathcal{L}}_2(x)\right)\d x\\
& = \tau^{\frac{n+1}{2}}\s \int_{B(p_0)} v_0(x)\s q(x) A(x)   \s e^{-\tau x\cdot \mathcal{H}\s x} \d x  +  \int_{B(p_0)} v_0(x)\s q(x) \left( \mathcal{L}_1(x) +  \widehat{\mathcal{L}}_2(x)\right)\d x \\
& \qquad   + \tau^{\frac{n+1}{2}}\s \int_{B(p_0)} v_0(x)\s q(x) A(x)  \s \left( e^{i \tau \hat{\Theta}(x)}-1 \right)\s e^{-\tau x\cdot \mathcal{H}\s x} \d x.
\end{aligned}
\end{equation}
Here $B(p_0)$ is a ball in $\R^{n+1}$ centered at $p_0$ such that $U_1\cap U_2\subset B(p_0)$. We now analyze each term in \eqref{id:gaussian_z_z} above. Thanks to \eqref{eq:estimate_for_optimization_z}, we can control the term on the left-hand side of \eqref{id:gaussian_z_z} in terms of $\delta$, $\eps_1,\ldots,\eps_m$ and the size of $f_j$. The first term after the second equality in \eqref{id:gaussian_z_z} contains information about $q_1-q_2$ and will be analyzed last. At this point, the exponential function $e^{-\tau x\cdot \mathcal{H}\s x}$ will play a crucial role, as it will act as an approximate delta function. This is due to the fact that $\mathcal{H}$ is a positive definite matrix, see \eqref{eq:mathcalH}. By combining \eqref{est:remainder_z} and \eqref{est:remainder_z_z}, and using the fact that both $v_0$ and $q$ are uniformly bounded, we have for $\tau\geq\tau_0$ that 
\begin{equation}\label{z_:z}
\left| \int_{B(p_0)} v_0(x)\s q(x) \left( \mathcal{L}_1(x) +  \widehat{\mathcal{L}}_2(x)\right)\d x\right| \lesssim  \tau^{-1}.
\end{equation}
Making the change of variables $x\mapsto \tau^{-1/2} x$, we obtain
\begin{equation}\label{estz:reg_z}
\begin{aligned}
&\left|\tau^{\frac{n+1}{2}}\s \int_{B(p_0)} v_0(x)\s q(x) A(x)  \s \left( e^{i \tau \hat{\Theta}(x)}-1 \right)\s e^{-\tau x\cdot \mathcal{H}\s x} \d x\right|\\
& \quad = \left| \int_{B(p_0)} (v_0 \s q \s A)( \tau^{-1/2} x)\s\left( e^{i \tau \hat{\Theta}( \tau^{-1/2} x)}-1 \right)\s e^{- x\cdot \mathcal{H}\s x} \d x \right| \lesssim \tau^{-1/2}.
\end{aligned}
\end{equation}
In the last inequality we used that $|\e^{z_1}- \e^{z_2}| \leq |z_1- z_2| \s \e^{\max \left\{ |z_1|, |z_2|\right\}}$ for all  $z_1, z_2 \in \mathbb{C}$ and $\hat{\Theta}(x)=O(|x|^3)$ to deduce that
\[
\left|e^{i \tau \hat{\Theta}( \tau^{-1/2} x)}-1 \right| \leq \tau^{-1/2} |x|^3 \e^{\tau^{-1/2}|x|^3}, \quad \tau\geq \tau_0.
\]
We also used that the functions $v_0 \s q \s A$, $\e^{- x \cdot \mathcal{H}\s x}$, $|x|^3$ and $ \e^{\tau^{-1/2}|x|^3}$ are uniformly bounded in $B(p_0)$. 

Let us then analyze the first term after the second equality in \eqref{id:gaussian_z_z}. Since $\mathcal{H}$ is positive definite, there exists another positive definite matrix $B$ so that $B^2=\mathcal{H}$. Making the change of variables $x\mapsto Bx$, we deduce that
\begin{equation} \label{id:gaussian_z_z_z_1}
\begin{aligned}
&\int_{B(p_0)} v_0(x)\s q(x) A(x)  e^{-\tau x\cdot \mathcal{H}\s x} \d x \\
& \qquad  \qquad  \qquad  = \int_{\R^{n+1}} v_0\s (B\s z)\s q (B\s z)\s A(Bz)\s |g(z)|^{1/2} \s |\det B|^{-1} \s  e^{-\tau |z|^2}  \d z. 
\end{aligned}
\end{equation}
For convenience, we set 
\[
b(z):=  v_0 (B\s z) \s q (B\s z)\s A(Bz)\s |g(z)|^{1/2} \s |\det B|^{-1}.
\] 
By using \eqref{q_0_z_v}, we see that in normal coordinates
%\f{What is the value of $ |\det H\big|_{z=0}|^{-1/2}$. In this case $ |g(0)|=1$. We also need an uniform bound (from above) of  $ |\det H\big|_{z=0}|^{-1/2}$ (for stability).}
\begin{equation}\label{b_0_z}
b(0)=v_0(0)(q_1(0)-q_2(0)) \s \hat{v}(0) \s |\det \mathcal{H}|^{-1/2}.
\end{equation}
The identities \eqref{id:gaussian_z_z} and \eqref{id:gaussian_z_z_z_1}, combined with estimates \eqref{z_:z} and \eqref{estz:reg_z} yield
\[
\left|\left(\frac{\tau}{\pi}\right)^{(n+1)/2}\int_{\R^{n+1}} b(z) e^{-\tau |z|^2} \s \d z\right|\lesssim \tau^{-1/2} + \left|  \int_{ [0,T]\times\Omega} v_0\s(q_1-q_2)\s v_1\s \ldots\s v_m \s \d V_{g}\right|.
\]
Thanks to \eqref{eq:estimate_for_optimization_z}, the second term on the right can be controlled in terms of $\delta$, $\eps_1,\ldots,\eps_m$ and sizes of the functions $f_j$. Thereby, applying Lemma \ref{est:tau_z} with $z_0=0$ and $d=n+1$, we get 
\begin{equation}\label{est:optimising}
\begin{aligned}
 &|b(0)|  \leq \left| b(0)- \left(\frac{\tau}{\pi}\right)^{(n+1)/2}\int_{\R^{n+1}} b(z) e^{-\tau |z|^2} \s \d z\right|\\
& \qquad \qquad   +  \left| \left(\frac{\tau}{\pi}\right)^{(n+1)/2} \int_{\R^{n+1}} b(z) e^{-\tau |z|^2} \s \d z\right| \\
& \lesssim c_{n+1}\s \left\|b \right\|_{C^1}\tau^{-1/2} +\tau^{-1/2} +  \Big[ \delta\s\s \eps_1^{-1}\s \cdots \s \eps_m^{-1} \\
& \qquad \qquad  \qquad   \qquad  +  \eps_1^{-1}\s \cdots \s\eps_m^{-1} \left(\eps_1\Vert  f_1\Vert_{H^{s+1}(\Sigma)}+\cdots +\eps_m\Vert  f_m\Vert_{H^{s+1}(\Sigma)}\right)^{2m-1}\Big] \\
& \lesssim \dfrac{C_{\Omega, m, T, q_j, \chi}\s M}{\kappa_0^{2m-1}}   \left[2\tau^{-1/2} + \dfrac{\kappa_0^{2m-1} \s\delta}{m\s M} \s\s \eps_1^{-1}\s \cdots \s \eps_m^{-1} \right.\\
& \qquad \qquad  \quad  \left. + \frac{1}{m-1} \s\s \eps_1^{-1}\s \cdots \s\eps_m^{-1} \left(\eps_1\Vert  f_1\Vert_{H^{s+1}(\Sigma)}+\cdots +\eps_m\Vert  f_m\Vert_{H^{s+1}(\Sigma)}\right)^{2m-1}\right].
\end{aligned}
\end{equation}
%\[
%\frac{C_\omega}{\kappa^{2m-1}} \big( 2\tau^{-1/2}+ \frac{\kappa^{2m-1}\delta}{m M}\eps^{-m} +\frac{1}{m-1}\eps^{-m} m^{2m-1}\eps^{2m-1}\tau^{(2m-1)(s-n/8+13/8)} \big)
%\]
The above holds for any $M>0$ and $\kappa_0>0$.
%satisfying $\kappa_0^{2m-1}/(mM)<1$.  %  Here $M>0$ and $\kappa_0>0$ will be determined later. 
In the last step, we scaled $\delta$ by $\kappa_0^{2m-1}/(mM)$. The coefficients $2$ and $1/(m-1)$  in front of $\tau^{-1/2}$ and $ \eps_1^{-1}\s \cdots \s\eps_m^{-1}$ in \eqref{est:optimising} were included to simplify formulas later on. We will determine the constants $M$ and $\kappa_0$  later. Their role in obtaining a stability estimate will be clarified in Lemma~\ref{lemma:final_estimate} below.

%%%%%%%%%%%%%%%%%%%%%%%%%%%%%%%%%%%%%%%%%%%%
%%%%%%%%%%%%%%%%%%%%%%%%%%%%%%%%%%%%%%%%%%
%%%%%%%%%%%%%%%%%%%%%%%%%%%%%%%%%%%%%%%%%%%%
%%%%%%%%%%%%%%%%%%%%%%%%%%%%%%%%%%%%%%%%%%
%%%%%%%%%%%%%%%%%%%%%%%%%%%%%%%%%%%%%%%%%%%%
%%%%%%%%%%%%%%%%%%%%%%%%%%%%%%%%%%%%%%%%%%%%%%%%%%%%%%%%%%%%%%%%%%%%%%%%%%%%%%%%%%%%%%
%%%%%%%%%%%%%%%%%%%%%%%%%%%%%%%%%%%%%%%%%%
%%%%%%%%%%%%%%%%%%%%%%%%%%%%%%%%%%%%%%%%%%%%
%%%%%%%%%%%%%%%%%%%%%%%%%%%%%%%%%%%%%%%%%%

\subsection{Step 3: Optimizing the error terms}\label{sec:optimizing_eps_tau}
 The last step of the proof of Theorem \ref{thm:stability} (in this simplified setting) is to choose $\tau$ and $\eps_1, \ldots, \eps_m$ in terms of $\delta$ to have the right hand side of~\eqref{est:optimising} as small as possible.  We begin by setting 
 \[
 \eps_1= \cdots=\eps_m=:\eps.
 \]
Note that by \eqref{f_j_restrcit_z} and \eqref{choice_j_5_m}, we have for $\tau\geq \tau_0$ that
\begin{equation}\label{eq:epsHj}
\begin{aligned}
 \eps\norm{f_j}_{H^{s+1}(\Sigma)}& \sim \eps\s \tau^{s -\frac{n}{8} + \frac{13}{8}},& j=1,2,3,4, &  \qquad \tau\geq \tau_0,\\
  \eps\norm{f_j}_{H^{s+1}(\Sigma)}& \sim \eps\s \tau_0^{s -\frac{n}{4} + \frac{3}{2}},& j=5, \ldots, m. &
 \end{aligned}
\end{equation}
To guarantee the unique solvability of our non-linear wave equation \eqref{eq:nonlinear_equation_lemma}, we require the quantities on the right-hand sides of \eqref{eq:epsHj} are bounded by $\kappa$, which was given by Lemma \ref{lemma:nonlinear-solutions}. Recall that $\tau_0>0$ is a fixed large parameter, which we chose at \eqref{choice_j_5_m}. The parameter was especially chosen so that the Gaussian beams $v_j$ for $j=5,\ldots,m$ have small enough correction terms.

The following Lemma~\ref{lemma:final_estimate} shows how to choose the parameters $\tau$ and $\eps$ in \eqref{est:optimising} optimally given $\kappa>0$ and $\delta\in (0,M)$. By choosing $\kappa_0\leq \kappa$, we will see that the optimal value for $\tau$ is at least $\tau_0$ and we also have that $\eps\norm{f_j}_{H^{s+1}(\Sigma)}\leq \kappa$.

\begin{lemma}\label{lemma:final_estimate}
Let $C,M,s>0$ and $m\in\N$. Let also $\tau_0\geq 1$, $\delta\in(0,M)$ and $\kappa\in(0,1)$.
Then there are $\eps>0$, $\tau\geq \tau_0$ and $\kappa_0\leq \kappa$ such that
\begin{equation}\label{eq:final_estimate}
\begin{aligned}
f(\epsilon, \tau)&:= 2\tau^{-1/2} + \frac{\kappa_0^{2m-1}\s\delta}{m\s M} \epsilon^{-m}+\frac{1}{m-1}\s \eps^{m-1} \, \tau^{(2m-1)(s -\frac{n}{8}+ \frac{13}{8})} \\
& \quad  \leq \displaystyle{C_{s, m, M,\kappa_0} \, \delta^{\frac{8(m-1)}{2m(m-1)(8s-n+13)+2m-1}}}
\end{aligned}
\end{equation}
and we also have
$$
\eps\s \tau^{s- \frac{n}{8} + \frac{13}{8}}\leq C \kappa.
$$
%
%{\color{blue}If $\tau_0>0$ is given, then we can choose $\kappa>0$ so small that $\tau\geq \tau_0$.}
%
%The constant $C_{s,m, M,\kappa}$ is independent of $\delta$.
\end{lemma}
%\begin{lemma}\label{lemma:final_estimate}
%For any given $\delta\in(0,M)$ and $\kappa\in(0,1)$ small enough we find $\eps(\delta,\kappa)=\eps>0$ and $\tau(\delta,\kappa)=\tau\geq 1$ such that
%\begin{equation}\label{eq:final_estimate}
%\begin{aligned}
%f(\epsilon, \tau)&:= 2\tau^{-1/2} + \frac{\kappa^{2m-1}\s\delta}{m\s M} \epsilon^{-m}+\frac{1}{m-1}\s \eps^{m-1} \, \tau^{(2m-1)(s -\frac{n}{8}+ \frac{13}{8})} \\
%& \quad  \leq \displaystyle{C_{s, m, M,\kappa} \, \delta^{\frac{8(m-1)}{2m(m-1)(8s-n+13)+2m-1}}}
%\end{aligned}
%\end{equation}
%and we also have
%%
%$$
%\eps\s \tau^{s- \frac{n}{8} + \frac{13}{8}}\leq \kappa.
%$$
%%
%{\color{blue}If $\tau_0>0$ is given, then we can choose $\kappa>0$ so small that $\tau\geq \tau_0$.}
%%
%The constant $C_{s,m, M,\kappa}$ is independent of $\delta$.
%\end{lemma}

\begin{proof}
To simplify the notation, let us denote $\widehat{s}:=(2m-1)(s-n/8 + 13/8)$ and $\gamma_0=\kappa_0^{2m-1}/M$. We take $\kappa_0\leq \kappa$ to be so that $\gamma_0<1$. We will redefine $\kappa_0>0$ smaller later if necessary. A direct computation shows that
\begin{equation}
\partial_{\eps}f= -(\gamma_0\delta) \eps^{-m-1}+ \eps^{m-1} \tau^{\widehat{s}}, \quad \partial_\tau f=- \tau^{-3/2}+ \frac{\widehat{s}}{m-1}\eps^{m-1} \tau^{\widehat{s}-1}.
\end{equation}
Making $\partial_{\eps}f=\partial_{\tau}f=0$, we obtain the critical points of $f$, namely
\begin{equation}\label{eq:critical1d_z}
\begin{aligned}
\tau&= ({(m-1)\widehat{s}^{\s\, -1}})^{\frac{2(2m-1)}{2\widehat{s}m+ 2m-1}} (\gamma_0\delta)^{-\frac{2(m-1)}{2\widehat{s}m+2m-1}},\\
 \eps& = ({(m-1)\widehat{s}^{\s\s\, -1}})^{\s\s-\frac{2\widehat{s}}{2\widehat{s}m+2m-1}} (\gamma_0\delta)^{\frac{4\widehat{s}m+2m-1-2\widehat{s}}{(2\widehat{s}m+2m-1)(2m-1)}}.
 \end{aligned}
\end{equation}
(One can also verify that the Hessian of $f$ at the critical point is positive definite, hence the critical point is a local minimum.)

Note now that
\begin{align}\label{eq:tau_ineq}
\tau&= ({(m-1)\widehat{s}^{\s\, -1}})^{\frac{2(2m-1)}{2\widehat{s}m+ 2m-1}} (\gamma_0\delta)^{-\frac{2(m-1)}{2\widehat{s}m+2m-1}}\notag\\
&\geq
({(m-1)\widehat{s}^{\s\, -1}})^{\frac{2(2m-1)}{2\widehat{s}m+ 2m-1}} 
\kappa_0^{-\frac{2(m-1)(2m-1)}{2\widehat{s}m+2m-1}},
\end{align}
because by assumption $0<\delta<M$ and since $\gamma_0 =\kappa_0^{2m-1}/M$.
Since the constant
\[
({(m-1)\widehat{s}^{\s\, -1}})^{\frac{2(2m-1)}{2\widehat{s}m+ 2m-1}} >0 
\]
and the exponent
\[
-\frac{2(m-1)(2m-1)}{2\widehat{s}m+2m-1}<0
\]
%in \eqref{eq:tau_ineq} depend only on $m,s>0$ and the dimension $n\geq 2$
do not depend on $\kappa_0$, we may choose $\kappa_0$ so that $\kappa_0<C\kappa$ and that $\tau$ in \eqref{eq:tau_ineq} satisfies
\[
 \tau= ({(m-1)\widehat{s}^{\s\, -1}})^{\frac{2(2m-1)}{2\widehat{s}m+ 2m-1}} (\gamma_0\delta)^{-\frac{2(m-1)}{2\widehat{s}m+2m-1}}\geq\tau_0.
\]
With these choices, we have at the critical point of $f(\eps,\tau)$ given by  \eqref{eq:critical1d_z} that
that 
\[
\eps\s \tau^{s- \frac{n}{8} + \frac{13}{8}} = \eps\s \tau^{\frac{\widehat{s}}{2m-1}} = (\gamma_0\delta)^{\frac{1}{(2m-1)}} = \left(\frac{\kappa_0^{2m-1}}{M} \delta\right)^{\frac{1}{2m-1}} \leq \kappa_0 < C\kappa
\]
for all $0<\delta<M$. A straightforward calculation using \eqref{eq:critical1d_z} shows that  $\tau^{-1/2}$, $(\gamma_0\delta) \eps^{-m}$ and $\eps^{m-1} \tau^{\widehat{s}}$ are all bounded by $C_{s,m,M,\kappa_0}\,  (\gamma_0\delta)^{\frac{m-1}{2\widehat{s}m+2m-1}}$, where the constant $C_{s,m,M,\kappa_0}$ is independent of $\eps$ and $\tau$. This concludes the proof. 
 \end{proof}

Recall the equations \eqref{b_0_z} and \eqref{est:optimising}. We set $\eps_1=\cdots=\eps_m=:\eps$ and apply Lemma \ref{lemma:final_estimate} to obtain
\begin{equation}\label{est:optimising_z_z}
\begin{aligned}
& \left| v_0(p_0)\right|\left| q_1(p_0)- q_2(p_0)\right|\left| \hat{v}(p_0)\right|| \det \mathcal{H}|^{-1/2} \\
& \lesssim \dfrac{C_{\Omega, T, q_j, \chi}M}{\kappa_0^{2m-1}}   \left(2\tau^{-1/2} + \frac{\kappa_0^{2m-1}\s\delta}{m\s M} \epsilon^{-m} +\frac{1}{m-1}\s \eps^{m-1} \, \tau^{(2m-1)(s -\frac{n}{8}+ \frac{13}{8})}\right)\\
& \leq C_0 \delta^{\frac{8(m-1)}{2m(m-1)(8s-n+13)+2m-1}}.
\end{aligned}
\end{equation}
% there exists a future-directed optimal geodesic $\gamma$ from $\Sigma$ to $x$ 
%  and the first intersection of $\gamma$ and $\Sigma$ is transverse.
Since $p_0\in I^-(\Sigma)\cap ([0,T]\times \Omega)$, by Lemma \ref{optimal_geo} there exists a past-directed optimal geodesic from $\Sigma$ to $p_0$ such that the first intersection of the geodesic and $\Sigma$ is transverse. 
% intersects $\Sigma$ for the first time and the intersection is transverse. 
Since the intersection is transverse, the geodesic does not intersect $\{t=T\}$. Therefore, we may choose $v_0$ to be a Gaussian beam corresponding to the geodesic  such that $v_0|_{t=T}=\p_tv_0|_{t=T}=0$.  We may assume by normalizing that $v_0(p_0)=1$. Recall also that $\hat{v}(p_0)>c>0$ and $|\det \mathcal{H}|>0$ by \eqref{est:v_hat_z} and \eqref{eq:mathcalH} respectively. Dividing \eqref{est:optimising_z_z} by the norm of $v_0(p_0)\hat{v}(p_0)| \det \mathcal{H}|^{-1/2}$, we have a stability estimate 
\begin{equation}\label{est:optimising_z_z_div}
 \left| q_1(p_0)- q_2(p_0)\right| \leq C\delta^{\frac{8(m-1)}{2m(m-1)(8s-n+13)+2m-1}}
\end{equation}
at the point $p_0$. We next show that the constant $C$ can be redefined to be independent of $p_0$.

%%%%%%%%%%%%%%%%%%%%%%%%%%%%%%%%%%%%%%%%%%%%%%%%%%%%%%%%%%%%%%%%%%%%%%%%%%%%%%%%%%%%%%%%%%%%%%%%%%%%%%%%%%%%%%%%%%%%%%%%%%%%%%%%%%%%%%%%%%%%%%%%%%%%%%%%%%%%%%%%%%%%%%%%%%%%%%%%%%%%

%%%%%%%%%%%%%%%%%%%%%%%%%%%%%%%%%%%%%%%%%%%%%%%%%%%%%%%%%%%%%%%%%%%%%%%%%%%%%%%%%%%%%%%%%%%%%%%%%%%%%%%%%%%%%%%%%%%%%%%%%%%%%%%%%%%%%%%%%%%%%%%%%%%%%%%%%%%%%%%%%%%%%%%%%%%%%%%%%%%%

%%%%%%%%%%%%%%%%%%%%%%%%%%%%%%%%%%%%%%%%%%%%%%%%%%%%%%%%%%%%%%%%%%%%%%%%%%%%%%%%%%%%%%%%%%%%%%%%%%%%%%%%%%%%%%%%%%%%%%%%%%%%%%%%%%%%%%%%%%%%%%%%%%%%%%%%%%%%%%%%%%%%%%%%%%%%%%%%%%%%

%%%%%%%%%%%%%%%%%%%%%%%%%%%%%%%%%%%%%%%%%%%%%%%%%%%%%%%%%%%%%%%%%%%%%%%%%%%%%%%%%%%%%%%%%%%%%%%%%%%%%%%%%%%%%%%%%%%%%%%%%%%%%%%%%%%%%%%%%%%%%%%%%%%%%%%%%%%%%%%%%%%%%%%%%%%%%%%%%%%%

%%%%%%%%%%%%%%%%%%%%%%%%%%%%%%%%%%%%%%%%%%%%%%%%%%%%%%%%%%%%%%%%%%%%%%%%%%%%%%%%%%%%%%%%%%%%%%%%%%%%%%%%%%%%%%%%%%%%%%%%%%%%%%%%%%%%%%%%%%%%%%%%%%%%%%%%%%%%%%%%%%%%%%%%%%%%%%%%%%%%

%%%%%%%%%%%%%%%%%%%%%%%%%%%%%%%%%%%%%%%%%%%%%%%%%%%%%%%%%%%%%%%%%%%%%%%%%%%%%%%%%%%%%%%%%%%%%%%%%%%%%%%%%%%%%%%%%%%%%%%%%%%%%%%%%%%%%%%%%%%%%%%%%%%%%%%%%%%%%%%%%%%%%%%%%%%%%%%%%%%%

\subsection{Step 4: Uniformity of the constant $C$}\label{sec:uniformity_of_constants}
So far we have obtained the estimate~\eqref{est:optimising_z_z_div} regarding the difference of $q_1$ and $q_2$ at the single point $p_0$. The constant $C$ may at this point depend on $p_0$. Next we argue that the constant $C$ can be redefined to be independent of $p_0$. This will then yield~\eqref{eq:esimate_for_potential_difference} and conclude the proof of Theorem~\ref{thm:stability} in the simplified setting, where we assumed that lightlike geodesics can intersect only once.

To show that $C$ in~\eqref{est:optimising_z_z_div}  can be taken to be independent of $p_0$, we first construct an open cover of 
%$\supp(q_1)\cup \supp(q_2)\subset I^+(\Sigma)\cap I^-(\Sigma)$
$W\subset I^+(\Sigma)\cap I^-(\Sigma)$ 
as follows.  (Recall from  \eqref{eq:recovery_set}  that $W$ is a compact set which we can reach and observe from $\Sigma$.) Let %$z\in \supp(q_1)\cup \supp(q_2)$.
$z\in W$.
 By Lemma~\ref{optimal_geo} there are optimal lightlike geodesics $\gamma_1$ and $\gamma_2$ that intersect at $z$ and which do not intersect $\{t=0\}$. We may reparametrize so that $\gamma_1(0)=\gamma_2(0)=z$. Let $\eps=\abs{\dot\gamma_1(0)-\dot\gamma_2(0)}$. Here and below $\abs{\ccdot}$ denotes the $\R^n$ norm of vectors in local coordinates. 

By Corollary~\ref{uniform_family_of_Gaussian_beams} there are open neighbourhoods $\mathcal{U}_1$ and $\mathcal{U}_2$ of $z$ and families of Gaussian beams $v_\tau(x,l,\ccdot)$ (including the correction term) parametrized by $x\in \mathcal{U}_l$ and $l=1,2$, such that all the implied constants, such as $\tau_0$, in the construction of $v_\tau(x,l,\ccdot)$ are uniform in $x$. Moreover, still by using Corollary~\ref{uniform_family_of_Gaussian_beams}, the geodesics $\gamma_{x,l}$ corresponding to the Gaussian beams $v_\tau(x,l,\ccdot)$ satisfy $\abs{\dot\gamma_l(0)-\dot\gamma_{x,l}(0)}\leq \eps/3$, $l=1,2$. Then, for $x\in \mathcal{U}_1\cap \mathcal{U}_2$, we also have that
\begin{equation}\label{eq:positive_angle}
 \abs{\dot\gamma_{x,1}(0)-\dot\gamma_{x,2}(0)}\geq \eps/3>0.
\end{equation}
We conclude that  the geodesics $\gamma_{x,1}$ and $\gamma_{x,2}$ intersect at $x$ and do not have the same graphs. 
%hus, the geodesics $\gamma_{x,1}$ and $\gamma_{x,2}$ intersect at $x$ and do not have the same graphs. 
We also set
\[
 \hat v_x(\ccdot)=(v_\tau(x,l, \ccdot))^{m-4}|_{\tau=\tau_0, l=1}
\]
for $x\in \mathcal{U}_1\cap \mathcal{U}_2$. By redefining $\tau_0$ larger, if necessary, we have that $\abs{\hat v_x(x)}\geq d>0$ for all $x\in \mathcal{U}_1\cap \mathcal{U}_2$.

%\dfrac{}{} for all $x\in \mathcal{U}_1\cap \mathcal{U}_2$.

%We cover $\supp(q_1)\cap \supp(q_2)$ by the sets $U_1$ and $U_2$ described above. 
%We may repeat the arguments in this Section~\ref{est:optimising_z_z} to have that in each of the sets $U_1\cap U_2$ described above, we have that the constant $C$ is uniform for . 

In deriving~\eqref{est:optimising_z_z_div} in this Section~\ref{sec:proof_of_stabilit_estimate}, we used normal coordinates. Normal coordinates are uniquely defined by choosing an orthonormal basis at a point. By using a local orthonormal frame on a neighbourhood $\mathcal{U}_3$ of $z$, we may find a family of normal coordinates smoothly parametrized by $x\in \mathcal{U}_3$. It follows that the contribution to $C$ in~\eqref{est:optimising_z_z_div} coming from the use of normal coordinates may be taken to be uniformly bounded for all $x\in \mathcal{U}_3$. All things considered, by repeating the arguments in this Section~\ref{sec:proof_of_stabilit_estimate}, we may take the constant $C$ to be uniform for all $x\in \mathcal{U}_1\cap \mathcal{U}_2\cap \mathcal{U}_3$, where $\mathcal{U}_1\cap \mathcal{U}_2\cap \mathcal{U}_3$ is a neighbourhood of $z$.

%Since $q_1$ and $q_2$ are admissible by assumption, we have that $\supp(q_1)\cup \supp(q_2)$ is compactly contained in $I^+(\Sigma)\cap I^{-}(\Sigma)$. Thus, by covering first the compact set $\overline{\supp(q_1)\cup \supp(q_2)}$ by the sets $\mathcal{U}_1\cap \mathcal{U}_2\cap \mathcal{U}_3$ as described above and then passing to a finite subcover, we have that~\eqref{est:optimising_z_z} holds for all $z\in \supp(q_1)\cup \supp(q_2)$.
Recall that we aim to estimate the difference of $q_1$ and $q_2$ in the compact set $W \subset I^+(\Sigma)\cap I^{-}(\Sigma)$. By covering first the compact set $W$ by the sets $\mathcal{U}_1\cap \mathcal{U}_2\cap \mathcal{U}_3$ as described above and then passing to a finite subcover, we have that~\eqref{est:optimising_z_z_div} holds for all $z\in W$.
Finally, we apply Lemma~\ref{lem:separation_filter} with $P=1$ to deduce that there is a finite family of functions $v_{z,0}$ satisfying $\square_g v_{z,0}=0$ in $[0,T]\times \Omega$ and $v_{z,0}|_{t=T}=\p_t v_{z,0}|_{t=T}=0$ and such that $|v_{z,0}(z)|\geq c>0$. (Only finitely many of the functions $v_{z,0}$ are actually distinct.) Combining everything yields the estimate
\begin{equation}\label{eq:uniform_estimate_with_H}
\begin{aligned}
&  \left| (v_{z,0}(z)\hat v_z(z)(q_1-q_2))(z)\right|| \det \mathcal{H}_{z}|^{-1/2} \leq C\delta^{\frac{8(m-1)}{2m(m-1)(8s-n+13)+2m-1}},
\end{aligned}
\end{equation}
which holds for all $z\in W$.
Here the point $z$ corresponds to the origin $0$ of normal coordinates centered at $z$ and all the quantities are expressed in these coordinates. The point $z$ is also the point where the geodesics $\gamma_{z,1}$ and $\gamma_{z,2}$ corresponding to the Gaussian beams $v_\tau(z,1,\ccdot)$ and $v_\tau(z,2,\ccdot)$ intersect. 

By Remark \ref{remove_hat_v_z}, we have that $|v_{z,0}(z)|\geq c>0$ and hence $\abs{\hat v_z(z)}\geq d>0$ in~\eqref{eq:uniform_estimate_with_H}. Let us estimate $\abs{\det \mathcal{H}_{z}}$ where
\[
 \mathcal{H}_z=2\nabla^2\text{Im}(\Theta_{z,1}(x)+\Theta_{z,2}(x))\big|_{x=z}.
\]
Here $\Theta_{z,1}$ and $\Theta_{z,2}$ are the phase functions corresponding to the Gaussian beams $v_\tau(z,1,\ccdot)$ and $v_\tau(z,2,\ccdot)$ respectively. Here also $\nabla^2$ is the invariant Hessian. In the normal coordinates centered at $z$ we have that the geodesics $\gamma_{z,1}$ and $\gamma_{z,2}$ are rays emanating in from origin. Since $\gamma_{z,1}$ and $\gamma_{z,2}$ do not have the same graphs, the rays are not same and there is a positive angle (in $\R^{n+1}$ metric) between the rays in the normal coordinates. Due to~\eqref{eq:positive_angle}, the angle is uniformly bounded from below by a positive constant. Consequently, using also the facts that 
\[
 \mathrm{Im}(\nabla^2 \Theta_{z,l})(z) \geq 0, \quad \mathrm{Im}(\nabla^2 \Theta_{z,l})(z)|_{\dot{\gamma}_{z,l}(0)^{\perp}} > 0
\]
we conclude that there is $h>0$ such that $\abs{\det \mathcal{H}_{z}}>h$ for all 
%$z\in \supp(q_1)\cup \supp(q_2)$.
$z\in W$. 
Dividing~\eqref{eq:uniform_estimate_with_H} by $|v_{z,0}(z)|$, $\abs{\hat v_x(x)}$ and $\abs{\det \mathcal{H}_{z}}^{-1/2}$, and redefining $C$ larger, if necessary, concludes the proof in the special case where we assumed that lightlike geodesics intersect can  only once. 
  
\subsection{Step 5: Multiple intersections}\label{sec:multiple_intersections}
We have proven Theorem~\ref{thm:stability} in the special case, which assumed that the used lightlike geodesics intersect only once. In the case of multiple intersections, we can perform a similar analysis as in the special case, but this leads to  an estimate for a sum of terms regarding the difference $q_1-q_2$ at the intersection points. To separate the contributions coming from several intersection points, we will use separation matrices and a separation filter introduced in Lemma~\ref{lemma:separation_of_multiple_points} and Lemma~\ref{lem:separation_filter}. Most of the work needed to handle the case of several intersection was already done in proving these two lemmas. By Lemma~\ref{lemma:geodesics_intersections} we know that there is $P\in \N$ such that lightlike geodesics can intersect at most $P$ times in $[0,T]\times \Omega$. 

%Let $p_0\in W$ and l
Let $\gamma_1$ and $\gamma_2$ be future-directed lightlike geodesics starting from $\Sigma$ that intersect for the first time at $z$ and which do not intersect $\{t=0\}$.
%We do not assume that $\gamma_1$ and $\gamma_2$ only at $z$.
Let
\[
 z_1,\ldots,z_{P_0}
\]
be the intersection points of $\gamma_1$ and $\gamma_2$ arranged as $z_1\leq z_2\leq \cdots\leq z_{P_0}$, where $P_0\leq P$ and 
\[
 z=z_1.
\]
As in~\eqref{sol:decay_tau_z}, we choose  
\begin{equation*}
 v_{j}=\tau^{1/8}(v_{\tau, j}+r_j), %\quad \text{and} \quad f_j= v_{j}|_{\Sigma}, 
 \quad j=1,2,
\end{equation*}
to be Gaussian beams associated to $\gamma_1$ and $\gamma_2$. We also choose 
\[
 v_j=\overline{v}_{j-2}, \quad j=3,4, \ \text{ and } \hat v =(v_1|_{\tau=\tau_0})^{m-4}
\]
as before. Since the product $v_1\s \cdots\s v_m$ is supported on neighbourhoods of the intersection points, the term 
\[
\langle v_0(q_1-q_2), v_1\s \cdots\s v_m \rangle_{L^2([0,T]\times\Omega)}= \int_{[0,T]\times \Omega} v_0(q_1-q_2) \s v_1\s \cdots\s v_m dV_g,
\]
%on the right hand side of~\eqref{eq:estimate_for_optimization_z} 
becomes a sum of terms
\begin{equation}\label{eq:sum_of_contributions}
 \sum_{j=1}^{P_0}\tau^{\frac{n+1}{2}} \s \int_{B(z_j)} v_0(x)\s (q_1-q_2)(x) A(x)  \s  e^{i \tau \hat{\Theta}(x)} \s e^{-\tau x\cdot \mathcal{H}_{z_j}\s x} \d V_{g},
\end{equation}
where each set $B(z_j)$ is a neighbourhood of $z_j$, $j=1,\ldots,P_0$. Here 
$\hat{\Theta}(x)$ and $A(x)$ are defined similarly to~\eqref{eq:def_of_hatTheta} and \eqref{eq:def_of_A} respectively and 
\[
 \mathcal{H}_{z_j}=2\nabla^2\text{Im}(\Theta_1(x)+\Theta_2(x))\big|_{x=z_j}, \quad j=1,\ldots,P_0
\]
as before. 
%By redefining $C$ larger, if necessary, the estimate~\eqref{mult_ine_z} hold for any finite number of functions $v_0$ satisfying $\square_gv_0=0$ with $v_0|_{t=T} =\p_t v_0|_{t=T} = 0$ in $\Omega$. 

Let $\mathcal{M}=\{f_k\}_{k\in\mathcal{K}}$ be a separation filter of $[0,T]\times \Omega$ given by Lemma~\ref{lem:separation_filter} with the compact set 
%$K=\supp(q_1)\cup \supp(q_2)$
$W$ and $P_0$ as $P$. Here $f_k\in C^\infty(\Sigma)$ and $\mathcal{K}$ is a finite index set. According to Lemma~\ref{lem:separation_filter}, the corresponding solutions $v_{f_k}$ to $\square_gv=0$ in $[0,T]\times \Omega$ can be chosen so that the associated separation matrix $(v_{f_k}(z_j))_{k,j=1}^{P_0}$ is invertible. By repeating the calculation in~\eqref{eq:estimate_for_optimization_z} %and optimizing $\tau$ and $\eps$ as in Section~\ref{sec:optimizing_eps_tau} 
we have for each $k\in \mathcal{K}$ that
\begin{align*}
\abs{\langle v_{f_k}(q_1-q_2), v_1\s \cdots\s v_m &\rangle_{L^2([0,T]\times\Omega)}} \\ %\leq  \sum_{k=1}^{P_0}\left|\int_{[0,T]\times \Omega} v_{f_k}(q_1-q_2), v_1\s \cdots\s v_m dV_g\right| \\
&\leq C_k  \,(\eps_1 \s\cdots \s\eps_m)^{-1}\, \left[ \delta  + \left( \sum_{j=1}^m \varepsilon_j \Vert  f_j\Vert_{H^{s+1}(\Sigma)}\right)^{2m-1} \right]. % \leq C\delta^{\frac{8(m-1)}{2m(m-1)(8s-n+13)+2m-1}}.
\end{align*}
We apply~\eqref{eq:sum_of_contributions} with $v_{f_k}$ in place of $v_0$ and note that the integrals in~\eqref{eq:sum_of_contributions} are the value of the integrand at $z_k$ plus a term of size $O(\tau^{-1/2})$ by calculations \eqref{eq:mathcalH}--\eqref{est:optimising} and Lemma \ref{est:tau_z}. Optimizing as in Section~\ref{sec:optimizing_eps_tau} in $\tau$ and $\eps_1,\ldots,\eps_m$ yields that
\[
 \left| \sum_{j=1}^{P_0}v_{f_k}(z_j)(q_1(z_j)- q_2(z_j))\hat{v}(z_j) |\det \mathcal{H}_{z_j}|^{-1/2}\right|\leq C\delta^{\frac{8(m-1)}{2m(m-1)(8s-n+13)+2m-1}}
\]
for all $k=1,\ldots,P_0$. Let us define a matrix $A$ and a vector $\mathcal{Q}$ as
\[
A_{kj}= v_{f_k}(z_j), \quad \mathcal{Q}_j=(q_1(z_j)- q_2(z_j))\hat{v}(z_j) |\det \mathcal{H}_{z_j}|^{-1/2},
\]
where $j,k=1,\ldots,P_0$.
Since the separation matrix $\{v_{f_k(x_j)}\}_{k,j=1}^{P_0}$ is invertible, we have that 
\begin{align*}
 \abs{\mathcal{Q}_1}\leq \norm{\mathcal{Q}}=\norm{A^{-1} (A \mathcal{Q} )}\leq \norm{A}^{-1} \norm{A \mathcal{Q}}
\leq \norm{A}^{-1}  \s C\delta^{\frac{8(m-1)}{2m(m-1)(8s-n+13)+2m-1}}.
\end{align*}
Recalling that $z_1=z$,  we thus have 
\begin{equation}\label{eq:estim_with_A}
 \left|(q_1(z)- q_2(z))\hat{v}(z) |\det \mathcal{H}_z|^{-1/2}\right|\leq C\norm{A}^{-1}\delta^{\frac{8(m-1)}{2m(m-1)(8s-n+13)+2m-1}}.
\end{equation}

In~\eqref{eq:estim_with_A}, $\hat v_z$, $\det \mathcal{H}_z$, but also $\norm{A}^{-1}$ depend on the point $z$. We argued in Section~\ref{sec:uniformity_of_constants} that $\hat v_z$, $|\det \mathcal{H}_z|^{-1/2}$ have norms, which are uniformly bounded from below with respect to $z$. Since the separation filter $\mathcal{M}$ is a finite collection, we may also bound $\norm{A}^{-1}$ uniformly when we consider different points in 
%$\supp(q_1)\cup \supp(q_2)$.
$W$.
 By using these facts and by dividing by $|\hat{v}(z) \det \mathcal{H}_z|^{-1/2}$ and redefining $C$ shows that 
\[
% \norm{q_1- q_2}_{L^\infty([0,T]\times \Omega)}\leq C\delta^{\frac{8(m-1)}{2m(m-1)(8s-n+13)+2m-1}}.
 \norm{q_1- q_2}_{L^\infty(W)}\leq C\delta^{\frac{8(m-1)}{2m(m-1)(8s-n+13)+2m-1}}.
\]
This concludes the proof of Theorem~\ref{thm:stability}.

\appendix
\section{A bound on number of intersections}\label{app:geometry}
The following lemma shows that given a compact set $K\subset N$ of a globally hyperbolic Lorentzian manifold there is a uniform bound on the number of possible intersections of pairs of causal geodesics in $K$.
In our particular application we will apply this lemma with $K=[0,T]\times \Omega$ and $N=\R\times M$.
\begin{lemma}\label{lemma:geodesics_intersections}
Let $(N,g)$ be a globally hyperbolic Lorentzian manifold and let $K\subset N$ be a compact set. There is $P\geq 1$ with the following property.
%
%There exists a positive integer $P$ such that the following holds.
%
%Let $\{\gamma_j\}_{j\in I}$ be any collection of causal geodesics in $K$ whose graphs do not coincide.
%
Let $\gamma_1$ and $\gamma_2$ be two causal geodesics.
Then the number of intersection points of $\gamma_1$ and $\gamma_2$ is bounded by $P$, that is,
$$
\#(\gamma_1\cap \gamma_2) \leq P.
$$
%$$
%\#(\gamma_i\cap \gamma_j) \leq P, \quad i\neq j.
%$$
\end{lemma}
\begin{proof}
Every point $p\in N$ has arbitrarily small convex normal neighbourhoods $U_p$, whence any two distinct geodesics can intersect at most once in these sets.
As $N$ is globally hyperbolic, $U_p$ can be taken causally convex, see e.g. \cite{Min19}.
Because $K$ is compact, there exists a finite cover
$$
\bigcup_{j=1}^P U_{p_j} \supset K
$$
formed of sets $U_{p_j}$.
As the sets $U_{p_j}$ are causally convex, a geodesic leaving $U_{p_j}$ never returns to $U_{p_j}$.
Therefore a pair of geodesics in $K$ can intersect at most once in each of the $P$ open sets.
\end{proof}

\section{Proof of Proposition~\ref{thm:energy}}\label{app:forward}

Before proceeding to the proof of Proposition~\ref{thm:energy}, which concerns the well-posedness of the linear wave equation \eqref{wave-eq_with_data}, we need the following lemma.
\begin{lemma}\label{lem:epsilon_ngbh}
Let $(\R\times M, g)$ be globally hyperbolic manifold. Let also $t_0\in \R$ and let $S_{t_0} = \{t=t_0\}\times M$ be the corresponding Cauchy surface.
Suppose $V\subset S_{t_0}$ is a compact set in $S_{t_0}$ and $W$ is an open neighbourhood of $V$ in $\R\times M$.
Then there exists $\eps>0$ such that $([t_0,t_0+\eps]\times M) \cap J^+(V) \subset W$. Especially if $V\Subset U$, where $U$ is open in $S_{t_0}$, there exists $\eps>0$ such that $([t_0,t_0+\eps]\times M) \cap J^+(V) \subset [t_0,t_0+\eps]\times U$.
\end{lemma}
\begin{proof}
For the first claim, assume that there is no such $\eps>0$.
Then there are numbers $\eps_k>0$ with $\eps_k\to 0$ as $k\to\infty$ and points $p_{k}\in ([t_0,t_0+\eps_k]\times M) \cap J^+(V)$, but $p_k\not\in W$.
Since $W$ is open, any accumulation points of $p_k$, if they exist, are not in $W$.
As $\eps_k\to 0$ there is $\eps \geq \eps_k$ for all sufficiently large $k\in\N$, say, $k\geq k_0$.
It follows that $p_{k}\in ([t_0,t_0+\eps]\times M) \cap J^+(V)$ for all $k\geq k_0$.
Because $\R\times M$ is foliated by the space-like Cauchy surfaces $S_t$, we have 
$$
[t_0,t_0+\eps]\times M = \bigcup_{t\in [t_0,t_0+\eps]} S_t.
$$
Also $S_t \subset J^-(S_T)$ for all $t\leq T$, because if $\gamma$ is any inextendible future-directed causal curve with $\gamma(s)\in S_t$ for some $s\in\R$, then
% by definition 
this curve intersects $S_T$ in the future.
By \cite[Corollary A.5.4]{BGP}, the intersection $J^-(S_{t_0+\eps})\cap J^+(V)$ is compact.
So $[t_0,t_0+\eps]\times M$ being a closed subset of $J^-(S_{t_0+\eps})$ implies that $([t_0,t_0+\eps]\times M)\cap J^+(V)$ is compact and there exists a convergent subsequence $p_{k_i}\to p\in([t_0,t_0+\eps]\times M) \cap J^+(V)$.
Due to the construction, as $\eps_{k_i} \to 0$ we have $p_{k_i}\to p\in \{t=t_0\}\times M \cap J^+(V) = V\subset W$. Thus $p\in W$, which is a contradiction.
%we must have %$\tau(p) = t_0$ and hence $p\in V\subset U$, a contradiction.
%

%
Suppose now that $W= (a,b)\times U$ where $t_0\in (a,b)\subset\R$.
Then if $\eps>0$ is so small that $([t_0,t_0+\eps]\times M)\cap J^+(V)\subset (a,b)\times U$, we have $([t_0,t_0+\eps]\times M)\cap J^+(V)\subset [t_0,t_0+\eps]\times U$.
If not, we would have some $p = (t,x)\in ([t_0,t_0+\eps]\times M)\cap J^+(V)$ with $t\not\in [t_0,t_0+\eps]$ or $x\not\in U$.
Both options are invalid, so also the second claim holds.
\end{proof}

\begin{proof}[Proof of Proposition~\ref{thm:energy}]
%
%{\color{red}This result is well-known folklore, but we could not find a reference to the estimates on manifolds, so we give a sketch of the proof below.}

Let us first recall results in the special case where $\Omega$ is a domain $\Omega\subset \R^n$.
From \cite{LLT86} we know that there exists a unique solution $v\in E^s$ to the problem
\begin{equation}\label{eq:lasiecka}%see LLT86 remark 2.10 on page 167 (19)
\begin{cases}
(\p_t^2 - \Delta_h) v = F, &\text{in } [0,T]\times \Omega,\\
v=f, &\text{on } [0,T]\times \p\Omega, \\
v=u_0,\, \p_t v = u_1,&\text{in } \{t=0\}\times \Omega,
\end{cases}
\end{equation}
if $h(t,\ccdot)$ is a smooth $1$-parameter family of Riemannian metrics on $\R^n$ and if we assume that $F$, $f$, $u_0$ and $u_1$ satisfy the regularity and compatibility conditions of our proposition in $\R^n$. Under these assumptions, 
we also know from classical results such as \cite{Ikawa} that there exists a unique solution $w\in E^{s+1}$ to
\begin{equation}\label{eq:Ikawa}%see Ikawa Theorem 2 on page 593 (14)
\begin{cases}
(\p_t^2 - \Delta_h) w + Aw= G, &\text{in } [0,T]\times \Omega,\\
w=0, &\text{on } [0,T]\times \p\Omega, \\
w= \p_t w = 0,&\text{in } \{t=0\}\times \Omega
\end{cases}
\end{equation}
when $A\in C^\infty([0,T]\times \Omega)$ %
and $G\in E^s$. By combining the mentioned results, we have that the problem
\begin{equation}\label{eq:wave_eq_in_Rn}
\begin{cases}
(\p_t^2 - \Delta_h) u + Au= F, &\text{in } [0,T]\times \Omega,\\
u=f &\text{on } [0,T]\times \p\Omega, \\
u=u_0,\, \p_t u = u_1,&\text{in } \{t=0\}\times \Omega
\end{cases}
\end{equation}
has a unique solution $u\in E^{s+1}$ and the regularity results of \cite{Ikawa,LLT86} also show that $\p_\nu u\in H^s([0,T]\times \p\Omega)$. Indeed, by solving first \eqref{eq:lasiecka} for $v\in E^{s+1}$ and then defining $G:= Av\in E^{s+1}$ for the problem \eqref{eq:Ikawa} we find $w\in E^{s+1}$ (in fact $w\in E^{s+2}$) solving \eqref{eq:Ikawa} and so that $u:= v-w$ solves~\eqref{eq:wave_eq_in_Rn}.
% 
% there 
% %
% Now, solving first \eqref{eq:lasiecka} for $v\in E^1$ and then taking $G:= Av\in C([0,T];L^2(\Omega))$ in \eqref{eq:Ikawa} we find $w\in E^1$ (in fact $w\in E^2$) solving \eqref{eq:Ikawa}, which implies that the function $u:= v-w$ solves
% \begin{equation}\label{eq:wave_eq_in_Rn}
% \begin{cases}
% (\p_t^2 - \Delta_h) u + Au= F, &\text{in } [0,T]\times \Omega,\\
% u=f &\text{on } [0,T]\times \p\Omega, \\
% u=u_0,\, \p_t u = u_1,&\text{in } \{t=0\}\times \Omega.
% \end{cases}
% \end{equation}
% Here the fact that $v,w\in E^1$ implies $u\in E^1$ and the regularity results of \cite{Ikawa,LLT86} show that $\p_\nu u\in L^2([0,T]\times \p\Omega)$.
%

Let us then explain how these results translate to the case of a globally hyperbolic manifold $[0,T]\times M$ equipped with a Lorentzian metric $g = \beta(t,x)dt^2 - h(t,x)$. Here $\beta>0$ is a smooth function and $h(t,\ccdot)$ is a smooth $1$-parameter family of Riemannian metrics on $M$.
The function  $\beta>0$ is bounded from above and below by compactness of $[0,T]\times\Omega$. % is compact and $\beta>0$ is smooth.
Via a conformal change of variables we obtain a scaled metric $\tilde g = dt^2-\beta^{-1}h$ for which the wave operator transforms as
\[
\mathcal{P}:=\beta^\frac{3}{2} \square_g \beta^{-\frac{1}{2}}
=
\square_{\tilde g} + V = \p_t^2 - \Delta_{\beta^{-1}h} + V. 
\]
Here $V(t,x)$ is a smooth function and $\Delta_{\beta^{-1}h}$ for each $t\in [0,T]$ is the Laplace-Beltrami operator of the Riemannian metric $(\beta^{-1}h)(t,\ccdot)$ on $M$.
Then $u$ solving \eqref{wave-eq_with_data} is equivalent to $v:=\beta^{\frac{1}{2}}u$ solving
\begin{equation}\label{eq:wave_eq_for_v}
\begin{cases}
\mathcal{P} v = \beta^{\frac{3}{2}}F, &\text{in } [0,T]\times \Omega,\\
v=\beta^{\frac{1}{2}}f, &\text{on } \Sigma, \\
v=\beta^{\frac{1}{2}}u_0,\, \p_t v = \frac{1}{2}\beta^{-\frac{1}{2}}\p_t\beta u_0 + \beta^{\frac{1}{2}}u_1,&\text{in } \{t=0\}\times \Omega.
\end{cases}
\end{equation}
From \cite[Theorem 24.1.1]{Hormander3} we know that there exists a unique solution to \eqref{eq:wave_eq_for_v}. (The result of \cite{Hormander3} is not however sufficient to us.) Also, in local coordinates in $\Omega$ this equation is of the form \eqref{eq:wave_eq_in_Rn}. %~\f{Do we need this anymore, as we are constructing the solution by hand?}

Let us denote 
\[
R = \beta^{\frac{3}{2}}F, \quad r = \beta^{\frac{1}{2}}f,  \quad r_{0} = \beta^{\frac{1}{2}}u_0, \quad r_1 = \frac{1}{2}\beta^{-\frac{1}{2}}\p_t\beta u_0 + \beta^{\frac{1}{2}}u_1.
\]
Note that $\{t=0\}\times M$ is a space-like Cauchy surface in $\R\times M$.
Because $\Omega\subset M$ is a compact manifold, there exists a finite atlas $\{(U_j,\varphi_j)\}_{j=1}^k$ covering $\Omega$.
Let $\chi_j$ be a partition of unity subordinate to $\{U_j\}_{j=1}^k$ and let us denote the support of $\chi_j$ as
\[
V_j=\supp(\chi_j)\Subset U_j.
\]
%$ denote the support of $\chi_j$.
%
Let us also denote
\[
R_j=\chi_j R, \quad r_j = \chi_j\big|_{\Sigma} r, \quad r_{0,j} = \chi_j r_0, \quad r_{1,j} = \chi_j r_1,
\]
denote the corresponding coordinate representations as
\[
 %\tilde u_j = u\circ\varphi_j^{-1}, \quad 
 \tilde R_j = R_j\circ \varphi_j^{-1}, \quad \tilde r_j=r\circ \varphi_j^{-1}, \quad \tilde r_{0,j}=r_0\circ \varphi_j^{-1} \quad \tilde r_{1,j}=r_1\circ \varphi_j^{-1},
\]
and denote
\[
\tilde U_j = \varphi_j(U_j).
\]
% , $\tilde F_j = F_j\circ \varphi_j^{-1}$, and similarly for $\tilde f_j$, $\tilde u_{0,j}$ and $\tilde u_{1,j}$.
%

%
We construct a solution to~\eqref{wave-eq_with_data} by patching up local solutions following partly the proof of~\cite[Proposition 3.2.11]{BGP}.
%by setting
%\[
% u=\sum u_j,
%\]
%where $u_j$ are constructed by solving in small sets~\eqref{eq:wave_eq_for_v}. 
%
As we will see, this is possible due to the finite speed of propagation of solutions to a wave equation. 
Let $K_j$ be an open set with compact closure, such that $V_j\subset K_j$ and $\overline{K_j}\subset U_j$.
%(this is possible, because $M$ is locally compact and Hausdorff, see~\cite[Theorem 2.7]{Rudin}).
%
If $t\in \R$, we may use Lemma~\ref{lem:epsilon_ngbh} to deduce that there exists $\eps>0$ so that 
\[
 \big((t,t+\eps)\times \Omega \big)\cap J^+(V_j)\subset (t,t+\eps)\times K_j\subset (t,t+\eps)\times U_j
\]
%\big((t,t+\eps)\times \Omega \big)\cap J_{\tilde g}^+(V_j)\subset (t,t+\eps_t)\times U_j$ 
holds. (This is similar to \cite[proof of Proposition 3.2.11]{BGP}.) Here $J^+$ is defined with respect to the conformal metric $\tilde g$. We remark that $J^+$ of a set is conformally invariant.
By compactness of $[0,T]$, there is a finite set of numbers $\eps_i>0$ and $t_i\in\R$ so that the intervals 
\[
I_i:=(t_i,t_i+\eps_i)
\]
cover $[0,T]$. 
We are going to find a solution to our wave equation~\eqref{wave-eq_with_data} iteratively in the index $i$ so that at each step of the iteration we have $\big(I_i\times \Omega \big)\cap J^+(V_j)\subset I_i\times U_j$, $j=1,\ldots, k$.  
Let us set $t_1=0<t_2<\cdots<t_l$ and $t_l+\eps_l = T$ and consider first the set $\big((0,\eps_1)\times \Omega \big)\cap J^+(V_j)$ first.
%
% Denote for brevity
% \[
% \tilde U_j := \varphi_j(U_j).
% \]
%
%If $(U_j,\varphi_j)$ is a boundary chart such that $\{x_n=0\}$ corresponds to points on $\p \Omega$, let us also denote $\Gamma_j=\tilde U_j\cap \{x_n=0\}\subset \R^+_n$. % denote the boundary in $\R^3_+$.
%

%
%In local coordinates the problem~\eqref{eq:wave_eq_for_v} is of the form \eqref{eq:wave_eq_in_Rn}. 
By the discussion around~\eqref{eq:wave_eq_in_Rn}, we have that there is a unique solution $\tilde u_j\in E^{s+1}$ to
%and the finite speed of propagation of the solutions to the wave equation to solve
\begin{equation}\label{eq:linear_wave_eq_in_Rn}
\begin{cases}
\tilde{\mathcal{P}}\s \tilde u_j = \tilde R_j, &\text{in } (0,\eps_1)\times \tilde U_j,\\
\tilde u_j=\tilde r_j, &\text{on } (0,\eps_1)\times \p \tilde U_j\cap \varphi_j(\p \Omega) , \\
\tilde u_j=0, &\text{on } (0,\eps_1)\times\p \tilde U_j\setminus\varphi_j(\p \Omega), \\
\tilde u_j=\tilde r_{0,j},\, \p_t \tilde u_j = \tilde r_{1,j},&\text{in } \{t=0\}\times \tilde U_j.
\end{cases}
\end{equation}
in each coordinate chart $\tilde U_j$, $j=1,\ldots,k$, in the time interval $(0,\eps_1)$. (Here and below we understand $\varphi_j(\p \Omega)=\emptyset$ if $U_j\cap \p \Omega=\emptyset$.)
Since our equation~\eqref{wave-eq_with_data} satisfies the compatibility conditions~\eqref{eq:compatibility}, one can verify by a direct calculation that \eqref{eq:linear_wave_eq_in_Rn} satisfies the compatibility conditions of~\cite{Ikawa, LLT86} that were needed for unique solvability  \eqref{eq:wave_eq_in_Rn}. In particular, at the intersection of $\{t=0\}$ and $\p \tilde U_j\cap \varphi_j(\p \Omega)$ the compatibility conditions follow from the assumptions of the proposition we are proving. At the intersection of $\{t=0\}$ and a neighbourhood of $\p \tilde U_j\setminus\varphi_j(\p \Omega)$ the initial values vanish due to the cut-off functions $\chi_j$.  Thus~\eqref{eq:linear_wave_eq_in_Rn} has a unique solution. 
% one can verify that equation \eqref{eq:linear_wave_eq_in_Rn} also satisfies the compatibility conditions \eqref{eq:compatibility}.
%

Next, let us define
% We define
% \[
%  u_j:=\tilde{u}_j\circ \phi_j
% \]
% smoothly by zero to whole $\Omega$.
\begin{equation}
u_j=
\begin{cases}
\tilde{u}_j\circ \varphi_j &\text{ in } [0,\eps_1]\times U_j,\\
0 &\text{ in } [0,\eps_1]\times (\Omega\setminus U_j).
\end{cases}
\end{equation}
By the finite speed of propagation of solutions to a wave equation, see for example~\cite[Proposition 3.2.11]{BGP}, we have $\mathrm{supp}(u_j)\subset J^+(V_j)$,
and by the condition $\big((0,\eps_1)\times \Omega \big)\cap J^+(V_j)\subset (0,\eps_1)\times K_j\subset (0,\eps_1)\times U_j$, we have that
\[
 \tilde{u}_j=0 \text{ in a neighbourhood of } \p \tilde U_j\setminus\varphi_j(\p \Omega).
\]
Consequently, $u_j$ is the smooth continuation of $\tilde{u}_j\circ \varphi_j:U_j\to \R$ by zero and $u_j\in E^{s+1}$. We also continue $\tilde u_j$ smoothly by zero to $\R^n$ (or to $\R_n^+$ if $U_j$ is a boundary chart.)

We now patch up the functions $u_j$ as 
\[
u=\sum_{j=1}^k u_j\in E^{s+1}
\]
to have a solution to our equation~\eqref{eq:wave_eq_for_v} in the case $T=\eps_1$. Indeed, we have on $\big((0,\eps_1)\times U_j\big)$ that 
\[
 \mathcal{P}u=\sum_{j=1}^k (\tilde{\mathcal{P}}\tilde u_j)\circ \varphi_j=\sum_{j=1}^k\tilde{R}_j\circ \varphi_j=\sum_{j=1}^k\chi_jR=R.
\]
We also have that
% 
% %
% %This way we obtain a solution $\tilde u_j\in E^1$ to \eqref{eq:linear_wave_eq_in_Rn} in each coordinate chart $\tilde U_j$ for a short time-interval $(0,\eps_1)$.
% %
% Since $\tilde u_j=0$ in a neighbourhood of $\p \tilde U_j\setminus\{x_3=0\}$ we can extend $\tilde u_j$ smoothly by zero to all of $\Omega$.
% %
% % % % % Here $u_j=\tilde u_j \circ \phi_j$ for $x\in U_j$, where $u_j$ now solves
% % % % % \begin{equation}
% % % % % \begin{cases}
% % % % % \tilde\square_{\tilde g} u_j = F_j, &\text{in } [0,\eps_1]\times U_j,\\
% % % % % u_j=f_j, &\text{on } [0,\eps_1]\times U_j\cap \p \Omega, \\
% % % % % u_j=0, &\text{on } [0,\eps_1]\times \p U_j\setminus \p \Omega, \\
% % % % % u_j=u_{0,j},\, \p_t u_j = u_{1,j},&\text{in } \{t=0\}\times U_j.
% % % % % \end{cases}
% % % % % \end{equation}
%By linearity of the wave equation, summing over the partition of unity now yields for $u=\sum_{j=1}^k u_j$ 
\begin{equation}
\begin{cases}
%\tilde\square_{\tilde g} u = \sum_{j=1}^k \tilde\square_{\tilde g} u_j = \hat F, &\text{in } [0,\eps_1]\times \Omega,\\
u=f, &\text{on } [0,\eps_1]\times \p\Omega, \\
u=r_{0},\, \p_t u = r_{1},&\text{in } \{t=0\}\times \Omega,
\end{cases}
\end{equation}
which is  \eqref{eq:wave_eq_for_v} for $T=\eps_1$.

We continue iteratively and extend $u$ to a solution of~\eqref{wave-eq_with_data} in increasing time steps $t_i$. At each iteration step, which concerns the time-interval $I_i$, we use as the initial values $\tilde u\big|_{t=t_i}$ and $\p_t u\big|_{t=t_i}$. These are well defined since $t_i < t_{i-1}+\eps_{i-1}$. This way, we found a unique solution $u\in E^{s+1}$ to~\eqref{eq:wave_eq_for_v} in $[0,T]\times \Omega$, and consequently a unique solution to~\eqref{wave-eq_with_data} in the class $E^{s+1}$.

The regularity and unique existence results of solutions for~\eqref{wave-eq_with_data}, which we have now shown, can be turned into the energy estimate \eqref{energy_estimate} using the closed graph theorem.
Consider the Banach space $E^{s+1}$ and define a linear map
\[
A: E^s\times H^{s+1}(\Sigma)\times H^{s+1}(\Omega) \times H^s(\Omega) \to E^{s+1}
\]
by $A(F,f,u_0,u_1) = u$, where $u$ is the unique solution to \eqref{wave-eq_with_data}.
To have the energy estimate~\eqref{energy_estimate} it is sufficient to show that $A$ is continuous. By the closed graph theorem, this is in turn equivalent to showing that if
%Now it is enough to verify that if
\[
\begin{cases}
(F_k,f_k,u_{0,k},u_{1,k})\to (F,f,u_0,u_1) &\text{in } E^s\times H^{s+1}(\Sigma)\times H^{s+1}(\Omega) \times H^s(\Omega),\\
A(F_k,f_k,u_{0,k},u_{1,k}) \to u_\infty &\text{in } E^{s+1},
\end{cases}
\]
then 
\[
u_\infty= A(F,f,u_0,u_1).
\]
Here $F_k \to \square_g u_\infty$ in $\mathcal{D}'([0,T]\times \Omega)$, $f_k\to u_\infty\big|_{\Sigma}$ in $\mathcal{D}'(\Sigma)$ and similarly for $t=0$, $u_{0,k}\to u_\infty$ and $u_{1,k}\to \p_t u_\infty$ in $\mathcal{D}'(\Omega)$.
Due to the uniqueness of limits, we have that $u_\infty$ solves~\eqref{wave-eq_with_data}.
%~\footnote{For example, consider the distributional formulation $\langle \square_g u - F_k,\varphi\rangle = \langle \square_g u - \square_g u_k,\varphi\rangle = \langle u-u_k,\square_g^*\varphi\rangle\to 0$, where $\square_g^*$ is the formal $L^2$-adjoint of $\square_g$. So $F_k\to \square_gu$ in the sense of distributions and the uniqueness of limit implies $\square_gu=F$. Similarly for $f_k,u_{0,k},u_{1,k}$. By the uniqueness of wave equation, $u$ must solve \eqref{wave-eq_with_data} with the data $(F,f,u_0,u_1)$.}
%
Therefore, by uniqueness of solutions to the wave equation, we have that $u=A(F,h,u_0,u_1)$.
Hence $A$ is a bounded linear map and the energy estimate follows.%\footnote{Check do we need to add behaviour on boundary to $u$ in closed graph. Maybe not: as $u$ is in $E^1$ it has a good trace on both $\Omega$ and $\Sigma$, so again by uniqueness of limits everything should be fine.}
\end{proof}

\subsection*{Acknowledgements}

M. L. was supported by the Academy of Finland, grants 320113, 318990, and 312119. L. P-M. and T. L. were supported by the Academy of Finland (Centre of Excellence in Inverse Modeling and Imaging, grant numbers 284715 and 309963) and by the European Research Council under Horizon 2020 (ERC CoG 770924). T. T  was supported by the Academy of Finland (Centre of Excellence in Inverse Modeling and Imaging, grant number 312119).

% All references of our previous paper. Remove the extra ones

\noindent{\footnotesize E-mail addresses:\\
Matti Lassas: {matti.lassas@helsinki.fi}\\
Tony Liimatainen: {tony.liimatainen@helsinki.fi}\\
Leyter Potenciano-Machado: {leyter.m.potenciano@jyu.fi}\\
Teemu Tyni: {teemu.tyni@helsinki.fi}
}

\end{document}